\documentclass[11pt,pdflatex, makeidx]{amsart}
\usepackage[T1]{fontenc}
\usepackage{mathpazo}
\usepackage{bm,mathrsfs}
\usepackage{enumerate}
\usepackage[table]{xcolor}
\usepackage{setspace}
\usepackage{ytableau}
\usepackage{bbding}
\usepackage{float}\restylefloat{table} 

\usepackage{color}
\usepackage[color,matrix,arrow]{xy}

\usepackage{tikz}
\usetikzlibrary{arrows, intersections, calc}
\usepackage{comment}
\usepackage{makecell}
\usepackage{mathtools}
\usepackage{xfrac}

\usepackage{bbold}

\numberwithin{equation}{section}

\let\oldtocsection=\tocsection

\let\oldtocsubsection=\tocsubsection

\let\oldtocsubsubsection=\tocsubsubsection

\renewcommand{\tocsection}[2]{\hspace{0em}\bfseries\oldtocsection{#1}{#2}}
\renewcommand{\tocsubsection}[2]{\hspace{1em}\oldtocsubsection{#1}{#2}}
\renewcommand{\tocsubsubsection}[2]{\hspace{2em}\oldtocsubsubsection{#1}{#2}}

\usepackage{pdfpages} 
\usepackage[colorlinks=true, linkcolor=blue]{hyperref}
\usepackage{multicol} 

\usepackage{ulem}
\usepackage{amssymb}
\usepackage{graphicx, epsfig}
\usepackage{latexsym, amsfonts, amscd, amsmath}
\usepackage{mathrsfs}
\usepackage{lscape}
\usepackage{wrapfig}
\usepackage{tabularx}\usepackage{booktabs} 
\makeatletter
\def\hlinewd#1{%
\noalign{\ifnum0=`}\fi\hrule \@height #1 %
\futurelet\reserved@a\@xhline}
\makeatother

\makeindex 
\input xy
\xyoption{all}
\xyoption{rotate}\xyoption{pdf}
\usepackage{wallpaper}

\definecolor{orange}{rgb}{1,0.5,0}
\definecolor{Indigo}{rgb}{0.2,0.1,0.7}
\definecolor{Violet}{rgb}{0.5,0.1,0.7}

\definecolor{gold}{rgb}{1, 0.87498,0.4}

\newtheorem{thm}{Theorem}[subsection]
\newtheorem{prop}[thm]{Proposition}
\newtheorem{lem}[thm]{Lemma}
\newtheorem{cor}[thm]{Corollary}

\theoremstyle{definition}

\newtheorem{exa}[thm]{Example}

\theoremstyle{remark}
\newtheorem{rmk}[thm]{Remark}

\newenvironment{mat}{\left(\begin{smallmatrix}
\end{smallmatrix}\right)}{enddef}

\DeclareMathOperator{\Aut}{Aut}

\DeclareMathOperator{\diag}{{{diag}}}

\DeclareMathOperator{\End}{{{End}}}
\DeclareMathOperator{\Fr}{{{Fr }}}
\DeclareMathOperator{\Hom}{{{Hom}}}
\DeclareMathOperator{\Id}{{{Id}}}
\DeclareMathOperator{\Ima}{{{Im}}}

\DeclareMathOperator{\Jac}{{{Jac }}}
\DeclareMathOperator{\Ker}{{{Ker}}}

\DeclareMathOperator{\Nm}{{{Nm }}}

\DeclareMathOperator{\ord}{{{ord }}}

\DeclareMathOperator{\Spec}{{{Spec }}}

\DeclareMathOperator{\Tr}{{{Tr }}}

\DeclareMathOperator{\Ver}{{{Ver }}}

\DeclareMathOperator{\vol}{vol}

\DeclareMathOperator{\SL}{{{SL }}}
\DeclareMathOperator{\GL}{{{GL}}}

\DeclareMathOperator{\Gal}{{{Gal}}}

\newcommand{\gera}{{\frak{a}}}
\newcommand{\gerb}{{\frak{b}}}
\newcommand{\gerc}{{\frak{c}}}
\newcommand{\gerd}{{\frak{d}}}

\newcommand{\gerf}{{\frak{f}}}
\newcommand{\gerg}{{\frak{g}}}

\newcommand{\gerl}{{\frak{l}}}

\newcommand{\gerp}{{\frak{p}}}

\newcommand{\gerH}{{\frak{H}}}

\newcommand{\uA}{{\underline{A}}}
\newcommand{\uB}{{\underline{B}}}

\newcommand{\calM}{{\mathcal{M}}}

\newcommand{\calO}{{\mathcal{O}}}

\newcommand{\calT}{{\mathcal{T}}}

\newcommand{\calW}{{\mathcal{W}}}

\def\AA{\mathbb{A}}

\def\CC{\mathbb{C}}

\def\FF{\mathbb{F}}
\def\GG{\mathbb{G}}

\def\KK{\mathbb{K}}

\def\PP{\mathbb{P}}
\def\QQ{\mathbb{Q}}
\def\RR{\mathbb{R}}

\def\TT{\mathbb{T}}

\def\ZZ{\mathbb{Z}}

\newcommand{\scrB}{{\mathscr{B}}}

\newcommand{\scrE}{{\mathscr{E}}}

\newcommand{\scrG}{{\mathscr{G}}}

\newcommand{\scrP}{{\mathscr{P}}}
\newcommand{\scrQ}{{\mathscr{Q}}}

\newcommand{\scrS}{{\mathscr{S}}}

\newcommand{\id}{{\noindent}}

\newcommand{\arr}{{\; \rightarrow \;}}

\newcommand{\injects}{{\; \hookrightarrow \;}}

\newcommand{\ol}{{\mathcal{O}_L}}

\newcommand{\fpbar}{\overline{\mathbb{F}}_p}

\DeclareMathOperator{\disc}{disc}

\newcommand{\g}{\alpha}
\newcommand{\h}{\beta}
\newcommand{\x}{s}

\begin{document}
\marginparwidth 50pt

\title[Supersingular curves]{
Supersingular elliptic curves, quaternion algebras and applications to cryptography
}
\author{Eyal Z. Goren \& Jonathan Love}
\address{Department of Mathematics and Statistics,
McGill University, 805 Sherbrooke St. W., Montreal H3A 2K6, QC,
Canada.}
\email{eyal.goren@mcgill.ca; jon.love@mcgill.ca}

\begin{abstract} This paper contains a survey of supersingular isogeny graphs associated to supersingular elliptic curves and their various applications to cryptography. Within limitation of space, we attempt to address a broad audience and make this part widely accessible. For those graphs we also present three recent results and sketch their proofs. We then discuss a generalization to superspecial isogeny graphs associated to superspecial abelian varieties with real multiplication. These graphs were introduced by Charles, Goren and Lauter and so our discussion is brief. Motivated by their cryptographic applications, we prove a general theorem concerning generation of lattices over totally real fields by elements of specified norm. Throughout the paper we have attempted to clarify certain considerations that are either vaguely stated in the literature, or are folklore. We hope this paper will be useful both to a novice wishing to familiarize themselves with this very active area, and to the expert who may enjoy some vignettes and an overview of some new results. 
\end{abstract}

\maketitle
{\small{
\tableofcontents
}}

\section{Introduction}

\id In recent years the use of supersingular elliptic curves in cryptography has been an active field of research and at various times some cryptographic primitives based on supersingular elliptic curves, such as the CGL hash function, the SIKE key-exchange protocol or  the SQIsign digital signature protocol, have been strong contenders for post-quantum cryptography protocols. This created a fertile playing field in which results from arithmetic geometry, some more than a century old, others very recent, were brought to bear on questions from mathematical cryptography. More recently, higher dimensional superspecial graphs have been brought into this area. For one, they offer a generalization of the cryptographic primitives previously mentioned; in addition, recent devastatingly successful attacks on the SIKE protocol can be viewed as happening within such superspecial graphs. The introduction of superspecial graphs raises and motivates many interesting questions concerning the arithmetic of higher dimensional abelian varieties and the study of their moduli spaces and endomorphism rings. Although this is a well-developed area already, the questions raised in this context are new and of much mathematical interest. 

This paper offers a fairly thorough introduction to supersingular isogeny graphs and provides an overview of several applications of supersingular elliptic curves to cryptography. In this part of the paper, we have attempted to reach a somewhat general audience: we assume a background in algebraic number theory at the level of a first-year graduate course, as well as some basic exposure to algebraic geometry. As our goal is to reach people who are new to the subject, the presentation includes a considerable amount of background material; many proofs are omitted, but in these cases references are provided at the beginning of each relevant section.

We also present three recent theorems on these topics, motivated by cryptographic applications but also of independent interest. As the results have appeared in print already with complete proofs, we only provide a sketch of the proofs here, for completeness. These are Theorems~\ref{thm: Mayo}, \ref{thm: generation of orders} and \ref{thm: theta determines E} in the article, where the first is due to S. Mayo, and the last two to the authors of the present paper. 
Theorem~\ref{thm: Mayo} states, roughly and under some conditions, that the only non-trivial automorphism of a supersingular isogeny graph is an involution that comes from the Frobenius map. Theorem~\ref{thm: generation of orders} states that endomorphism rings of supersingular elliptic curves in characteristic~$p$ are generated by isogenies of order in the set $\{\ell^k: k = 0, 1, 2, 3\dots\}$, where $\ell$ is any fixed prime different from $p$. (In fact, Theorem~\ref{thm: generation of orders} follows from Theorem~\ref{thm: special generators for a lattice}, that treats generation of general lattices by elements of specified norms.) Finally, Theorem~\ref{thm: theta determines E} asserts that supersingular elliptic curves $E$ are determined up to Frobenius transform by their theta function $\Theta_E(q):=\sum_{n = 0}^\infty r_E(n) q^n$, where $r_E(n)$ is the number of endomorphisms of $E$ of norm $n$. 
The material in this part of the paper cannot be viewed as original anymore, but along the way we took care to clarify some points that are hard to find explicitly, or in clear form, in the literature. We do however include one new application of Theorem~\ref{thm: special generators for a lattice} to the theory of even unimodular lattices (Corollary~\ref{cor: even unimodular}).

\vspace{0.2cm}

\noindent Following that, we pick up the pace and assume more background from the reader, since keeping the same pace of exposition would have more than doubled the length of the paper. Following \cite{CGL2}, we introduce superspecial isogeny graphs associated to superspecial abelian varieties with real multiplication. For real quadratic fields, these so-called RM isogeny graphs seem to be interpolating between the graphs associated to elliptic curves and those associated to abelian surfaces; for instance, they retain many properties of the elliptic curve graphs, such as the Ramanujan property, that the more commonly studied graphs of abelian surfaces do not. They raise several attractive questions that we hope some of the readers will pick up; one of them being the actual implementation of cryptographic primitives based on superspecial abelian varieties with real multiplication; another is to compare the computational complexity of walking supersingular isogeny graphs of elliptic curves, of (say) superspecial abelian surfaces with real multiplication and of superspecial principally polarized abelian surfaces. Yet another is to generalize all of that for totally real fields of arbitrary class number.

In the last part of the paper, motivated by possible cryptographic applications, we generalize Theorem~\ref{thm: special generators for a lattice}, and so Theorem~\ref{thm: generation of orders}, to lattices over totally real fields. These appear in the text as Theorems \ref{thm: totally real local-global generators} and \ref{thm: End generation totally real}, respectively. Despite generalizing our results over $\QQ$, the argument is constructed differently. We make use of inhomogeneous theta functions and the theta correspondence, as well as results on the asymptotic growth of Fourier coefficients. We thus stand on the shoulders of giants, as we use deep results by Shimura, Siegel-Weil, and Deligne. Besides their application to endomorphism rings of superspecial abelian varieties with real multiplication, we expect such theorems to have other lovely applications, for example of the sort we provide for even unimodular lattices over $\ZZ$.

\medskip

\id A precise description of the contents of this paper can be gleaned from the detailed table of contents.

\medskip

\id {\bf Acknowledgments.} The first named author was supported by NSERC grant 223148. The second author was a CRM-ISM 
postdoctoral  fellow  at McGill University 
when this research was done. 
The first author wishes to thank Professors Tony Shaska and Shaul Zemel for the invitation to attend the inspiring  \textit{Advanced Research Workshop: Isogeny based post-quantum cryptography} they organized in the Hebrew University of Jerusalem, July 29-31, 2024, and both authors are grateful to them for the opportunity to submit a paper to the proceedings. 

\

\section{Quaternion algebras and supersingular elliptic curves}
\subsection{Quaternion algebras} \label{sec: quat alg basics}
In this discussion we provide only the mere definitions and results that would assist in reading the rest of the paper. For more background, a general discussion of quaternion algebras can be found in \cite{Vigneras} and \cite{Voight}.  

Let $L$ be a number field, or one of its completions; thus, $\RR$, $\CC$, or a finite extension of $\QQ_p$ for some prime $p$. A {\bf quaternion algebra} $B$ over $L$ is an $L$-algebra that, as a vector space of dimension $4$ over $L$, has a basis $1, i, j, k$:
\[ B = L 1 \oplus Li \oplus Lj \oplus Lk, \]
where there exists some non-zero $\alpha, \beta \in L$ such that 
\[ i^2 = \alpha,\;\; j^2 = \beta, \;\;ij = -ji = k.\]
One refers to such an algebra by the symbol $\left( \frac{\alpha, \beta}{L}
\right)$.

On any such quaternion algebra we have the {\bf canonical involution} which can be written in terms of the basis as 
\[b = x+ yi + zj + wk \quad \mapsto\quad  \bar b = x- yi - zj - wk\]
and satisfies $\overline{ab} = \bar b \cdot \bar a$ for $a,b\in B$.
We also define the reduced {\bf trace}, which is an $L$-linear map $\Tr\colon B \arr L$, and the reduced {\bf norm}, which is a multiplicative homomorphism $\Nm\colon B^\times \arr L^\times$. These are defined for $b=x+yi+zj+wk\in B$ by
\renewcommand\arraystretch{1.5}
\[ \begin{array}{rcccl}
	\Tr(b) &=& b+\bar b & =& 2x,\\
	\Nm(b) &=& b\cdot\bar b & =& x^2 + \alpha y^2 + \beta z^2 + \alpha\beta w^2.
\end{array} \]
The norm is a quadratic form on $B$, and we can associate to it a bilinear form $(a, b) = \frac12\Tr(a\bar b)$ satisfying the relations
\begin{align*}
	\Nm(a)&=(a,a), & (a,b)&=\frac12\left(\Nm(a+b)-\Nm(a)-\Nm(b)\right).
\end{align*}
Note that some sources use a different normalization for the bilinear form.

\subsubsection{Ramification and splitting} A quaternion algebra $B$ over $L$ is called {\bf split} if $B$ is isomorphic as an $L$-algebra to the $2\times 2$ matrix ring $M_2(L)$. Otherwise it is called {\bf ramified} and in this case it is a division algebra. If $L= \CC$, then every quaternion algebra is split; if $L$ is equal to $\RR$, then the only ramified quaternion algebra is the Hamilton quaternions $\left( \frac{-1, -1}{\RR}\right)$. If $L$ is a finite extension of~$\QQ_p$ and~$L_2$ its unramified quadratic extension with $\langle \sigma \rangle = \Gal(L_2/L)$, then the only ramified quaternion algebra over $L$ is isomorphic to

\[ \left\{ \left(\begin{matrix}a & \pi b \\
b^\sigma & a^\sigma
\end{matrix}\right): a, b \in L_2 \right\},\]
where $\pi$ is a uniformizer of $L$.

If $\varphi \colon L \arr M$ is a field homomorphism and $B$ is a quaternion algebra over $L$, then $B\otimes_{L, \varphi}M$ is a quaternion algebra over $M$. For example, $\left( \frac{\alpha, \beta}{L}
\right)\otimes_{L, \varphi} M \cong \left( \frac{\varphi(\alpha), \varphi(\beta)}{M} \right)$. In particular, if $L$ is a number field, $v$ is a place of $L$ with corresponding completion $L_v$, and $B$ a quaternion algebra over $L$, then $v$ is called ramified (respectively, split) relative to $B$ if $B\otimes_L L_v$ is ramified (resp., split) over $L_v$.

A fundamental theorem states that (1) the number of ramified places for $B$ is finite and even; (2) $B$ is determined up to isomorphism by its set of ramified places; (3) given any set $S$ consisting of an even number of places of $L$, there is a quaternion algebra $B$ over $L$ with $S$ as its set of ramified places. As such a $B$ is unique up to isomorphism, one may write $B = B_S$. 

Suppose that $M/L$ is an extension of number fields, $V$ is a place of $M$, and $v$ is the restriction of $V$ to $L$. If $B$ is a quaternion algebra over $L$ that is split at $v$ then $B\otimes_L M$ is split at $V$. If $B$ is ramified at $v$ then $B\otimes_L M$ is ramified at $V$ if and only if $[M_V: L_v]$ is odd.

\subsubsection{The quaternion algebra $B_{p, \infty}$} We shall use the notation $B_{p, \infty}$ to denote that quaternion algebra over $\QQ$ ramified precisely at $p$ and $\infty$, where $p$ is a finite prime. If $p=2$ we obtain the rational Hamilton quaternions $B_{2, \infty}=\left( \frac{-1, -1}{\QQ}\right)$. For $p \equiv 3\pmod{4}$, $B_{p, \infty}$ is the quaternion algebra $\left( \frac{-1, -p}{\QQ}\right)$ \cite[p. 79]{Vigneras}. This is the algebra \[\QQ\otimes\QQ i\oplus \QQ j \oplus \QQ k,\] with \[i^2 = -1, \;\;j^2 = -p, \;\;ij = -ji = k, \] and so $k^2 = -p$ and the norm form is $x^2 + y^2 + pz^2 + pw^2$. Similar models can be given for all primes $p$, and they are rather explicit (loc. cit. p. 98).

\subsection{Orders and ideals} We now assume that $L$ is either a number field or a finite extension of~$\QQ_p$, and we denote $\ol$ its ring of integers. As before, let $B$ be quaternion algebra over $L$. 

An {\bf order} $\calO$ in $B$ is a rank $4$ $\ol$-module that contains a basis for $B$ over $L$ and is also a subring of $B$; it therefore contains $\ol$. Since $\ol$ need not have class number $1$, it is possible that $\calO$ is not a free $\ol$-module, but it always contains a free rank $4$ $\ol$-submodule of finite index. Every element $b$ of an order is integral over $\ol$ and, in particular, $\Tr(b), \Nm(b)$ lie in $\ol$. As $\bar b = b - \Tr(b)$, any order is closed under the canonical involution.  An order is called {\bf maximal} if it is not strictly contained in any other order. 

Let $(x, y)$ denote the bilinear form $\frac12\Tr(x\bar y)$ on $B$, and hence $2(a,b)\in\ol$ for any $a,b\in \calO$. For any free $\ol$-submodule $\oplus_{i=1}^4 \ol v_i$ of $\calO$, we consider the value
\[ d(v_1, \dots, v_4) := \det (2(v_i, v_j))_{1\leq i, j \leq 4}\in\ol.\]
This determinant depends on the choice of the basis $v_i$ for $\oplus_{i=1}^4 \ol v_i$ only up to squares of units of $\mathcal{O}_L^\times$. We define $\mathfrak{d}_\calO$, the {\bf discriminant} of $\calO$, to be the $\ol$-ideal whose square is the ideal generated by $d(v_1, \dots, v_4)$ as $\{v_1, \dots, v_4\}$ runs over all $L$-bases of $B$ in $\calO$.\footnote{The usual definition of discriminant uses the form $\frac12\Tr(ab)$ and not $\frac12\Tr(a \bar b)$, \cite[p. 23 ff.]{Vigneras}, \cite[\S 15.2]{Voight}. However, using the form $\frac12\Tr(a\bar b)$ gives the same discriminant ideal. See \cite[Ex. 13, p. 254]{Voight}.}

If $\calO^\prime \subset \calO$ is a suborder then $\mathfrak{d}_{\calO^\prime} = [\calO: \calO^\prime]^2 \cdot \mathfrak{d}_{\calO}$. It follows that every order is contained in a maximal order. It is a theorem that the discriminant of any maximal order is the product of all the finite ramified places of the algebra $B$. 

It is possible for the discriminant to be $\calO_L$ and yet for $B$ not to be isomorphic to $M_2(L)$. For example, take $L$ to be a real quadratic field in which $p$ is inert. Then $B_{p, \infty}\otimes_\QQ L$ is ramified precisely at the two infinite primes $\infty_1, \infty_2$ of $L$ and at no finite prime. 

\subsubsection{$B_{p, \infty}$ again.} \label{subsubsection: Bp again} For $p\equiv 3 \pmod{4}$, an example of a maximal order of $B_{p, \infty} = \left( \frac{-1, -p}{\QQ}\right)$ is given by 
\[ \ZZ \oplus \ZZ \cdot i \oplus \ZZ\cdot \frac{i+j}{2} \oplus \ZZ \cdot\frac{1+k}{2}\]
(\cite[p. 98]{Vigneras}). The matrix $(2(v_i, v_j))_{1\leq i,j\leq 4}$ associated to this basis is
\[ \left(\begin{smallmatrix} 2 & 0 & 0 & 1 \\
0 & 2 & 1 & 0\\  
0 & 1 & (p+1)/2 & 0\\
1 & 0 & 0& (p+1)/2 \\
\end{smallmatrix}\right);\]
its determinant is $p^2$, and so the discriminant of the order is $(p)$ as expected.

If $\calO$ is any maximal order of a quaternion algebra over a number field $L$, then, for any finite place $v$ of $L$, the base change to $\calO_{L_v}$ of the quadratic space $(\calO, \Nm)$ is isomorphic to $(M_2(\calO_{L_v}), \det)$ if $v$ is unramified. If $v$ is ramified then the corresponding base change is isomorphic to 
\begin{align}\label{eq:ramified_maxorder}
	\left( \left\{ \left(\begin{matrix}a & \pi_vb \\
		b^\sigma & a^\sigma
	\end{matrix}\right): a, b \in \calO_{L_{v,2}} \right\},\; \det\right),
\end{align}
where $\pi_v$ is a uniformizer of $L_v$ and $L_{v, 2}$ is the quadratic unramified extension of $L_v$.
In other words, this base change is isomorphic to
$((\calO_{L_{v,2}} \oplus \calO_{L_{v,2}}),\; aa^\sigma - \pi_v bb^\sigma)$. Choosing a basis for~$\calO_{L_{v, 2}}$ over~$\calO_{L_v}$, one may write this as a quadratic rank $4$ $\calO_{L_v}$-module. 

\subsubsection{Ideals} Let $\calO$ be an order in a quaternion algebra $B$ over $L$, a number field or a finite extension of $\QQ_p$. A {\bf right} (respectively, left) {\bf ideal} $I$ of $\calO$ is a rank $4$ $\calO_L$-module that contains a basis of $B$ over $L$, such that 
\[ I\cdot \calO \subseteq I \quad ({\rm respectively,}\; \calO \cdot I \subseteq I).\]
For a rank $4$ $\calO_L$-module $I$ that contains a basis of $B$ over $L$, define its {\bf left} (respectively, right) {\bf order} by
\[ \calO_\ell (I) = \{ b\in B: bI \subseteq I \}, \;\; ({\rm respectively,} \; \calO_r (I) = \{ b\in B: Ib \subseteq I \}).\]
These are indeed orders of $B$.

If $I$ is a left $\calO$-ideal then $\calO_\ell(I) \supseteq \calO$.  If $\calO$ is a maximal order and $I$ is a left ideal of $\calO$, then $\calO_\ell(I) = \calO$, while $\calO_r(I)$ is a maximal order that in general is not even conjugate to $\calO$ in~$B$. Analogous statements hold for right ideals. An ideal $J$ of $\calO$ is called a {\bf two-sided ideal} if $\calO = \calO_\ell(J) = \calO_r(J)$. 

\subsubsection{Ideal classes}\label{subsubsec: ideal classes} Let $B$ be a quaternion algebra over $L$ (still a number field or a finite extension of $\QQ_p$) and $\calO$ a maximal order of $B$. If $I$ is a right $\calO$-ideal, then so is $\lambda I$ for any $\lambda \in B^\times$. The relation $I \sim \lambda I$ is an equivalence relation on the set of right $\calO$-ideals. The equivalence classes are called the {\bf ideal classes} and we denote this by ${\rm Cl}(\calO)$; it is a set with a distinguished element, namely the class of $\calO$, but in general is not a group in any natural way. A fundamental theorem states that the number of ideal classes is finite, and we refer to this number as the {\bf class number} of $\calO$. It is a theorem that the class number is independent of the choice of maximal order $\calO$. We shall denote it $h_B$, or $h$ if the quaternion algebra $B$ is clear from context. More precisely: 
\begin{enumerate}
\item If $L$ is a finite extension of $\QQ_p$ and $B = M_2(L)$ then a maximal order is given by $M_2(\calO_L)$. All maximal orders in $B$ are conjugate in $B$, and each has class number $1$. If $B$ is the ramified quaternion algebra over $L$ (recall that such $B$ is unique up to isomorphism), then there is a unique maximal order, given by (\ref{eq:ramified_maxorder}), and again the class number is $1$. 
\item If $L$ is a number field, in general the class number is not $1$, but sometimes it is. If $L = \QQ$ and $B = B_{p, \infty}$ then the class number $h_{B_{p, \infty}}$ is equal to 
\begin{enumerate}
\item $\left\lfloor\frac{p}{12}\right\rfloor$ if $p \equiv 1 \pmod{12}$; 
\item $\left\lfloor\frac{p}{12}\right\rfloor+1$ if $p \equiv 5, 7 \pmod{12}$; 
\item $\left\lfloor\frac{p}{12}\right\rfloor+2$ if $p \equiv 11 \pmod{12}$; 
\item  $1$ if $p=2, 3$.
\end{enumerate}
(This follows from \cite[Theorem 30.1.5]{Voight}, as explicated by \cite[Chapter 30 Exercise 6]{Voight}.)
\end{enumerate}

\medskip
\id Given a maximal order $\calO$, let $I_1, \dots, I_h$ be representatives for the right ideal classes of $\calO$. Let $\calO_i: = \calO_\ell(I_i)$. Then every maximal order of $B$ is conjugate to one of the orders $\calO_i$. In general, the number of maximal orders in $B$, up to conjugation in $B$ (equivalently, up to isomorphism of rings) is called the {\bf type number} $t_B$ of $B$. Therefore, given these results, we see that the type number is smaller or equal to the class number. One can be more precise: 

Consider the collection of two-sided ideals of $\calO$. These ideals, considered up to equivalence $I \sim \lambda I$, where now $\lambda \in L^\times$, form a group $\text{\rm biCl}(\calO)$ that for $L$ a number field of class number~$1$ is an abelian group of order $2^f$, where $f$ is the number of finite ramified places of $B$. In fact, for every choice of $\epsilon_v\in\{0,1\}$ for the finite ramified places $v$ of $B$, there is a two sided ideal $I$ of $\calO$ characterized by the fact that for each finite ramified place of $B$ we have
\[I_v=\left\{ \left(\begin{matrix}a & \pi_vb \\ b^\sigma & a^\sigma 
\end{matrix}\right): a,b\in\calO_{L_v},\;\pi_v^{\epsilon_v} \vert a\right\}\]
and we have $I_v=\calO_v$ for all other places. Each equivalence class of two-sided ideals contains exactly one of these ideals, and this provides an isomorphism $\text{\rm biCl}(\calO) \cong (\ZZ/2\ZZ)^f$. This holds for every choice of maximal order $\calO$. For any fixed $\calO$, there is a map of pointed sets
\[ \text{\rm biCl}(\calO) \arr {\rm Cl}(\calO), \]
that may have a kernel, and this kernel may very well depend on the choice of $\calO$. For example, for $\calO$ a maximal order in $B_{p, \infty}$ we have $\text{\rm biCl}(\calO)\cong \ZZ/2\ZZ$, and its image in ${\rm Cl}(\calO)$ is trivial if and only if $\calO$ has an element of norm $p$.

Returning to our choice of representatives $I_1, \dots, I_h$ for ${\rm Cl}(\calO)$, 
suppose that $O_i$ is conjugate to $O_j$ for some $i \neq j$, say $\calO_i = \alpha  \calO_j \alpha^{-1}$. By changing the ideal $I_j$ to $\alpha I_j$ (which does not change the right ideal class), we may assume without loss of generality that $\calO_i = \calO_j$. One can prove that this is the case if and only if $I_i = J  I_j$ for some two-sided ideal $J$ of $\calO_j$ (see \cite[Proposition 2.18]{Pizer1}, \cite[p. 87 ff.]{Vigneras}, \cite[\S\S 18.4-18.5]{Voight}). 
Note that necessarily $J$ is not principal in ${\rm Cl}(\calO)$. Thus, such a coincidence $\calO_i = \calO_j$ can happen only if the image of $\text{\rm biCl}(\calO_j)$ in ${\rm Cl}(\calO_j)$ is not trivial, and conversely, elements of $\text{\rm biCl}(\calO_j)$ that are not principal ideals create such coincidences. 
For $B=B_{p, \infty}$ one concludes the formula 
\[ h_B = 2 t_B - u_B, \]
where $u_B$ is the number of maximal orders of $B$, considered up to conjugation, containing an element of norm $p$.
These maximal orders counted by $u_B$ are discussed further in \S \ref{subsubsec: brandt matrices}.

\subsection{Supersingular elliptic curves}
General references for this section are \cite{Silverman1}, \cite{Washington}. We also assume some familiarity with algebraic geometry, mostly at the level of \cite{Silverman1}, but occasionally a bit more. 

We fix an algebraic closure $\fpbar$ of $\mathbb{F}_p$, and denote by $\FF_{p^r}$ the unique finite field with $p^r$ elements in $\fpbar$.

\subsubsection{Weierstrass equations} Let $k$ be a field. An {\bf elliptic curve} over $k$ is a proper non-singular geometrically-irreducible curve $E$ of genus $1$ with a specified $k$-rational point denoted $0_E$. Using the Riemann-Roch theorem, one can prove that if ${\rm char}(k) \neq 2$ then such a curve can be given by an equation in $\PP^2_k$ of the form
\[ y^2 = f(x), \]
for some separable polynomial $f(x) \in k[x]$ of degree $3$. In fact, if ${\rm char}(k) \neq 2, 3$ one may take 
\begin{equation}  \label{eqn: short W}f(x): = x^3 + ax + b \in k[x].
\end{equation}This is called a {\bf short Weierstrass equation}. To be precise, it is the closure of this curve in the projective plane $\PP^2_k$ that is the elliptic curve, and we take $0_E$ to be the unique point at infinity given by $(x:y:z) = (0:1:0)$. A slightly more complicated equation, likewise obtained by using the Riemann-Roch theorem, may be needed over fields of characteristic $2$ or $3$. 

An elliptic curve $E$ has a {\bf $j$-invariant} $j(E)$; if $E$ is given by the equation $y^2 =  x^3 + ax + b$, then 
\begin{equation}
j(E) = 1728\cdot  \frac{4a^3}{27b^2 + 4a^3}.
\end{equation}
If there is an isomorphism over $k$ between $E_1$ and $E_2$ then $j(E_1) = j(E_2)$, and when $k$ is algebraically closed the converse also holds. Furthermore, for any field $k$ and element $j\in k$ there is an elliptic curve $E$ over $k$ with $j$-invariant $j(E) = j$. In fact, if $j \neq 0, 1728 $, one may take 
\[E\colon y^2+xy = x^3 - \frac{36}{j-1728}  -\frac{1}{j-1728}\]
\cite[\S III, Proposition 1.4]{Silverman1}. See \S \ref{sec:0 and 1728} for the cases $j =0, 1728$.

An elliptic curve is an abelian algebraic group; that is, there are morphisms 
\[ m\colon E \times E \arr E, \quad inv\colon E \arr E, \]
that make $E(k')$ into a group for any extension $k'/k$, where $m(P, Q)$ is the sum of $P,Q\in E(k')$ (we shall also write this as $P+Q$) and $inv(P)$ is the negative of $P$ (written as $-P$). The zero point of this group is $0_E$, and three points $P, Q, R\in E(k')$ are the intersection points of $E$ with a line (counted with multiplicity if necessary) if and only if $P+Q+R = 0_E$.

\subsubsection{Isogenies, degrees, kernels} 
It is a theorem that every morphism $f\colon E_1 \arr E_2$ of elliptic curves satisfying $f(0_{E_1}) = 0_{E_2}$ is also a group homomorphism. If $f$ is not the zero map then it is called an {\bf isogeny}. An isogeny $f$ -- being a non-constant morphism of irreducible curves -- must be surjective as a map of varieties (though the induced map $E_1(k')\to E_2(k')$ for $k'/k$ need not be surjective unless $k'$ is algebraically closed). Given an isogeny $f$, we may define its {\bf kernel} $\Ker(f)$ to be the fiber product 
\[ \xymatrix{ \Ker(f) \ar[r]\ar[d] & \{ 0_{E_2} \}\ar[d] \\ E_1 \ar[r]^f & E_2. }\]
As $f$ is a non-constant morphism of curves, this fiber product is $\Spec$ of a $k$-algebra which is finite-dimensional as a $k$-vector space; this  dimension is called the {\bf rank} of the finite group scheme $\Ker(f)$. The {\bf degree} $\deg(f)$ of the isogeny $f$ is defined to be the rank of $\Ker(f)$. This is also equal to the degree of the extension of function fields $[k(E_1) : f^\ast k(E_2)]$.
If $\bar k$ is an algebraically closed field then the $\bar k$-points $\Ker(f)(\bar k)$ of $\Ker(f)$ form an abelian group of cardinality dividing $\deg(f)$, but equality does not always hold. Equality does hold in the cases where ${\rm char}(k) = 0$, or where ${\rm char}(k) = p>0$ and $\gcd(p, \deg(f)) = 1$.  

As $E$ is an abelian group, for any integer $n$ the {\bf multiplication by $n$} map, denoted $[n]$, is an {\bf endomomorphism} (a morphism from $E$ to itself fixing $0_E$) and, in fact, an isogeny if $n\neq 0$. The kernel is denoted in this case $E[n]$ and it is always a group scheme of degree $n^2$. If $n = -1$ this map is given by $(x, y) \mapsto (x, -y)$ in terms of the short Weierstrass equation~(\ref{eqn: short W}). Since a point~$P$ is of order $2$ if and only if $P = -P$, we see that 
\[ E[2] = \{0_E, (\alpha_1, 0), (\alpha_2, 0), (\alpha_3, 0)\}, \]
where $\alpha_1, \dots, \alpha_3$ are the roots of $f(x)$. 

For an extension $k'/k$, the collection of homomorphisms $\Hom_{k'}(E_1, E_2)$ (that is, the collection of isogenies $E_1\to E_2$ defined over $k'$ together with the zero map) is an abelian group, and the collection of endomorphisms $\End_{k'}(E_1)$ is a ring with multiplication given by composition. If $\Hom_{k'}(E_1, E_2)$ is nontrivial then it is a projective rank $1$ module over $\End_{k'}(E)$. The degree map is a quadratic form $\Hom_k(E_1, E_2) \arr \ZZ$. When we write $\Hom(E_1, E_2)$ or $\End(E_1)$ without subscripts, we mean $\Hom_{\bar k}(E_1, E_2)$ and $\End_{\bar k}(E_1)$, where $\bar k$ is an algebraic closure of $k$.

If $H$ is a finite subgroup scheme of $E$ then the quotient scheme $E/H$ exists and is an elliptic curve whose $\bar k$ points are $E(\bar k) /H(\bar k)$. The morphism $E \arr E/H$ is an isogeny with kernel $H$. If $f\colon E_1 \arr E_2$ is an isogeny with kernel $H$, then there is a natural isomorphism $E_1/H \arr E_2$ such that the following diagram commutes:
\begin{equation}
 \xymatrix{E_1 \ar[rr]^f\ar[rd] && E_2 \\ &E_1/H. \ar[ru]^\cong &} 
 \end{equation}
Similarly, analogues of all the main isomorphism theorems from group theory hold for elliptic curves.

\medskip

\id To every isogeny $f\colon E_1 \arr E_2$ there is an associated {\bf dual isogeny} $f^\vee\colon E_2 \arr E_1$ uniquely determined by the property 
\[ f^\vee \circ f = [\deg(f)]\]
(meaning, the composition is the multiplication by the integer $\deg(f)$ on $E_1$). It is also true that $\deg(f) = \deg(f^\vee)$ and $f^{\vee\vee} = f$. Moreover, mapping an isogeny to its dual gives an isomorphism of quadratic modules $\Hom(E_1, E_2) \arr \Hom(E_2, E_1)$. On $\End(E)$ this map is an anti-isomorphism of rings, meaning $(f\circ g)^\vee = g^\vee \circ f^\vee$. In particular, these statements imply the fact that $[n]^\vee = [n]$.

\subsubsection{Complex multiplication}\label{subsubsec:CM} If $k$ is a characteristic zero field and $E\colon y^2 = x^3 + ax + b$ is an elliptic curve over $k$, one can embed $\QQ(a, b) \injects \CC$ and so one may view $E$ as a complex elliptic curve. Let $\calT_E$ denote that tangent space of $E$ at $0_E$. By the theory of complex Lie groups, the exponential map $\calT_E \arr E(\CC)$ is a surjective analytic homomorphism whose kernel is a full lattice~$\Lambda$ in~$\calT_E$; that is, a rank $2$ discrete subgroup of $\calT_E\cong \CC$. By a homothety of $\CC$, one may assume that $\Lambda = \ZZ \oplus \ZZ \tau$ where $\tau \in \CC$ is such that $\Ima(\tau) > 0$. Since we have an isomorphism 
\[\CC/\Lambda \cong E(\CC),\]
we see that $E[n] \cong \frac{1}{n} \Lambda /\Lambda$, a group of cardinality $n^2$. Drawing further from the theory of complex Lie groups, we find that for elliptic curves $E_i \cong \CC/\Lambda_i$,
\[ \Hom(E_1, E_2) = \{ \alpha\in \CC: \alpha \Lambda_1\subset \Lambda_2\}.\]
From this, it is not hard to conclude that if $E\cong \CC/\ZZ\oplus \ZZ\tau$ is an elliptic curve over a field $k$ of characteristic $0$, then 
\[ \End(E) \cong \begin{cases} \text{\rm an order}\;  R \; {\rm of} \;\QQ(\tau)& \tau \; \text{\rm quadratic over } \; \QQ\\
	\ZZ & {\rm otherwise}.
\end{cases}\]
In case $\End(E) = R$ is a quadratic order, one says that $E$ has {\bf complex multiplication}.

\medskip

\subsubsection{The $j$-invariants $0$ and $1728$}\label{sec:0 and 1728} Over $\CC$, the elliptic curve 
\[ E_{1728}\colon y^2 = x^3 + x, \]
has an automorphism $\varphi\colon (x, y) \mapsto (-x, iy)$. Note that $\varphi^2 = [-1]$ and so, given that $\End(E_{1728})$ is at most a quadratic order, we must have $\End(E) \cong \ZZ[i]$ (where $\varphi$ corresponds to $i$). The $j$-invariant of $E_{1728}$ is $1728$. Note that the same equation defines an elliptic curve over $\FF_p$ for any $p \neq 2$ (but the endomorphism ring may become larger). 

Over $\CC$, the elliptic curve 
\[ E_{0}\colon y^2 = x^3 + 1, \]
has an automorphism $\varphi\colon (x, y) \mapsto (\omega x, y)$, where $\omega = e^{2\pi i /3}$ is a primitive third root of unity. Note that $\varphi^3 = [1]$ and so, given that $\End(E_{0})$ is at most a quadratic order, we must have $\End(E) \cong \ZZ[\omega]$ (where $\varphi$ corresponds to $\omega$ and $\omega^2+\omega + 1 =0$). The $j$-invariant of $E_{0}$ is $0$. This equation is good also over $\FF_p$ for any $p > 3$ (though again the endomorphism ring may be larger).

In a field of characteristic $2$ or $3$ we have $0=1728$. The equation $y^2+y=x^3$ defines an elliptic curve of $j$-invariant $0$ over $\FF_2$, and $y^2=x^3+x$ defines an elliptic curve of $j$-invariant $0$ over $\FF_3$.

\subsubsection{Endomorphisms of elliptic curves over fields of positive characteristic} Let $k$ be a field of positive characteristic $p$ and let $E$ be an elliptic curve over $k$. In this case, there are exactly 3 possibilities for $\End_k(E)$.
\[\End_k(E) \cong
\begin{cases} \ZZ, & \\
	R, & \text{\rm a quadratic imaginary order,}\\
	\calO, & \text{\rm a maximal order of } B_{p, \infty}.
\end{cases}
\]
In fact, the quadratic imaginary orders $R$ that arise have conductor prime to $p$ and $p$ splits in~$R$, and, conversely, every such~$R$ arises as $\End_k(E)$ for some elliptic curve $E$ over some finite field~$k$. Likewise, every maximal order of $B_{p, \infty}$ is isomorphic to some $\End_k(E)$.
For elliptic curves $E$ over a field $k$ of positive characteristic, it is a fact that  $\End_k(E)\neq \ZZ$ if and only if $E$ can be defined over a finite field (i.e.~there exists a finite subfield $\FF_{p^r}$ in $k$ and an elliptic curve $E'$ over $\FF_{p^r}$ such that $E$ is the base change of $E'$ to $k$).
If $\End_{\bar k}(E)$ is a maximal order of $B_{p,\infty}$ then $E$ is called {\bf supersingular}, and otherwise $E$ is called {\bf ordinary}. 
 
To fix ideas, assume that $E$ is given by a Weierstrass equation $y^2 = x^3 + ax + b$. Then $E^{(p)}$, the {\bf Frobenius base-change} of $E$, is the elliptic curve given by 
 \begin{equation}
 E^{(p)}\colon y^2 = x^3 + a^px + b^p.
  \end{equation}
 In fact, scheme-theoretically, $E^{(p)} \cong E \otimes_{\Spec(k), \varphi} \Spec(k)$, where $\varphi(z) = z^p$ is the Frobenius homomorphism of $k$. This observation shows that the construction can be carried in complete generality for any $k$-scheme. We may likewise define $E^{(p^n)}$ for any $n\geq 1$ by applying Frobenius base-change $n$ times.
 
 The canonical morphism
 \begin{equation}
 \Fr\colon E \arr E^{(p)}, \quad \Fr(x, y) = (x^p, y^p), 
  \end{equation}
 takes $0_E$ to $0_{E^{(p)}}$ and is therefore an isogeny. Its degree is $p$. Its dual isogeny $\Fr^\vee$ is called the {\bf Verschiebung} isogeny, 
 \[ \Ver\colon E^{(p)} \arr E, \]
 and it satisfies $\Ver\circ \Fr = [p]$ on $E$ and $\Fr\circ \Ver = [p]$ on $E^{(p)}$. When we need to be clear about the starting curve we write $\Fr_E\colon E \arr E^{(p)}$, and we define 
 \[\Fr^n_E:=\Fr_{E^{(p^{n-1})}}\circ \cdots \circ \Fr_{E^{(p)}}\circ \Fr_E,\]
 an isogeny of degree $p^n$ from $E$ to $E^{(p^n)}$.
 
 A supersingular elliptic curve $E_1$ over $k$ has a unique subgroup scheme of rank
 $p$ (while an ordinary elliptic curve has two). This implies that any isogeny 
$f\colon E_1 \arr E_2$ of degree $p$ fits into a commutative diagram
\begin{align}\label{eq:factor through Fr}
	\xymatrix{E_1 \ar[rr]^f\ar[rd]^{\Fr}&& E_2 \\ & E_1^{(p)} \ar[ru]^\gamma& }
\end{align}
where $\gamma$ is an isomorphism. 

\subsubsection{Equivalent definitions of supersingularity} There are many ways to characterize supersingularity. We assume that $E$ is defined over $\FF_{p^r}\subset \overline{\FF}_p$. The following are equivalent:
\begin{enumerate}
\item $E$ is supersingular ($\End_{ \overline{\FF}_p}(E)$ is isomorphic to a maximal order $\calO$ of $B_{p, \infty}$).
\item $E[p]( \overline{\FF}_p) = \{ 0_E\}$ (there are no non-trivial $p$-torsion points).
\item $E[p^n]( \overline{\FF}_p) = \{ 0_E\}$ for some (or all) $n\geq 1$.
\item The group scheme $\alpha_p:= \Ker(\Fr: \GG_a \arr \GG_a)$ embeds into $E$ over $ \overline{\FF}_p$.\footnote{The group scheme $\alpha_p$ is the group scheme whose $R$ points, for every commutative $\FF_p$-algebra $R$, are $\{x\in R: x^p = 0\}$, with the usual addition law.}
\item $E$ can be defined over $\FF_{p^2}$ and, identifying $E$ with $E^{(p^2)}$, we have $\Fr^2 = \epsilon \circ [p]$ for some automorphism $\epsilon$ of $E$. 
\item The morphism $\Ver\colon E^{(p)} \arr E$ is inseparable, meaning the extension of fields $\Ver^\ast k(E)\subset k(E^{(p)})$ is inseparable (of degree $p$). 
\item Assuming $p>2$, let $m = (p-1)/2$. Writing the elliptic curve $E$ in {\bf Legendre form} $E\colon y^2 = x(x-1)(x-\lambda)$ (which is always possible if one allows a finite extension of the base field), $\lambda$ is a root of the polynomial 
\[ H_p(t) = \sum_{i=0}^m \binom{m}{i}^2 t^i.\] 
(If $p=2$ then $E$ is supersingular if and only if $j(E)=0$.)
\item For every curve $E'$ that is isogenous to $E$ over $\overline{\FF}_p$, $j(E')\in \FF_{p^2}$.\footnote{One can make this into an effective criterion for determining whether a given elliptic curve $E$ is supersingular or not; see \cite{SutherlandIdentify}.}
\item For $p>2$, if $E$ is defined over $\FF_q$ and $\sharp(E(\FF_q)) = q+1 - A$, then $A \equiv 0 \pmod{p}$.\footnote{This condition yields an effective polynomial time test for supersingularity using Schoof's algorithm \cite{Schoof}.}
\end{enumerate}
The equivalence of most of these conditions is proven in \cite[Theorem V.3.1, Theorem V.4.1]{Silverman1}. For condition (4), which will be transparent to the reader versed in the theory of finite commutative group schemes, see for instance \cite[Proposition 2.1]{Ulmer}.
 
 \medskip
 
 \id Note that the degree of $H_p$ in definition (7) is $(p-1)/2$. On the other hand, we have \cite[\S III, Proposition 1.7]{Silverman1}
 \[ j(E) = 2^8 \frac{(\lambda^2 - \lambda +1)^3}{\lambda^2(\lambda-1)^2}\]
 and so the association $\lambda \mapsto j$ is ``usually'' $6:1$. This suggests that the number of supersingular $j$-invariants, namely, the number of supersingular elliptic curves up to isomorphism over $\overline{\FF}_p$, is roughly $(p-1)/12$. This counting can be made precise (\cite[\S V, Theorem 4.1]{Silverman1}) and one obtains that the number of supersingular elliptic curves is precisely the class number $h_{B_{p, \infty}}$ of~$B_{p, \infty}$ given in \S \ref{subsubsec: ideal classes}. In fact there is an explicit correspondence between isomorphism classes of supersingular curves and ideal classes of a maximal order in $B_{p,\infty}$, which is described in \S \ref{subsubsec: deuring corr}.
 \medskip 
 
 If $E_1$ is a supersingular elliptic curve and $f\colon E_1 \arr E_2$ is an isogeny, then $E_2$ is also supersingular. This follows immediately from (8), but it may also be instructive to see how this follows from (2).  Suppose there exists a nonzero point $Q\in E_2[p]$. For any $n\geq 1$, the map $[p^{n-1}]$ is an isogeny, and so $Q$ can be written as $p^{n-1}P$ for some point $P\in E_2(\overline{\FF}_p)$ of exact order $p^n$. 
 Take $n$ such that $p^n > \deg(f)$. Note that $f^\vee(P)$ is a point on $E_1$ that is annihilated by $[p^n]$, and therefore must be $0_{E_1}$ since $E_1$ is supersingular. But this implies $f(f^\vee(P))=[\deg f](P)=0_{E_2}$, a contradiction since $P$ has order $p^n>\deg(f)$.

\subsubsection{Models of supersingular elliptic curves over $\FF_{p^2}$}\label{subsec: model over Fp2} The $j$-invariant of a supersingular elliptic curve $E$ lies in $\FF_{p^2}$ and so $E$ has a model over $\FF_{p^2}$. In such a model we have $\Fr^2 = \epsilon\circ [p]$ for some automorphism $\epsilon$ of $E$. It turns out \cite{BGJGP} that one can always find such a model where $\Fr^2 = [-p]$.

For any such models $E_1, E_2$ we have
\[ \Hom_{\FF_{p^2}}(E_1, E_2) = \Hom_{\overline{\FF}_p}(E_1, E_2),\]
that is, all isogenies are defined over $\FF_{p^2}$.
Indeed, if $f \in \Hom_{\overline{\FF}_p}(E_1, E_2)$ then so is $f^{(p^2)}$, the base-change of $f$ by the square of Frobenius. This base change is characterized by the property that $f^{(p^2)}(\Fr^2 P) = \Fr^2(f(P))$ for any point $P$ on $E_1$. But, \[f^{(p^2)}(\Fr^2 P) = f^{(p^2)}(-pP) = -p f^{(p^2)}(P)\]
and $\Fr^2(f(P)) = -p f(P)$. As $[-p]$ is an isogeny, and $-p (f^{(p^2)} - f) = 0$, it follows that $f = f^{(p^2)}$ and thus $f \in \Hom_{\FF_{p^2}}(E_1, E_2)$.\footnote{Taking an appropriate twist, each supersingular curve also has a model over $\FF_{p^2}$ satisfying $\Fr^2=[p]$, and for such models we also have $\Hom_{\FF_{p^2}}(E_1, E_2) = \Hom_{\overline{\FF}_p}(E_1, E_2)$. This gives us an alternate setup, but we will not use this.}

\subsection{Reducing curves with complex multiplication modulo primes}\label{sec:CM}
 \cite{Lang, Serre, Silverman2} may be consulted for a complete overview. A powerful method of finding supersingular elliptic curves is the reduction of elliptic curves with complex multiplication (as defined in \S \ref{subsubsec:CM}). Let $E$ be an elliptic curve over $\CC$ such that $\End(E) = \calO_K$, the ring of integers of a quadratic imaginary field $K$. The complex uniformization $\CC/\Lambda\cong E(\CC)$ (see \S \ref{subsubsec:CM}) shows that $\Lambda$ must be an $\calO_K$-module. Being of rank $2$ and torsion free, $\Lambda$ must be a projective $\calO_K$-module of rank $1$ and thus its isomorphism class is an ideal class of $\calO_K$. Conversely, given an ideal $I$ of $\calO_K$ and an inclusion $K \arr \CC$, we may view $I$ as an $\calO_K$-invariant lattice in $\CC$ and get an elliptic curve with complex multiplication as $\CC/I$. Fixing an inclusion $K \arr \CC$, one obtains a bijection between the class group ${\rm Cl}(\calO_K)$ and isomorphism classes of elliptic curves with complex multiplication such that action of $K$ on their tangent space is given by the fixed inclusion $K \arr \CC$. 

Choose representatives $I_1, \dots, I_h$ for the ideal classes of $K$. Let $E_i$ be the elliptic curves such that $E_i(\CC) \cong \CC/I_i$. 

One of the main theorems of complex multiplication is that $j(E_i)$ belongs to the Hilbert class field $H_K$ of $K$. The Hilbert class field is the maximal unramified extension of $K$ with abelian Galois group $\Gal(H_K/K)$, and this Galois group is isomorphic to ${\rm Cl}(\calO_K)$. This theorem states further that $\Gal(H_K/H)$ acts simply transitively on $j(E_1), \dots, j(E_h)$ and it follows that $H_K = K(j(E_i))$ for any $i$. In particular, if $E$ has CM by $\calO_K$ then $E$ has a model over $H_K$. 

For a number field $L$,
every elliptic curve $E/L$ with complex multiplication has {\bf potential good reduction} everywhere. One interpretation of this statement is that there is a finite extension $M$ of $L$ such that $E$ has a general Weierstrass equation $y^2 + a_1xy + a_3y = x^3 + a_2x^2 + a_4x + a_6$, $a_i \in \calO_{k_1}$, such that modulo every finite place $v$ of $M$, the reduction of this equation modulo $v$ still defines a non-singular (elliptic) curve $E\pmod{v}$. Let $p$ be the rational prime determined by $v$. Then $E \pmod{v}$ is an elliptic curve over a finite field $\FF_{p^r}$ (with $r$ depending on $v$) and we can ask whether it is ordinary or supersingular. It is a theorem that $E\pmod{v}$ is supersingular if $p$ is inert or ramified in $K$, and it is ordinary if $p$ is split in $K$. 

As an example, consider $E_0$ that has CM by $\ZZ[\omega]$ (see \S \ref{sec:0 and 1728}). The only prime ramifying in~$\ZZ[\omega]$ is $3$. The prime $2$ is inert and a prime $p > 3$ is split if and only if there are non-trivial third roots of unity in $\FF_p$, which is the case if and only if $p \equiv 1 \pmod{3}$. Thus, the primes  congruent to $1$ mod $3$ split, and those congruent to $2$ mod $3$ are inert. In particular, $0$ is a supersingular $j$-invariant for $p = 2, 3$ and all $p>3$ congruent to $2$ mod $3$. 

Similarly, $E_{1728}$ has CM by $\ZZ[i]$. The prime $2$ ramifies, the primes $p \equiv 3 \pmod{4}$ are inert and the primes $p \equiv 1 \pmod{4}$ are split. Consequently, $1728$ is a supersingular $j$-invariant if $p=2$ or $p \equiv 3 \pmod{4}$. Working mod $12$ we see that we found explicit supersingular $j$-invariants if $p=2, 3$ or $p = 5,7,11 \pmod{12}$. This leaves the primes $p$ that are congruent to $1$ mod 12. For such primes one uses ad hoc methods that depend on the prime, for instance by looking for an imaginary quadratic field with small discriminant in which $p$ is inert. One could alternatively work with CM elliptic curves for the field $\QQ(\sqrt{-p})$. Note though that if $p$ is of cryptographic size $p \approx 2^{500}$, then by the Brauer-Siegel theorem the class number of $\QQ(\sqrt{-p})$ is approximately of size $\sqrt{p} \approx 2^{250}$. Thus, the minimal field of definition of such a CM curve would be of degree $\approx 2^{250}$ over $\QQ$, which is completely infeasible to do computations with.

\subsubsection{Lifting theorems} We now discuss a converse to the reduction of elliptic curves with CM. Let $E$ be an elliptic curve over $\FF_{p^r}$ and let $f\in \End(E)$ be a non-integer endomorphism of $E$. Let $x^2 + bx +c$ be the characteristic polynomial of $f$ as an element of the ring $\End(E)\subset \ZZ$. Deuring's lifting theorem says that there is an elliptic curve $\tilde E$ over the Witt ring $W(\FF_{p^r})$, or a ramified quadratic extension of it, and $\tilde f\in \End(\tilde E)$ with characteristic polynomial $x^2+bx+c$, such that the reduction of $(\tilde E, \tilde f)$ 
modulo the uniformizer of $W(\FF_{p^r})$ is $(E, f)$; see \cite[Proposition 2.7]{GZ}, \cite[13, \S 5]{Lang}. A simpler statement, which the above theorem can be used to prove, is that every elliptic curve $E$ over a finite field is a reduction of an elliptic curve over a number field with complex multiplication.

One can ask a related question: consider a quadratic imaginary field $K$ of class number $h_K$, and elliptic curves $\tilde E_1, \dots, \tilde E_{h_K}$ with CM by $\calO_K$, all defined over a number field $k$ and having good reduction at a prime $v$ of $k$ over $p$. Let $E_i = \tilde E_i \pmod{v}$ and assume that $p$ is inert in $K$. If $h_K \gg p$, is every supersingular elliptic curve over $\fpbar$ isomorphic to some $E_i$? By Deuring's lifting theorem, this is equivalent to asking whether every maximal order $\calO$ of $B_{p, \infty}$ admits an embedding of $\calO_K \arr \calO$. This is true for $h_K$ sufficiently large, and follows from a stronger result due to Elkies, Ono and Yang \cite{EOY}.

\subsection{$\Hom(E_1, E_2)$ as a quadratic module.}\label{subsec:Hom quadratic module} Let $E_1,E_2$ be supersingular elliptic curves over $\FF_{p^2}$, and recall that we denote $\Hom(E_1,E_2):=\Hom_{\fpbar}(E_1,E_2)$. If we set the degree of the zero homomorphism to be $0$, then the degree function 
\[ \deg\colon \Hom(E_1, E_2) \arr \ZZ, \]
is a quadratic form with associated bilinear form 
 \[ (f, g) = \frac12\left(\deg(f+g) - \deg(f) - \deg(g)\right).\]
 Note that we can interpret $\deg(f)$ as $f^\vee f$, where $f^\vee$ is the dual isogeny. Then, using also that  $(f+g)^\vee = f^\vee + g^\vee$, we find
  \[ (f, g) = \frac12\left(g^\vee f + f^\vee g\right).\]
It is easy to check that this is a bilinear form when written in this way.
  
 Fixing a $\ZZ$-basis $v_1, \dots, v_4$ for $\Hom(E_1, E_2)$, the matrix $A = (2(v_i, v_j))_{i, j = 1}^4$ has determinant~$p^2$, so the quadratic form $\deg$ has discriminant $p^2$.
 We say that $\deg$ has {\bf level} (or \textit{stufe}) $p$ because $p$ is the smallest positive integer with the property that $pA^{-1}$ has integer entries. For example, take $E_1 = E_2$ such that $\End(E_i)$ is the maximal order given in \S \ref{subsubsection: Bp again}. The determinant of the matrix given there is $p^2$ and its inverse is 
 \[ \frac{1}{p}\left(\begin{smallmatrix} \frac{p+1}{2} & 0 & 0 &-1\\0 & \frac{p+1}{2} & -1&0\\ 0 & -1& 2& 0\\-1&0&0&2
 \end{smallmatrix}\right).\]
 Recall that $\End_{\fpbar}(E_1)$ is isomorphic to a maximal order $\calO$ in $B_{p,\infty}$.
 As $\Hom(E_1, E_2)$ is a right projective rank $1$ module over $\End_{\fpbar}(E_1)$, it is isomorphic as a right $\calO$-module to some ideal~$I$ that is only determined up to equivalence, $I \sim \lambda I$ for some $\lambda\in B_{p,\infty}^\times$. One can show that under such an isomorphism $\deg$ corresponds to $\frac{1}{\Nm(I)} \Nm(\cdot )$, a formula that indeed depends only on the ideal class of $I$. Here $\Nm(I)$ is a positive generator of the fractional ideal in $\QQ$ generated by~$\Nm(\alpha)$ for $\alpha\in I$.

An important fact is that the quadratic forms $(\Hom(E_1, E_2), \deg)$, for varying supersingular elliptic curves $E_i$, all belong to the same genus; that is, their localizations at every place are isomorphic. This is easy to check, because locally at a prime $\ell \neq p$ we may assume that $\End(E_1) = M_2(\ZZ_\ell)$, the ideal $I$ corresponding to $\Hom(E_1, E_2)$ is $M_2(\ZZ_\ell)$, and the degree is the determinant. At $\ell = p$ we likewise have a description of $\calO\otimes \ZZ_p$ (\ref{eq:ramified_maxorder}) and we may again take $I = \calO\otimes \ZZ_p$ with the degree function being induced by the determinant on $M_2(\QQ_{p^2})$. Incidentally, this gives another way to find the discriminant and level. Less obvious, but true, is that every integral quadratic form in $4$ variables lying in this genus is realized by $(\Hom(E_1, E_2), \deg)$ for some $E_1,E_2$ \cite[Corollary 63]{Kohel}.

\section{Supersingular isogeny graphs} Supersingular isogeny graphs are directed graphs associated to supersingular elliptic curves and the isogenies between them. Similar constructions can be made for the minimal Ekedahl-Oort strata of the mod $p$ fiber of a Shimura variety, for example for superspecial abelian varieties with real multiplication -- an example we discuss below and that we hope to return to in future work. Supersingular isogeny graphs where considered already by Mestre \cite{Mestre}, if not before, as a tool for computing modular forms on $\Gamma_0(p)$, but have come to broader attention more recently for their cryptographic applications. 

Isogeny graphs can also be constructed for ordinary elliptic curves. These behave very differently from supersingular isogeny graphs, but have many interesting computational applications; see \cite{Sutherland} for a survey.

\subsection{Definition of supersingular isogeny graphs} Let $p$ be a fixed prime and $\ell \neq p$ another prime. In many applications one thinks of $p$ as ``big'', of ``cryptographic size'' (e.g.~$p\approx 2^{500}$), and of $\ell$ as ``small''; very often $\ell = 2$ is a completely adequate choice. The {\bf supersingular isogeny graph} is a directed graph $\scrG(p, \ell)$ with one vertex for each supersingular $j$-invariant. To every such $j$-invariant $j$ we choose once and for all an elliptic curve $E(j)$, defined over $\FF_{p^2}$, for which $\Fr^2 = [-p]$. 
Given vertices $j_1,j_2$, the edges from $j_1$ to $j_2$ correspond to isogenies $f\colon E(j_1) \arr E(j_2)$ of degree $\ell$, considered up to automorphisms of $E(j_2)$. That is, the edges out of $j_1$ correspond to subgroups $H \subset E(j_1)[\ell]$ of order $\ell$, and such an edge points to $j_2$ if $E(j_1)/H \cong E(j_2)$. See Figure~\ref{fig:supsing examples} for an example of two supersingular isogeny graphs with the same prime $p$ and distinct values of $\ell$.

It is confusing, but within tolerance, to denote the vertices by $1, \dots, h$ (with $h = h_{B_{p, \infty}}$), or by the supersingular $j$-invariants $j_1, \dots, j_h \in \FF_{p^2}$, or by the corresponding supersingular elliptic curves $E(j_1), \dots, E(j_h)$.

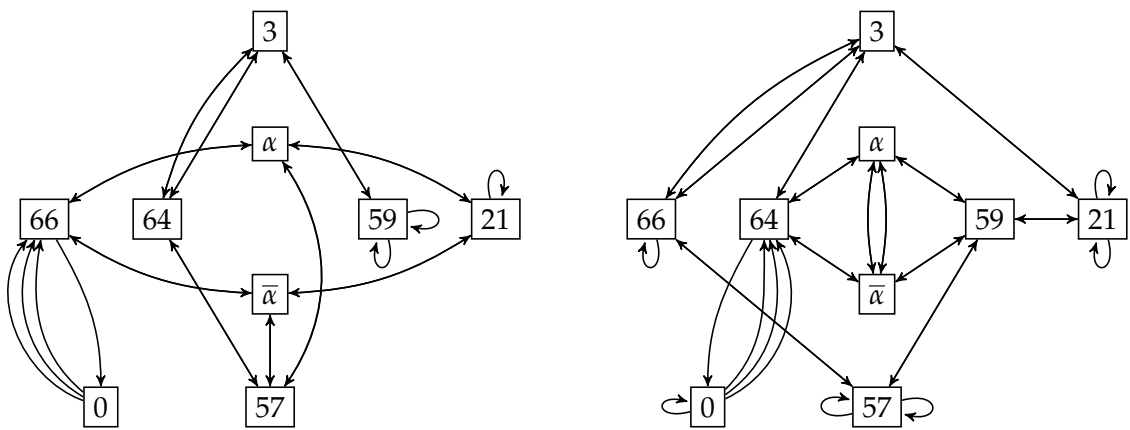
\begin{figure}
	\begin{tabular}{ccc}
		\begin{tikzpicture}[->, >=stealth', auto, semithick, node distance=2.5cm]
			
			\tikzstyle{every state}=[fill=none,draw=black,text=black]
			
			\node[draw] (0) at (3, 0) {$21$};
			\node[draw] (1) at (0, 1) {$\alpha$};
			\node[draw] (2) at (0, -1) {$\overline{\alpha}$};
			\node[draw] (3) at (-3, 0) {$66$};
			\node[draw] (4) at (0, -2.5) {$57$};
			\node[draw] (5) at (-2.25, -2.5) {$0$};
			\node[draw] (6) at (-1.5, 0) {$64$};
			\node[draw] (7) at (0, 2.5) {$3$};
			\node[draw] (8) at (1.5, 0) {$59$};
			
			\path (0) edge [loop above] node {} (0);
			\path (0) edge [bend left=15] node {} (2);
			\path (0) edge [bend right=15] node {} (1);
			
			\path (1) edge [bend right=15] node {} (3);
			\path (1) edge [bend left=15] node {} (0);
			\path (1) edge [out=305, in=55] node {} (4);
			
			\path (2) edge node {} (4);
			\path (2) edge [bend right=15] node {} (0);
			\path (2) edge [bend left=15] node {} (3);
			
			\path (3) edge [out=300,in=90] node {} (5);
			\path (3) edge [bend right=15] node  {} (2);
			\path (3) edge [bend left=15] node  {} (1);
			
			\path (4) edge node {} (6);
			\path (4) edge [out=55,in=305] node {} (1);
			\path (4) edge node {} (2);
			
			\path (5) edge [out=140, in=260] node {} (3);
			\path (5) edge [out=150, in=245] node {} (3);
			\path (5) edge [out=160,in=230] node {} (3);
			
			\path (6) edge node {} (4);
			\path (6) edge [bend left=15] node {} (7);
			\path (6) edge  node {} (7);
			
			\path (7) edge node {} (8);
			\path (7) edge node {} (6);
			\path (7) edge [bend right=15] node {} (6);
			
			\path (8) edge [loop right] node {} (8);
			\path (8) edge [loop below] node {} (8);
			\path (8) edge node {} (7);
			
		\end{tikzpicture}
		& \hspace{20pt} &
		\begin{tikzpicture}[->, >=stealth', auto, semithick, node distance=2.5cm]
			
			\tikzstyle{every state}=[fill=none,draw=black,text=black]
			
			\node[draw] (0) at (3, 0) {$21$};
			\node[draw] (1) at (0, 1) {$\alpha$};
			\node[draw] (2) at (0, -1) {$\overline{\alpha}$};
			\node[draw] (3) at (-3, 0) {$66$};
			\node[draw] (4) at (0, -2.5) {$57$};
			\node[draw] (5) at (-2.25, -2.5) {$0$};
			\node[draw] (6) at (-1.5, 0) {$64$};
			\node[draw] (7) at (0, 2.5) {$3$};
			\node[draw] (8) at (1.5, 0) {$59$};
			
			\path (0) edge  node {} (7);
			\path (0) edge  node {} (8);
			\path (0) edge [loop above] node {} (0);
			\path (0) edge [loop below] node {} (0);
			
			\path (1) edge node {} (6);
			\path (1) edge [bend left=10] node {} (2);
			\path (1) edge node {} (8);
			\path (1) edge [bend right=10]  node {} (2);
			
			\path (2) edge [bend left=10] node {} (1);
			\path (2) edge node {} (8);
			\path (2) edge [bend right=10] node {} (1);
			\path (2) edge node {} (6);
			
			\path (3) edge [loop below] node {} (3);
			\path (3) edge node {} (4);
			\path (3) edge  node {} (7);
			\path (3) edge [bend left=15] node {} (7);
			
			\path (4) edge node {} (3);
			\path (4) edge node {} (8);
			\path (4) edge [loop right] node {} (4);
			\path (4) edge [loop left] node {} (4);
			
			\path (5) edge [loop left] node {} (5);
			\path (5) edge [out=45, in=270] node {} (6);
			\path (5) edge [out=35, in=285] node {} (6);
			\path (5) edge [out=25, in=300] node {} (6);
			
			\path (6) edge [out=240, in=90] node {} (5);
			\path (6) edge node {} (7);
			\path (6) edge node {} (2);
			\path (6) edge node {} (1);
			
			\path (7) edge node {} (6);
			\path (7) edge  node {} (0);
			\path (7) edge  node {} (3);
			\path (7) edge [bend right=15] node {} (3);
			
			\path (8) edge node {} (4);
			\path (8) edge  node {} (0);
			\path (8) edge node {} (2);
			\path (8) edge node {} (1);
			
		\end{tikzpicture}
	\end{tabular}
	\caption{$\scrG(101,2)$ on the left and $\scrG(101,3)$ on the right. Vertices are labelled by $j$-invariants of supersingular curves over $\FF_{101^2}$, where $\alpha,\overline{\alpha}\in\FF_{101^2}$ are the roots of $x^2 + 27x + 54$. Each single-headed arrow is a directed edge of the graph, and each double-headed arrow corresponds to two directed edges, one in each direction.}\label{fig:supsing examples}
\end{figure}

\subsubsection{Variants} Fix a positive integer $n$, not divisible by $p$ or $\ell$. One can enrich the elliptic curves by a choice of a point of order $n$ and construct a supersingular isogeny graph $\scrG(p, \ell, \Gamma_1(n))$. The vertices correspond to $\fpbar$-isomorphism classes of pairs $(E, P)$, where $E$ is supersingular and~$P$ a point of exact order $n$, and edges correspond to isogenies $f\colon (E_1, P_1) \arr (E_2, P_2)$, such that~$f$ is an isogeny of degree $\ell$ such that $f(P_1) = P_2$, considered up to automorphisms of $(E_2, P_2)$. 
The reader familiar with modular curves will have no difficulty imagining other variants such as $\scrG(p, \ell, \Gamma_0(n))$ and $\scrG(p, \ell, \Gamma(n))$; collectively these graphs are referred to as supersingular isogeny graphs {\bf with level structure}.
These are interesting variants on which there has been some work recently \cite{Arpin, GK, PW, Roda}. Note that there is a natural morphism from any such  graph to $\scrG(p, \ell)$. 

In another direction, for the supersingular graph $\scrG(p, \ell)$ one can extract the vertices defined over $\FF_p$. There are about $\sqrt{p}$ of these; more precisely, letting $h$ be the class number of $\QQ(\sqrt{-p})$, there are $h, 2h$ or $h/2$ if $p\equiv 7\pmod{8}, 3\pmod{8}$ or $1 \pmod{4}$, respectively. For two vertices $E_i, E_j$ in this set, we include all the same edges from $E_i$ to $E_j$ as in the original graph $\scrG(p, \ell)$. This gives the so-called {\bf spine} of the supersingular graph $\scrG(p, \ell)$. Another variant is to consider supersingular elliptic curves defined over $\FF_p$ and $\FF_p$-rational isogenies of degree $\ell$ between them, up to $\FF_p$-rational automorphisms of the target. Such studies are motivated by cryptographic considerations. These have been studied in the literature, see \cite{Adj, Arpin2, DG} and references therein.

 \subsection{Properties of supersingular isogeny graphs} 
  \subsubsection{Degree} Since $\ell\neq p$, for every elliptic curve $E$ over $\fpbar$ we have $E[\ell]\cong (\ZZ/\ell \ZZ)^2$. As a result, $E$ has  exactly $\ell+1$ subgroups of order $\ell$ and so $\scrG(p, \ell)$ has out-degree $\ell+1$. Note that by the definition of edges, if there are $a$ edges from $E_1$ to $E_2$, then there are in fact $a\cdot \sharp \Aut(E_2)$ isogenies $f\colon E_1 \arr E_2$ of degree $\ell$. Every such  isogeny $f$ has a dual isogeny $f^\vee\colon E_2 \arr E_1$. This shows that if there is an edge from $E_1$ to $E_2$ then there is an edge from $E_2$ to $E_1$, and in fact
  \begin{align}\label{eq: in out aut}
  	\sharp \{ {\rm edges\;} E_1 \arr E_2 \}\cdot \sharp \Aut (E_2) =  \sharp \{ {\rm edges\;} E_2 \arr E_1 \}\cdot \sharp \Aut (E_1).
  \end{align}
Thus, if $\Aut(E_1) = \{ \pm 1\}$, and every neighbouring vertex $E_2$ of $E_1$ also has $\sharp \Aut(E_2) = \{ \pm 1\}$, then $E_1$ has in-degree $\ell+ 1$ as well. If $p \equiv 1 \pmod{12}$ then every supersingular elliptic curve $E$ has $\sharp \Aut(E) = \{ \pm 1\}$, and so $f \mapsto f^\vee$ establishes a bijection on edges. Identifying each isogeny with its dual allows us to consider $\scrG(p, \ell)$ as undirected graph of degree $\ell+1$. 

In general there may be vertices $E$ with extra automorphisms. By  Deuring's lifting theorem, these would be the elliptic curves with $j = 0$ or $1728$, provided they are supersingular; as we have seen, this depends only on the congruence class of $p$ modulo $12$. For $p=2, 3$ there is only one supersingular $j$-invariant, which is the $j$ invariant $0$. Whenever $j=0$ or $1728$ is a supersingular $j$-invariant, there is no longer a one-to-one correspondence between inward edges and outward edges at $j$ or at neighbours of $j$; the in-degree at $j$, or its neighbours, can be computed using (\ref{eq: in out aut}). See Figure~\ref{fig:supsing examples} for an example of this behaviour; note that $j=0$ is the unique supersingular $j$-invariant in $\FF_{101^2}$ with extra automorphisms. 

\subsubsection{Brandt matrices}\label{subsubsec: brandt matrices} The reference \cite{Gross} is very useful here. Enumerate the supersingular curves as $E_1, \dots, E_h$, where $h = h_p$ is the class number of $B_{p, \infty}$ and $E_i:=E(j_i)$. For every integer $n\geq 0$, define the {\bf Brandt matrix} $B(n)$ to the $h \times h$ matrix such that
\begin{align}\label{eq:brandt}
	B(n)_{ij} = \frac{1}{\sharp \Aut(E_j)} \cdot \sharp\{ f\colon E_i \arr E_j: \deg(f) = n\}.
\end{align}
The $n=0$ case uses our convention that the degree of the zero map is $0$.
For $n\geq 1$, we can also say that $B(n)_{ij}$ is the number of group schemes $H$ of rank $n$ of $E_i$ such that $E_i/H \cong E_j$. Note in particular that $B(1)$ is the identity matrix, and that $B(p)_{ii}$ is $1$ if $j(E_i) \in \FF_p$ and is equal to $0$ otherwise (this follows from (\ref{eq:factor through Fr})). That is, a supersingular $j$-invariant is in $\FF_p$ if and only if $E(j)$ has an endomorphism of degree $p$.  In particular, 
\[ \Tr(B(p)) = \text{number of supersingular $j$-invariants in $\FF_p$}.\]
Note that this trace is exactly the quantity $u_B$ discussed at the end of \S \ref{subsubsec: ideal classes}. More generally, $B(p)_{ij} = 1$ if $j(E_j) = j(E_i)^{p}$ and is $0$ otherwise. For prime $\ell\neq p$, $B(\ell)_{ij}$ equals the number of edges from $i$ to $j$ in $\scrG(p,\ell)$; that is, $B(\ell)$ is the adjacency matrix of the directed graph $\scrG(p,\ell)$.

In general, the traces of Brant matrices can be computed explicitly as follows \cite[Proposition 1.9]{Gross}: For any integer $n\geq 0$,
	\begin{align}\label{eq:brandt trace}
		 \Tr(B(n)) = \sum_{s\in \ZZ, s^2 \leq 4n} H_p(4n-s^2),
	\end{align}
where the $H_p(4n-s^2)$ are $p$-modified class Hurwitz class numbers; cf. loc. cit. (1.8). In loc. cit. (1.11), Gross uses this formula to deduce an explicit expression for the number of supersingular $j$-invariants over $\FF_p$ and over $\FF_{p^2}$. For further applications to the structure of supersingular isogeny graphs see \S \ref{subsubsec:loops multiple}.
 
The Brandt matrices are multiplicative in the sense that for $m,n\geq 1$,
\[ B(m)B(n) = B(mn), \quad \text{if } \gcd(m, n) = 1, \]
and they satisfy
\[ B(\ell^{a+1}) = B(\ell)B(\ell^a) - \ell B(\ell^{a-1}), \quad \ell\neq p \text{ prime, } a\geq 1.\]
Furthermore, $B(p)^2 = Id$.

\subsubsection{The modular polynomial} Let $n\geq 1$ be an integer. The {\bf modular polynomial}, initially defined over $\CC$, is the polynomial in two variables $j, j^\prime$
\[ \Phi_n(j, j^\prime), \]
that vanishes on $(j_1, j_2)\in\CC^2$ precisely when there is a cyclic degree $n$ isogeny $E_1 \arr E_2$ for elliptic curves $E_1,E_2/\CC$ with $j(E_i)=j_i$. Thus, the polynomial $\Phi_n(j_1, j_2)$ defines
a (typically singular) curve in the plane $\AA^2_{j_1, j_2}$ that is birationally equivalent to the modular curve $X_0(n)$. 
The definition implies that $\Phi_n(j_1, j_2)$ is a symmetric polynomial, and when considered as a polynomial in~$j_1$ it has degree equal to the number of cyclic subgroups of $(\ZZ/n\ZZ)$, which is nothing but the Dedekind $\psi$-function 
\[\psi(n) =n \prod_{p\vert n}(1 + \frac{1}{p}).\] 
Further, suppose that for $\tau\in\CC$ with $\Ima(\tau)>0$ we define $j(\tau)$ to be the $j$-invariant of the elliptic curve $\CC/(\ZZ\oplus\ZZ\tau)$. Then for every such $\tau$, $\Phi_n$ vanishes at the pair $(j(\tau), j(n\tau))$, and in fact every zero of $\Phi_n$ can be written in this way for appropriate $\tau$.
See\cite[5, \S\S 2-3]{Lang}.

A fact that is less straightforward to show is that normalizing  $\Phi_n(j_1, j_2)$ to be monic, we have $\Phi_n(j_1, j_2)\in \ZZ[j_1, j_2]$. We can therefore reduce $\Phi_n$ modulo $p$, and for every $p \neq n$ this reduction modulo $p$ vanishes on pairs $(j_1, j_2) \in \fpbar^2$ such that there is a cyclic degree $n$ isogeny from $E(j_1)$ to $E(j_2)$. In particular, if $E(j_1)$ is a supersingular elliptic curve then its neighbours in $\scrG(p, \ell)$ are the solutions of the degree $\ell+1$ monic polynomial $\Phi_\ell(j_1, j_2) \in \FF_p[j_2]$. 

The modular polynomials in $\ZZ[j_1, j_2]$ have enormous coefficients. For example, 
\begin{multline*}\Phi_2(j_1, j_2) =\\ \footnotesize{\text{$j_1^3 + j_2^3 + 1488(j_1^2j_2 + j_2^2j_1) -162000 (j_1^2 + j_2^2) - j_1^2j_2^2 + 8748000000 (j_1+j_2) + 40773375 j_1j_2 - 157464000000000,$}}\end{multline*}
and one of the coefficients of $\Phi_3$ is greater than $10^{21}$.
For more examples, see \cite{Sut}. Interesting work has been done on the calculation of $\Phi_n$, for instance see \cite{BLS}.

\subsubsection{Loops and multiple edges}\label{subsubsec:loops multiple} Fix a prime $\ell$. For $p\neq \ell$, the number of loops (edges from a vertex to itself) and multiple edges (pairs of edges with the same start and end vertices) of $\scrG(p, \ell)$ is bounded as a function of $\ell$ alone. To see why this is so, consider first loops. If $f\colon E \arr E$ is a loop then $f$ is an endomorphism of degree $\ell$ of $E$ and its characteristic polynomial is of the form $x^2 + ax + \ell$ for some $a\in \ZZ$. This $f$ generates a quadratic imaginary order and so we must have $a^2 < 4\ell$. Thus, there are finitely many such orders. By Deuring's lifting theorem $(E, f)$ can be lifted to $(\tilde E, \tilde f)$ and, considered as complex elliptic curves there are only finitely many possibilities for $\tilde E$; in fact, at most the sum of the class numbers of $\ZZ[x]/(x^2 + ax + \ell)$ and of all the orders of $\QQ[x]/(x^2 + ax + \ell)$ containing it. One can make this precise: the number of loops in $\scrG(p, \ell)$ is none other then the trace $\Tr(B(\ell))$ of the $\ell$-th Brandt matrix, as given in (\ref{eq:brandt trace}).

The situation for multiple edges is similar. 
Suppose $f_1, f_2\colon E_1 \arr E_2$ define multiple edges. Assume for simplicity that $\Aut(E_1)=\{\pm 1\}$, as the analysis is more involved when $E_1$ has extra automorphisms. Then $f_2^\vee f_1$ is an endomorphism of $E_1$ of degree $\ell^2$. If $\Ker(f_2^\vee f_1) =E_1[\ell]$ then $f_2^\vee f_1 = [\pm\ell]$, and so $f_1 = \pm f_2$, which means that they define the same edge, contrary to hypothesis. Thus, $f_2^\vee f_1$ must be a non-integer endomorphism of $E_1$ with degree $\ell^2$. 
By the same technique as for loops, the number of curves with a non-integer endomorphism of degree $\ell^2$ is bounded as a function of $\ell$ alone. For a detailed study see \cite[Lemma 2.7]{Arpin2},\cite{Ghantous}.

\subsubsection{Graph automorphisms} Given (any) two elliptic curves $E_1, E_2$ over a perfect field $k$, base change along the Frobenius map provides a homomorphism of rings 
\[ \Hom_k(E_1, E_2) \arr \Hom_k(E_1^{(p)}, E_2^{(p)}),\quad f \mapsto f^{(p)}.\]
This is entirely obvious when phrased in more generality:  for an injection of fields $k \arr k_1$ there is a ring homomorphism  
\[ \Hom_k(E_1, E_2) \arr \Hom_{k_1}(E_1\otimes_kk_1, E_2\otimes_kk_1),\quad f \mapsto f\otimes Id, \]
only that here the field injection is $\Fr\colon k \arr k$ and we denote the resulting homomorphism $f^{(p)}$. Since $k$ is perfect, we likewise have a base-change relative to $\Fr^{-1}$, and so there is an isomorphism 
\[ \Hom_k(E_1, E_2) \cong \Hom_k(E_1^{(p)}, E_2^{(p)}). \]
We conclude that there is an automorphism of the graph $\scrG(p, \ell)$ induced by Frobenius: it takes a vertex $E(j)$ to $E(j)^{(p)} \cong E(j^p)$, and takes an edge corresponding to $f\colon E_1 \arr E_2$ (up to $\Aut(E_2)$) to the edge $f^{(p)}\colon E_1^{(p)} \arr E_2^{(p)}$ (up to $\Aut(E_2^{(p)})$). Further, as all our objects -- both the elliptic curves and their isogenies -- are defined over $\FF_{p^2}$, this automorphism of $\scrG(p, \ell)$ is an involution; we will denote it also by~$\Fr$. As we shall explain in \S \ref{subsec: rigiditiy} below, under quite general conditions, this is essentially the only non-trivial automorphism of $\scrG(p, \ell)$.

\subsubsection{Connectedness and expansion}\label{subsec:connected and expansion} The graphs $\scrG(p, \ell)$ are connected. We sketch two proofs of this fact; each relies on a non-trivial theorem. The first method uses a local-global principle for representability of integers by quadratic forms, as in the proof of Theorem~\ref{thm: special generators for a lattice}; the second uses modular forms, starting a thread that will be continued further in \S\ref{sssection: thetas}--\ref{sssection: vertex functions}. These two perspectives are deeply related, as will be explored in \S\ref{section6}.

\medskip

\id {\bf A.} Let $Q$ be a positive-definite integral quadratic form in $n\geq 4$ variables and discriminant $d$. Assume that $q$ represents an integer $N$ everywhere locally and that $\gcd(N, d) = 1$. That is, one can solve $q(x_1, \dots, x_n)$ in $q$-adic integers for every prime $q$. Then, if $N\gg 0 $, $Q$ represents~$N$ over~$\ZZ$. For example, one may use \cite[Theorem 6.3]{Hanke1} and that a quadratic form can be anisotropic only at primes dividing its discriminant \cite[Remark 3.8.1]{Hanke1}.

Let $E, E^\prime$ be supersingular elliptic curves. We apply the theorem for $Q = \deg$ on $\Hom(E, E^\prime)$ and $N = \ell^k$. The local conditions require us to check that $\det$ on $M_2(\ZZ_q)$ represents $N$, which is obvious (take the matrix $\diag(1, N)$) and that $\det$ on the maximal order $\left\{ \left(\begin{smallmatrix}a & pb \\
b^\sigma & a^\sigma
\end{smallmatrix}\right): a, b \in \calO_{\QQ_{p,2}} \right\}$ represents $N$; this amounts to solving $aa^\sigma - pbb^\sigma = N$ for $a, b \in \calO_{\QQ_{p,2}}$, which is possible by local class field theory but can also be done ``by hand''. Thus, for every pair of vertices $E, E^\prime$, there is an isogeny $f\colon  E \arr E^\prime$ of degree $\ell^k$ for some $k$. Let $K_k = \Ker(f)$ and choose subgroups such that 
\[ K_k \supset K_{k-1} \supset \dots \supset K_1 \supset K_0=\{0\}, \quad \sharp K_i = \ell^i.\]
We  obtain a sequence of $\ell$-isogenies
\[ E = E/K_0 \arr  E/K_1 \arr \cdots \arr E/{K_{k-1}} \arr E/K_k \cong E^\prime,\]
which defines a path of length $k$ from $E$ to $E^\prime$ in $\scrG(p, \ell)$. Incidentally, note that if $k\gg 0$ then for every supersingular $E, E^\prime$ there is an isogeny $E \arr E^\prime$ of degree $\ell^k$ and also an isogeny of degree~$\ell^{k+1}$; this proves that the graphs $\scrG(p, \ell)$ are not bipartite. 

\medskip

\id {\bf B.} The adjacency matrix of the oriented graph $\scrG(p, \ell)$ is conjugate to the matrix of the Hecke operator $T_\ell$ acting on modular forms of weight $2$ on $\Gamma_0(p)$ \cite[II \S6]{Eichler1}, \cite[\S9]{Eichler2}. Now, for a finite oriented graph with out-degree $\ell+1$ the number of connected components of the graph is the multiplicity of the eigenvalue $\ell+1$ of the adjacency matrix $A= (a_{ij})$, where $a_{ij}$ is the number of edges from the vertex $i$ to the vertex $j$ (and so $\sum_{j} a_{ij} = \ell+1$) at least if the graph has the property that when $i$ is connected to $j$ then also $j$ is connected to $i$, and so there is no ambiguity about the meaning of a connected component. This property holds for our graphs. Thus, one needs to show that the multiplicity of the eigenvalue $\ell+1$ for $T_\ell$ acting on modular forms $M_2(\Gamma_0(p))$ is $1$. In fact, there is a decomposition of Hecke modules,
\[ M_2(\Gamma_0(p)) = \CC \cdot \scrE_2 \oplus S_2(\Gamma_0(p)),\]
where $\scrE_2$ is an {\bf Eisenstein series} for which $T_\ell \scrE_2 = (\ell+1) \scrE_2$. For any eigenform $f\in S_2(\Gamma_0(p))$, say $T_\ell f = \lambda_f f$, one has the Deligne bound $|\lambda_f|< 2 \sqrt{\ell}$. Thus, the eigenspace for the eigenvalue $\ell+1$ is $1$-dimensional, and one concludes that $\scrG(p, \ell)$ is connected. In fact, this also shows that the graph $\scrG(p, \ell)$ is Ramanujan in the sense of \cite{LPS}; for this statement it is perhaps best to assume that $p\equiv 1 \pmod{12}$ in which case we can think about $\scrG(p, \ell)$ as an undirected graph.

The Eisenstein series $\scrE_2$ is known explicitly\cite[(5.7)]{Gross}:
\begin{align}\label{eq:eisenstein p}
	\scrE_2 = \frac{p-1}{24} + \sum_{m\geq 1} \sigma(m)_p q^m,
\end{align}
where $\sigma(m)_p = \displaystyle\sum_{d\vert m,\; p\nmid d} d$. For any $m\geq 1$, $T_m(\scrE_2) = \sigma(m)_p \scrE_2$, and in particular $T_\ell \scrE_2 = (\ell+1) \scrE_2$.

\subsubsection{The Deuring correspondence}\label{subsubsec: deuring corr} Some good references for this section include \cite{Kohel} and \cite[Chapter 42]{Voight}.
As above, we list the supersingular $j$-invariants $j_1,\ldots,j_h$ with $h=h_{B_{p,\infty}}$, and let $E_i:=E(j_i)$ for each $i$. We will fix a base point $j_1$; note that by reordering, any supersingular $j$-invariant can be used as $j_1$. We let $\calO:=\End(E_1)$, a maximal order in the quaternion algebra $\calO\otimes \QQ\cong B_{p,\infty}$.

As described in \S \ref{subsec:Hom quadratic module}, for every $i=1,\ldots, h$, $\Hom(E_1,E_i)$ is isomorphic as a right $\calO$-module to some right ideal of $\calO$. Explicitly, if we take any isogeny $\phi\colon E_i\to E_1$, then 
\[I_{i}:=\phi\circ \Hom(E_1,E_i)\]
is a subset of $\calO=\End(E_1)$, and in fact a right ideal of $\calO$; the right ideal class does not depend on the choice of $\phi$. For $i=1$ we can take $\phi=\text{id}$ and therefore $I_1=\calO$. The {\bf Deuring correspondence} says that:
\begin{enumerate}
	\item $I_1, \dots, I_h$ form a complete set of representatives for the right ideal classes of $\calO_i$;
	\item The degree map on $\Hom(E_1,E_i)$ corresponds to $\frac{1}{\Nm(I_i)}\Nm(\cdot)$ on $I_i$;
	\item The left order of $I_i$ is isomorphic to $\End(E_i)$. 
\end{enumerate}
From (2) we see that there is an edge from $j_1$ to $j_i$ in $\scrG(p, \ell)$ if and only if the quadratic form $\frac{1}{\Nm(I_i)}\Nm(\cdot)$ on $I_i$ represents $\ell$. Moreover, $j_1\in\FF_p$ if and only if the norm form on $I_1=\calO$ represents~$p$. For $i\neq 1$, we have $\End(E_i) \cong \End(E_1)$ if and only if the elliptic curves~$E_i$ and~$E_1$ are the Frobenius base-change of each other, which is the case if and only if $\frac{1}{\Nm(I_i)}\Nm(\cdot)$ represents~$p$. One concludes that the class number $h_{B_{p, \infty}}$ is the number of vertices of $\scrG(p, \ell)$, while the type number $t_{B_{p, \infty}}$ (as in \S \ref{subsubsec: ideal classes}) is the number of orbits for the automorphism $\Fr$ of the graph.

\medskip

\id
Define the {\bf supersingular polynomial} 
\[ S_p(x) = \prod_{j \in \fpbar \text{ supersingular}} (x - j).\]
This is a polynomial of degree $h_{B_{p, \infty}}$ belonging to $\FF_p[x]$ that factors as a product of linear and quadratic terms over $\FF_p$. The number of linear terms is the number of supersingular $j$-invariants in $\FF_p$ and the number of quadratic terms is the number of pairs $\{j, j^p\}$ where $j$ is a supersingular $j$-invariant not in $\FF_p$. Thus, the number of irreducible factors of $S_p$ is the type number $t_{B_{p, \infty}}$. Here are a few examples: 
\begin{align*}
	S_{37}(x) &= (x-8)(x^2 - 6x - 6)\\
	S_{47}(x) &= x(x-1728)(x-9)(x-10)(x+3)\\
	S_{73}(x) &= (x-9)(x+17)(x^2-5x+9)(x^2-16x+8)\\
	S_{101}(x) &= x(x-3)(x-21)(x-57)(x-59)(x-64)(x-66)(x^2+27x+54).
\end{align*}
The primes $p$ for which $S_p(x)$ splits into linear factors -- in other words, for which all supersingular $j$-invariants are defined over $\FF_p$ -- are $p\leq 31$ and $p=41,47,59,71$ (these so-called ``supersingular primes'' are exactly the primes for which the quotient of the modular curve $X_0(p)$ by the Atkin-Lehner involution $w_p$ has genus $0$, and are also exactly the prime factors of the order of the monster group). For all other primes, $S_p(x)$ has at least one quadratic factor.

\medskip

\id Based on the Deuring correspondence, one can define the supersingular graphs $\scrG(p, \ell)$ and the Brandt matrices without any reference to elliptic curves. Namely, start with a choice of a maximal order $\calO_1$ of $B_{p, \infty}$ and let $I_1, \dots, I_h$ be representatives for its right ideal classes. Let $\calO_i := \calO_\ell(I_i)$. Every maximal order in $B_{p, \infty}$ is conjugate to some $\calO_i$ (though most conjugacy classes of maximal orders appear twice in the list $\calO_1, \dots, \calO_h$). The vertex~$1$ is connected to the vertex $i$ in $\scrG(p, \ell)$ by the number of times the form $\frac{1}{\Nm(I_i)}\Nm(\cdot)$ on $I_i$ represents $\ell$, divided by $\sharp \calO_i^{\times}$. The same procedure can now be done for the order $\calO_i$ to determine the edges from the vertex $i$ to any vertex $j$ -- there is no need to coordinate the choice of representatives for right ideal classes with those already made. Thus, the construction of the Brandt matrices can be done without any reference to supersingular elliptic curves. This allows the construction to be significantly generalized; for instance, Brandt matrices can also be defined for quaternion algebras over totally real fields \cite{Eichler2}. See further comments in \S \ref{sec: ssg for tot real}.

However, it is important to stress that even if the graph $\scrG(p,\ell)$ can be defined equally well using elliptic curves or quaternion ideals, these two implementations have drastically different security analyses when considered for cryptographic applications. In particular, path-finding on $\scrG(p,\ell)$ is believed to be hard when the vertices are labelled by elliptic curves or their $j$-invariants, but is known to be easy when the vertices are labelled by quaternion ideals. See \S \ref{subsubsec:security} for more detail on this.

\subsubsection{Theta functions} \label{sssection: thetas} For two supersingular elliptic curves $E_1, E_2$ we have a rank $4$ quadratic module $(\Hom(E_1, E_2), \deg)$ that can be identified with the quadratic module $(I, \frac{1}{\Nm(I)} \Nm(\cdot))$ for a suitable ideal $I$ in $B_{p, \infty}$. Since all maximal orders are everywhere locally conjugate and all maximal orders in a quaternion algebra over $\QQ_p$ have class number one, one concludes (as we have already mentioned) that the quadratic modules $(\Hom(E_1, E_2), \deg)$ have discriminant $p^2$ and level $p$, and are all in the same genus. Further, any integral quadratic form in this genus is realized as $(\Hom(E_1, E_2), \deg)$ for some $E_1, E_2$. 

We define the {\bf theta function} of this quadratic module by
\[ \Theta_{E_1, E_2}(q) := \sum_{n=0}^\infty \sharp \{f\in \Hom(E_1, E_2): \deg(f) = n\} \cdot q^n = \sum_{n=0}^\infty r_{E_1, E_2}(n)\cdot q^n. \]
If $E_1$ and $E_2$ correspond to ideal classes $I_i$ and $I_j$ respectively under the Deuring correspondence, then $r_{E_1,E_2}(n)$ is nothing other than $2w_j B(n)_{ij}$ (this holds for all $n\geq 0$ because of the way we chose to define $B(0)$). By the general theory of theta functions associated to quadratic forms, the theta function $\Theta_{E_1, E_2}$ evaluated at $q=e^{2\pi i z}$ is a weight $2$ modular form on $X_0(p)$. A version of Eichler's basis problem (which is a theorem in this case) asserts that the set of weight~$2$ modular forms 
\[\{\Theta_{E_1, E_2}: E_1, E_2 \;\text{\rm supersingular elliptic curves}\}\]
spans $M_2(\Gamma_0(p))$. Note that this is a set of roughly $\binom{p/12}{2}$ elements, while the dimension of $M_2(\Gamma_0(p))$ is about $p/12$. Thus, there must be many linear relations between these theta functions. There are some obvious relations that hold for all supersingular $E_1,E_2$:
\begin{itemize}
\item $\Theta_{E_1, E_2} = \Theta_{E_2, E_1}$, because $f\mapsto f^\vee$ is an isometry $\Hom(E_1,E_2)\to \Hom(E_2,E_1)$.
\item $\Theta_{E_1, E_2} = \Theta_{E_1^{(p)}, E_2^{(p)}}$, because $f\mapsto f^{(p)}$ is an isometry $\Hom(E_1,E_2)\to \Hom(E_1^{(p)},E_2^{(p)})$. 
\end{itemize}
Aside from these relations and those that can be derived from them, there is no reason to expect any of the linear relations to take the very simple form of an equality between two theta functions. However, such coincidences do occur (in fact, already for $p=67$); cf. \cite{GL} and references therein. 

Fix a supersingular elliptic curve $E_1$ and let $E_1, \dots, E_h$ be representatives for the isomorphism classes of supersingular elliptic curves. Then letting 
\[2w_j:=\sharp \Aut(E_j),\] we have
\[ \Theta(q) := \sum_{j=1}^h \frac{1}{2w_j} \Theta_{E_1, E_j}(q) = \scrE_2(q), \]
where $\scrE_2$ is the unique weight $2$ Eisenstein series on $\Gamma_0(p)$ with leading coefficient $\frac{p-1}{24}$ appearing in \S \ref{subsec:connected and expansion} and for which we have an explicit formula. 
This result is a special case of the Siegel-Weil formula, as we explain in \S\ref{subsec: eisenstein} below.

One also knows the action of Hecke operators on the theta functions. Simplifying the notation by writing $\Theta_{ij} = \Theta_{E_i, E_j}$, \cite[Proposition 5.5]{Gross} or \cite{Eichler2} prove that 
\begin{equation}\label{eqn: Tm theta} T_m \Theta_{ij} = \sum_k B(m)_{ik} \Theta_{kj} = \sum_k B(m)_{kj}\Theta_{ik}.\end{equation}

\subsubsection{Functions on the vertices of $\scrG(p, \ell)$ and modular forms}\label{sssection: vertex functions} The second proof that the supersingular graphs are connected (\S\ref{subsec:connected and expansion}) used the fact that $B(\ell)$, the adjacency matrix of $\scrG(p, \ell)$, also represents the action of the Hecke operator $T_\ell$ in its action on $M_2(\Gamma_0(p))$ in \textit{some} basis. Thus, there is an isomorphism of the Hecke algebra (with $\CC$-coefficients, say) of $X_0(p)$ with the algebra generated over $\CC$ by the Brandt matrices $\{B(n): n \geq 0\}$, such that for $n$ prime to $p$, $B(n)$ corresponds to the Hecke operator $T_n$, and $B(p)$ corresponds to the Atkin-Lehner involution. We will denote both algebras by $\TT$.

There is a compatible isomorphism between the $\TT$-module we call ${\rm Fun}(p, \ell)$, which is simply the complex vector space of $\CC$-valued functions on the set of supersingular $j$-invariants under the action of the Brandt matrices, and the module $M_2(\Gamma_0(p))$ with the action of the Hecke operators. We can already find a connection between these two modules in the construction of the theta functions (\ref{eqn: Tm theta}), but to provide an explicit isomorphism one needs the following construction:

Enumerate the supersingular $j$-invariants as $j_1, \dots, j_h$ (in an arbitrary fashion) and let $\delta_i$ be the delta function at $j_i$, namely, taking value $1$ on $j_i$ and value $0$ at any other $j_n$. One defines a Hermitian pairing on ${\rm Fun}(p, \ell)$, by defining
\[ \langle \delta_i, \delta_j \rangle = w_i\cdot  \delta_{ij}\]
(here $\delta_{ij}$ is Kronecker's delta, equal to $1$ if $i = j$ and $0$ otherwise). 
The dual basis $\delta_1^\vee, \dots, \delta_h^\vee$ is simply $\delta_i^\vee = \frac{1}{w_i} \delta_i$. If $u = \sum \alpha_i \delta_i\in {\rm Fun}(p, \ell)$, let $\deg(u) = \sum {\alpha_i}$; if $v = \sum \beta_i \delta_i^\vee\in {\rm Fun}(p, \ell)^\vee$, let $\deg v = \sum \frac{\beta_i }{w_i}$.
One lets $\TT$ act on ${\rm Fun}(p, \ell)$ by the transpose action so that that
\[ B(m)\delta_i = \sum_{j=1}^h B(m)_{ij} \delta_j.\]
Then the isomorphism of Hecke modules
\[ {\rm Fun}(p, \ell)\otimes_\TT {\rm Fun}(p, \ell)^\vee \arr M_2(\Gamma_0(p))\]
is given by \cite[Proposition 5.6]{Gross}
\[ \delta_i \otimes \delta_j^\vee \mapsto\frac{1}{2w_j} + \sum_{m\geq 1} \langle B(m)\delta_i, \delta_j^\vee \rangle q^m = \frac{1}{2w_j}  + \sum_{m\geq 1} B(m)_{ij}q^m =\frac{1}{2w_j}\Theta_{E_i, E_j} .\]
More generally, for $u \in {\rm Fun}(p, \ell), v \in {\rm Fun}(p, \ell)^\vee$, 
\[ u \otimes v \mapsto\frac{\deg u \cdot \deg v}{2} + \sum_{m\geq 1} \langle B(m)u, v \rangle q^m.\]

\section{Applications to cryptography} 

\id There are several interesting applications of supersingular elliptic curves to cryptography. It was recently shown \cite{CD, MM} that one of the most prominent applications (SIKE) is not secure, at least under the original implementation (see however the recent preprint \cite{Dartois} that suggests that minor adjustments to the implementation will not be able to prevent the attack). Others, though very promising, are still under close scrutiny.  That said, such applications instigated further research into the arithmetic of elliptic curves and even abelian varieties, and created much interest in the cryptographic community. This, in our view, is ample justification for further study and for reviewing some of these primitives.

\subsection{Cryptographic hash functions} Let $S$ be the set of all sequences of $0, 1$ of any finite length. Thus, $S = \{0,1\}^\ast:= \cup_{n=0}^\infty \{0, 1\}^n$. Fix a positive integer $N$. A {\bf hash function} valued in $\{0, 1\}^N$ is a function
\[ h\colon S \arr \{0, 1\}^N.\]
One can think about it as associating to any given datum $s\in S$ a certain ``trace'', or ``footprint'',~$h(s)$ of~$s$. Of course, as $S$ is infinite, such a function cannot be one-to-one and evidently there will be infinitely many $s\in S$ with the same value $h(s)$. More than that, for any $n>N$ there will be two data $s_1, s_2$ of length $n$ with $h(s_1) = h(s_2)$.

Cryptographic hash functions have to satisfy several restrictions, the most essential of which is that for a given $n\gg N$ and $s_1 \in \{0, 1\}^n$ it should be computationally infeasible to find a different $s_2 \in  \{0, 1\}^n$ such that $h(s_1) = h(s_2)$. Another property is that the number of preimages $h^{-1}(v) \cap \{0, 1\}^n$ should converge quickly to $2^n/2^N$ (that is, uniform) as $n\arr \infty$. 

As a ``pseudo-example'' of an application of such functions, one can imagine the following scenario: Alice sends to Bob a large file $s\in S$, and Bob wants to verify that he received the file sent by Alice, and not a file that had been purposely manipulated by a third party. One option is for Alice to send the file again through a different communication channel (assuming that it is unlikely that a single third party will be able to compromise both channels), and for Bob to compare the files; unfortunately this uses a lot of extra time and bandwidth if $s$ is very large. A more efficient option is for Alice and Bob to agree on a hash function $h$ in advance (that need not be secret), and then for Alice to send the file $s$ in one channel and its hash value $h(s)$ in another (perhaps this channel is much more secure but cannot handle as much data). Then Bob can apply the hash function $h$ to the received file, and check that it agrees with the hash value sent by Alice. Since it is computationally infeasible to generate a file $s^\prime$ of the same length with the same hash value $h(s) = h(s^\prime)$, Alice and Bob can be reasonably confident that the file Bob received is the one sent by Alice.

\subsubsection{CGL hash functions}\label{subsec:CGL} The CGL hash function proposed in \cite{CGL} is based on using data $s\in S$ as instructions for traversing the supersingular graph $\scrG(p, \ell)$. The $j$-invariant that is the final vertex of this walk is the hash value; it is an element of $\FF_{p^2}$. We explain the procedure in some detail for the case $\ell = 2$, as in the literature this precise description is almost always omitted. To simplify further, choose $p\equiv 1 \pmod{12}$ so that every supersingular elliptic curve has only $\pm 1$ as automorphisms. We have in mind that $p$ is of cryptographic size, say $p \approx 2^{500}$,  and the file has size at least $10 \log(p)$, to err on the side of caution.

For each supersingular $j$-invariant $j\in\FF_{p^2}$ we fix a model $E(j)$ over $\FF_{p^2}$ satisfying $\Fr^2 = [-p]$ (\S\ref{subsec: model over Fp2}). Note that for any such curve $E=E(j)$ the $2$-torsion points $E[2](\fpbar)$ are all defined over~$\FF_{p^2}$: indeed, $\Fr^2$ acts on them by $-p \equiv 1 \pmod{2}$. 

We fix once and for all an ordering on $\FF_{p^2}$: for instance, we may fix an isomorphism $\FF_{p^2} \cong \ZZ/p\ZZ \oplus \ZZ/p\ZZ$ as groups, take on $\ZZ/p\ZZ$ the ordering $\bar 0<\bar1< \dots \overline{p-1}$, and take on $\FF_{p^2}$ the lexicographic order. Fix an initial supersingular $j$-invariant $j_0$ and its corresponding model $E_0:=E(j_0)$. Fix also a $2$-torsion point $P_0 = (x, 0) \in E_0(\FF_{p^2})$. 
 
Given a string $s=s_1\cdots s_n$ of $0$'s and $1$'s, we define a walk on $\scrG(p, 2)$ as follows. Suppose we have reached $j_i$ for some $i=0,\ldots,n-1$, and assume we have a $2$-torsion point $P_i\in E_i(\FF_{p^2})$ for $E_i=E(j_i)$. We label the other $2$-torsion points of $E_i$ as $Q_0=(x_0,0)$ and $Q_1=(x_1,0)$, enumerated so that $x_0 < x_1$ according to the fixed order on $\FF_{p^2}$. 
The bit $s_i$ is either $0$ or $1$ and hence determines a $2$-torsion point $Q_{s_i}\in \{ Q_0, Q_1\}$. Using V\'elu's formulas \cite{Velu}, we can explicitly compute an elliptic curve $E':= E_i/\langle Q_{s_i}\rangle$, together with a map $E_i\arr E'$. Let $j_{i+1}$ be the $j$-invariant of $E'$, so there is an isomorphism $E'\arr E_{i+1}:=E(j_{i+1})$ that is uniquely determined up to $\pm 1$. Since $\pm1$ conveniently acts trivially on $2$-torsion, the map $E_i[2]\to E_{i+1}[2]$ is uniquely determined. The image of $E_i[2]$ in $E_{i+1}$ (or equivalently, the kernel of the dual isogeny $E_{i+1}\arr E_i$) is generated by some $P_{i+1}\in E_{i+1}[2]$. Thus we obtain a curve $E_{i+1}$ and a $2$-torsion point $P_{i+1}$ on it, and we may continue the process.

After $n$ steps, having used $s_1, \dots, s_n$, we arrive at a pair $(E_n, P_n)$. The hash function value is $j(E_n) \in \FF_{p^2}$; there are approximately $\frac{p-1}{12}$ possible values for the hash function. The Ramanujan property of $\scrG(p,\ell)$ ensures that the pre-image sizes converge to a uniform distribution extremely rapidly, and for large graphs of small degree one can hardly do better in this aspect. 

\subsubsection{Security of CGL hash functions}\label{subsubsec:security}
Not every Ramanujan graph will do for the creation of cryptographic hash functions and indeed some other suggestions made in \cite{CGL} were quickly discovered to be insecure \cite{PLQ, TZ}. 
At this point in time, it is understood that the difficulty of breaking the CGL hash function described above rests on the difficulty of the following problem:

\medskip

\id \textit{Key computational problem:\newline
Given a supersingular elliptic curve $E$, calculate $\End(E)$. More precisely, given a Weierstrass equation for a supersingular elliptic curve $E$, produce a list of elements of $B_{p,\infty}$ that generate a maximal order isomorphic to $\End(E)$.}

\medskip As we have already remarked, there is an efficient algorithm to verify that the given curve $E$ is indeed supersingular \cite{SutherlandIdentify}, and this can be done without computing any non-integer endomorphisms of $E$. The problem above is key because if one can perform such a calculation efficiently, then one can reduce the problem of finding a path from $E_0$ to $E$ in $\scrG(p, \ell)$ -- and thereby generating another input hashing to the same value -- to a computation solely in the quaternion algebra~$B_{p, \infty}$. Indeed, the problem reduces to the following: given two maximal orders, $\calO_0\cong\End(E_0)$ and $\calO\cong \End(E)$, find an ideal $I$ such that 
\[ \calO_r(I) = \calO_0, \quad \calO_\ell(I) \cong \calO;\]
in addition, for any two-sided ideal $J$ of $\calO_\ell(I)$, find an element $\alpha\in JI$ with scaled norm $\frac{1}{\Nm(JI)}\Nm(\alpha)$ equal to a power of $\ell$. Indeed, as $J$ varies over two-sided ideals of $\calO_\ell(I)$, $JI$ attains all (one or two) right ideal classes of $\calO_0$ that correspond to curves $E'$ with $\End(E')\cong\calO$ under the Deuring correspondence; in particular, the curve $E$ is attained for some $J$. The element $\alpha$ then corresponds to an isogeny $E_0\to E$ of degree a power of $\ell$, which can be used to construct a path from $E_0$ to $E$ in $\scrG(p,\ell)$. This quaternion algebra variant of the problem can be solved efficiently; see \cite{EHLMP, KLPT, LP, PW, Wesolowski} for more detail.

Of course, one cannot say with absolute confidence whether a cryptographic protocol is secure until is has been broken. At best, one can hope for a proof showing that protocol is as secure as a computational problem that has withstood assaults for a longer time. The fact that the quaternionic algebra problem can be solved efficiently could be viewed as a sign that the CGL hash functions, or other protocols based on supersingular isogeny graphs, are not secure. However, consider the following analogy with the discrete logarithm problem, which is to solve $g^n=h$ for given $g,h\in \FF_q^\times$. Choosing a generator $\gamma$ of $\FF_q^\times$, one has an easy isomorphism $f\colon \ZZ/(q-1)\ZZ \arr \FF_q^\times, a\mapsto \gamma^a$, and solving $nf^{-1}(g)=f^{-1}(h)$ in $\ZZ/(q-1)\ZZ$ is very easy if $f^{-1}(g)$ and $f^{-1}(h)$ are known. Thus the hardness of the discrete log problem depends on the assumption that the inverse function $f^{-1}\colon\FF_q^\times \arr \ZZ/(q-1)$ is hard to compute.

By analogy, choosing a supersingular elliptic curve $E$ with known $\calO:= \End(E)$ is akin to choosing a primitive root $\gamma$, and this is not hard in practice. As already discussed, there is a bijection $f\colon{\rm Cl}(\calO) \arr \{ \text{supersingular elliptic curves}\}/\hspace{-3pt}\cong$, based on the Deuring correspondence, from right ideal classes of $\calO$ to supersingular elliptic curves; this bijection is easy to compute, and allows one to effectively navigate $\scrG(p, \ell)$. As in the case of the discrete log problem, the difficulty of computing $f^{-1}$ is the basis the security of cryptographic applications of $\scrG(p, \ell)$. 

\medskip

\id
In \cite{BCNEMP} one finds an elegant criterion as to when two closed paths based in $E$ in $\scrG(p, \ell)$ generate a subring of $\End(E)$ of finite index. First, note that a closed path $\gamma$ gives, by composing the isogenies corresponding to its edges, an endomorphism $f$ of $E$. To be precise, $f$ is not uniquely defined by $\gamma$ (since edges correspond to isogenies only up to automorphisms of the image), but under the hypothesis $p\equiv 1\pmod{12}$, $f$ is well-defined up to a sign. The degree of $f$ is $\ell^{{\rm length}(\gamma)}$. The criterion says (under the simplifying assumption $p \equiv 1 \pmod{12}$) that if two closed paths $\gamma, \gamma_1$ at $E$, neither with any backtracking, pass through different sets of vertices -- meaning one of the paths passes through a vertex through which the other does not -- then the two corresponding endomorphisms generate a subring of $\End(E)$ of finite index. 

It has been used as a heuristic in analyzing the security of the hash functions (e.g. in \cite[\S 3.3]{EHLMP}) that one can in fact generate the whole ring $\End(E)$ by using additional closed paths. This follows from Theorem~\ref{thm: generation of orders} that we present in \S\ref{subsec: Generation by isogenies} below.

\subsection{SIDH: A Diffie-Hellman protocol}\label{subsec: SIDH} De Feo, Jao and Pl\^ut suggested a protocol called {\bf SIDH} (short for ``Supersingular Isogeny Diffie-Hellman'') for establishing a secret key between two parties using supersingular elliptic curves and their isogenies \cite{dFJP}. In a sense, one can view their protocol as ``completing squares'' in the superposition of two supersingular graphs $\scrG(p, \ell_A)$ and $\scrG(p, \ell_B)$. The protocol proceeds roughly as follows: two distinct primes $\ell_A, \ell_B$ and powers $r_A, r_B$ are chosen in advance. (Additional conditions on these parameters are posed in \cite{dFJP}, but these are largely immaterial for explaining the main idea. Some of these constraints will become clear over the course of the explanation; for instance we will see that $r_A,r_B$ must be fairly large). A base supersingular elliptic curve $E$ is also chosen. 

The information $E, \ell_A^{r_A}, \ell_B^{r_B}$ is available for all, as well as a basis $P_A, Q_A$ for $E[\ell_A^{r_A}]$ and $P_B, Q_B$ for $E[\ell_B^{r_B}]$ . Alice secretly chooses integers $m_A, n_A$ not divisible by $\ell$, and secretly calculates the subgroup $H_A = \langle m_A P_A + n_A Q_A\rangle$, a subgroup of order $\ell_A^{r_A}$ inside $E[\ell_A^{r_A}]$. Independently of Alice and also in secret, Bob likewise chooses secret integers $m_B, n_B$ and computes a subgroup $H_B = \langle m_B P_B + n_B Q_B\rangle$ that is a subgroup of order $\ell_B^{r_B}$ inside $E[\ell_B^{r_B}]$.

Alice publishes the elliptic curve $E/H_A$ and the images $\tilde P_B, \tilde Q_B$ of the points $P_B, Q_B$ on the curve $E/H_A$. The data $E/H_A$ is given as a Weierstrass equation, say, and so its publication reveals no information about $H_A$ (if $r_A \gg 0$ there will be many $H_A$ leading to the same elliptic curve). Bob does the same and publishes $E/H_B$ and the image $\tilde P_A, \tilde Q_A$ of the points $P_A, Q_A$ on $E/H_B$. In the next step, Alice uses the points $\tilde P_A, \tilde Q_A$ on the curve $E/H_B$ and her secret $m_A, n_A$ to create a subgroup $\tilde H_A = \langle m_A\tilde P_A + n_A \tilde Q_A\rangle$ on $E/H_B$; this is the image of $H_A$ on $E/H_B$ under the isogeny $E \arr E/H_B$. Bob does the same using $\tilde P_B, \tilde Q_B$ on $E/H_A$ to get a subgroup $\tilde H_B = \langle m_B\tilde P_B + n_B \tilde Q_B\rangle$ on $E/H_A$. It is elementary to check that 
\[(E/H_A)/\tilde H_B \cong E/(H_A+ H_B) \cong (E/H_B)/\tilde H_A\]
and the secret Alice and Bob share is the $j$-invariant of the elliptic curve $E/(H_A+ H_B)$. See Figure~\ref{fig:sidh} for a summary.
\begin{figure}
\centering
\[ \xymatrix@C=1.5cm{& E\ar[dl]_{\text{Alice knows how to\qquad}}\ar[dr]^{\text{\qquad Bob knows how to}} & &\text{Public: } E, P_A, Q_A, P_B, Q_B\qquad\qquad\quad\\ E/H_A\ar[dr]_{\text{Bob knows how to\qquad}} &&E/H_B\ar[dl]^{\text{\qquad Alice knows how to}}&\text{Public: }E/H_A, E/H_B,  \tilde P_A, \tilde Q_A, \tilde P_B, \tilde Q_B\\ & E/(H_A + H_B) &&\longleftarrow\text{Shared secret\qquad\qquad\qquad\qquad}}\]
\caption{A summary of the SIDH protocol.}\label{fig:sidh}
\end{figure}
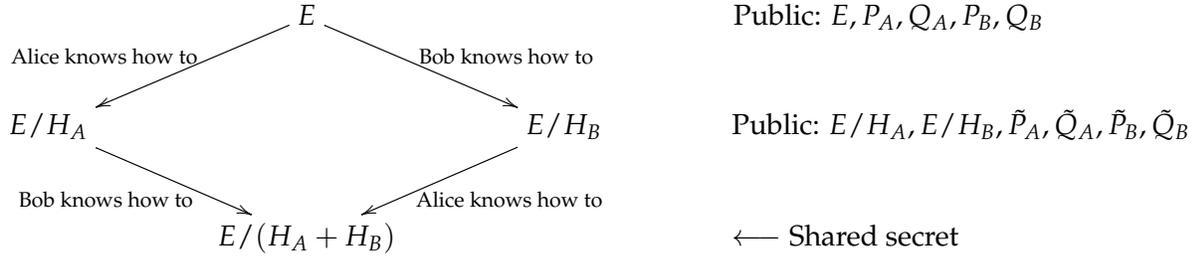
Note that if, for example, $r_A= r_B = 1$ then we are completing a square in the superposition of the graphs $\scrG(p, \ell_A)$ and $\scrG(p, \ell_B)$ such that parallel edges have the same degree. 

A complete implementation of a cryptographic protocol based on this method was submitted under the name {\bf SIKE} (short for ``Supersingular Isogeny Key Encapsulation'') to the NIST Post-Quantum Cryptography Standardization process in 2017 \cite{SIKE}.
Unfortunately, this elegant idea is not cryptographically secure, as shown by Castryck - Decru \cite{CD} and, independently, by Maino - Martindale, see \cite{MM} and references therein: they were able to show that the data of the auxiliary points $P_A,Q_A$ on $E$ together with their images $\tilde P_A,\tilde Q_A$ on $E/H_B$ is enough to efficiently compute an isogeny $E\to E/H_B$. Interestingly, these attacks use in a very strong way what one may call superspecial graphs associated to principally polarized abelian surfaces and $(2, 2)$-isogenies between them. This suggests that the study of higher dimensional analogues of the supersingular isogeny graphs is both arithmetically rich and has relevance to cryptographic protocols based on elliptic curves. Already there are many cryptographic protocols that take advantage of this extra structure. See for example  \S\ref{subsec: SQIsign variants} and \cite{Dartois}.

\medskip

\id There are other Diffie-Hellman protocols based on supersingular isogeny graphs
that do not require the data of auxiliary points and their images under a secret isogeny. This means that so far, they are immune to attacks of the form that cracked SIDH. A few notable examples include {\bf CSIDH}~\cite{CLMPR}, which restricts to supersingular elliptic curves defined over $\FF_p$, and {\bf OSIDH}~\cite{CK}, which enriches each supersingular curve with the data of a marked quadratic order in $\End(E)$. However, the additional structure does lead to some potential vulnerabilities, which have been explored in several works, for instance in \cite{BA, BIJ, BLMP,DdF}. More research is needed to determine whether these algorithms are viable for applications.

\subsection{SQIsign: a digital signature protocol} \label{subsec: SQIsign}  SQIsign (short for ``Short Quaternion and Isogeny Signature'') is another cryptographic protocol based on supersingular isogeny graphs. It was introduced by de Feo, Kohel, Leroux, Petit and Wesolowski \cite{dFKLPW}. To explain the idea we draw on the landmark papers \cite{DH, GMR}. The main challenge is to send messages together with  additional information -- the signature -- that authenticates the sender. To achieve this, the sender uses a one-way function $f$; that is, an invertible function $f$ that is easy to compute and hard to invert (at least without additional secret information). Given a message $M$, the signing party publishes the function $f$ and sends the message $M$ together with a signature $f^{-1}(M)$. The receiver, and anyone else, can authenticate the signature by verifying that $f(f^{-1}(M)) = M$. No one else can feasibly calculate $f^{-1}(M)$, and so the signature $f^{-1}(M)$ must be from the party that knows the secret inverse to $f$, namely the original sender. It is technically possible to ``imitate'' the original sender by coming up with a signature $S$ first and declaring the message to be $N:=f(S)$, but such a message $N$ is likely to be gibberish; in practice, no one else can feasibly come up with both a ``reasonable'' message $N$ together with a corresponding signature $f^{-1}(N)$, and thus it is not feasible to imitate the sender.

\medskip

\id The security of  SQIsign  relies on the hardness of the following computation problem:

\medskip

 \textit{Given a cyclic isogeny $f\colon E\to E^\prime$, find any other cyclic isogeny $g\colon E\to E^\prime$ with $\Ker(f) \neq \Ker(g)$.} 

\medskip
\id (If one specifies that both isogenies have degree a power of $\ell$, this is exactly the security assumption for the CGL hash functions.)

The key idea of SQIsign is the following. Suppose one is given an explicit cyclic isogeny from a curve $E_A$ to a curve $E_1$. By the hardness assumption, if all one knows is this one isogeny, it is believed to be extremely difficult to find any other cyclic isogeny from $E_A$ to $E_1$. However, if a party knows the endomorphism ring of $E_A$, then using quaternion methods they can produce a rich collection of isogenies from $E_A$ to $E_1$. So if Alice wants to sign a message, she will do so by proving that she has access to this rich collection of isogeny paths, and can create distinct cyclic isogenies between the two curves.

\medskip

The SQIsign protocol progresses roughly as follows; see \cite{dFKLPW} for details. First some public information is agreed on: a prime $p$ and a supersingular elliptic curve $E_0$ over $\FF_p$ with known endomorphism ring $\calO_0 = \End(E_0)$. If $p \equiv 3 \pmod{4}$, one may take the elliptic curve $E: y^2 = x^3 +x$ of $j$-invariant $1728$, whose endomorphism ring can be calculated to be $\ZZ[i, \frac{i+j}{2}, \frac{1+k}{2}]$, where $i^2 = -1, j^2 = -p$ as in \S\ref{subsubsection: Bp again}. 
To see this, on the one hand, we can calculate the discriminant of this order to check that it is maximal. On the other hand, it visibly contains $\ZZ[i]$. But $\ZZ[i]$ has class number $1$, and so by Deuring's lifting theorem there is a unique maximal order of $B_{p, \infty}$, up to isomorphism, that contains $\ZZ[i]$. Since $\End(E)$ clearly contains $\ZZ[i]$, it follows that the specified maximal order is isomorphic to $\End(E)$. 

In addition, there is a public algorithm that takes in a starting elliptic curve $E$ and an integer~$M$ and produces a cyclic isogeny from $E$ by ``following the directions given by $M$.'' An example of such an algorithm would be to convert $M$ to a sequence of $0, 1$ bits and apply the CGL hash function starting from $E$ (\S\ref{subsec:CGL}), but in general this algorithm is allowed to (and, in fact, should) use isogenies of multiple small prime degrees. We will see at the end of the section that for security reasons the directions should not depend on $M$ alone, but for simplicity we ignore this issue for now.

Before Alice signs any messages, she computes a random (secret) isogeny $\tau\colon E_0\to E_A$ to some $E_A$, and publishes $E_A$. The curve $E_A$ will be used to compute the verification function $f$, but since it is infeasible for anyone else to compute an isogeny from $E_0$ to $E_A$, $\tau$ will be used to compute the signature function $f^{-1}$.

Now suppose Alice wants to sign a message $M$. She first computes a secret random isogeny~$\psi$, originating at $E_0$, and ending at some $E_1$, that at the time being is kept secret. Alice then starts from $E_1$ and follows the directions given by $M$ to compute a cyclic isogeny $\varphi\colon E_1\to E_2$. Alice now has the complete information of an isogeny $\varphi\circ\psi\circ\tau^\vee\colon E_A\to E_2$, but in fact she knows more: since the endomorphism ring of $E_0$ is known, she can use $\tau$ to compute the endomorphism ring of $E_A$, as well as a right ideal of $\End(E_A)$ corresponding to $E_2$ under the Deuring correspondence. Thus she can use quaternion computations to find in this ideal any kind of isogeny $\sigma\colon E_A\to E_2$ she might want; for instance, she can ensure that $\varphi^\vee\circ\sigma\colon E_A\to E_1$ is cyclic. She then publishes the pair $(E_1,\sigma)$ as the signature on $M$. Note that $\sigma$, and any other isogeny, are published as paths, meaning as a sequence of cyclic isogenies of low degree. This is forced on all such protocols, because there is no other feasible way to write down an isogeny of large degree.

Once Alice's signature has been published, anyone can reproduce $E_2$ and $\varphi$ as these depend only on $E_1$ and $M$ that are public knowledge. To authenticate the signature Bob verifies that $\varphi^\vee\circ\sigma$ is cyclic, which he can do, as Bob, and any other party, knows $\varphi\colon E_1 \arr E_2$ and so also its dual $\varphi^\vee\colon E_2 \arr E_1$. The protocol is summarized in Figure~\ref{fig:sqisign}.
\begin{figure}
\centering
\[ \xymatrix@C=1.5cm{& E_0\ar[dl]_{\text{(Alice's secret) }\tau\;\;\;}\ar[dr]^{\;\;\;\psi\text{ (Alice's secret)}} & & \\ E_A\ar[dr]_{\text{(Alice's signature) }\sigma\;\;\;} & &E_1\ar[dl]^{\;\;\;\;\varphi(E_1,M) \text{ (public algorithm)}} &  \\ 
& E_2 & &} \]
\caption{A summary of the SQIsign protocol. All four curves are public; only $E_0$ has public endomorphism ring.}\label{fig:sqisign}
\end{figure}
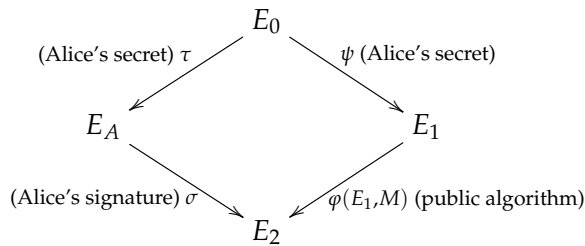

Note that once Alice has published her signature $(E_1, \sigma)$, if an adversary intercepting the communication could find an isogeny from $E_0$ to \textit{any} of the curves $E_1,E_2,E_A$, then they could compute the endomorphism rings of \textit{all} the curves involved, revealing Alice's secret and enabling them to forge her signature! This is why it is crucial that Alice starts by producing a curve $E_1$ that is very far from $E_0$. It is also crucial that Alice is able to produce multiple isogenies from $E_A\to E_2$ that proceed in very different directions: if Alice just publishes $\varphi\circ\psi\circ\tau^\vee$ for instance, then anyone could follow this path partway to find an isogeny $E_A\to E_0$, again revealing her secret.  The key observation is that quaternion methods allow one to choose an alternate path of isogenies $E_A\to E_2$ that completely avoids $E_0$, as well as any of the curves along the path $\varphi$ from $E_1$ to $E_2$.

If a fraudster wanted to imitate Alice, they could produce a curve $E_1$ however they wanted, and follow the directions given by a (fake) message $N$ of their choice to reach a curve $E_2$. But however they try, there will be difficulties in producing a valid signature.
\begin{itemize}
	\item If they choose $E_1$ to have known endomorphism ring (for instance via an isogeny from $E_0$), then it will likely be impossible to find any isogeny connecting $E_A$ to $E_2$.
	\item The fraudster could let $E_1$ be the endpoint of some isogeny $\mu$ from $E_A$; they would then be able to compute an isogeny $\sigma\colon E_A\to E_2$, namely the composition $\varphi\circ\mu$. But in this case $\varphi^\vee\circ\sigma=[\deg\varphi]\circ \mu$ would be very far from cyclic. It is infeasible in this case to find an isogeny $\sigma$ for which $\varphi^\vee\circ\sigma$ is cyclic, as none of the endomorphism rings of the curves involved in this isogeny are known.
	\item The fraudster could start with an isogeny $\sigma\colon E_A\to E_2$, then backtrack through the algorithm to find the curve $E_1$ which, after following the instructions given by $N$, yields $E_2$; if they could do this, then they could successfully replicate the signature. To avoid this scenario, we must ensure in the initial setup that the algorithm can't be backsolved. For instance, instead of following directions given by $M$, the user might first be asked to plug $E_1$ and $M$ into a hash function $h$, and follow the directions given by $h(E_1,M)$ instead. By ensuring that both $M$ and $E_1$ are required before one can even write down the directions, it becomes impossible to start with a desired endpoint $E_2$ and backtrack.
\end{itemize} 

\subsubsection{NIST submission and variants}\label{subsec: SQIsign variants}
In the past few years there has been a flurry of renewed interest around the SQIsign protocol, owing to a confluence of two factors. On one hand, while the initial NIST Post-Quantum Cryptography Standardization process did not recommend any isogeny-based systems \cite{NIST1}, a new call for digital signature schemes was made in 2023 \cite{NIST2} to which SQIsign was submitted. Thus SQIsign has now replaced SIKE as being the only isogeny-based algorithm that is actively being considered for widespread adoption. 

On the other hand, it was discovered that the methods used in \cite{CD} to break the SIDH protocol (\S\ref{subsec: SIDH}) could be used as a new way to generate the isogenies needed for a signature scheme, leading to implementations that are faster, more scaleable, use less memory, and have stronger security analyses. The first of these was {\bf SQISignHD} (HD for ``Higher Dimension'') \cite{DLRW}, but similar variants continue to be developed on a regular basis.

\subsection{Insecure subsets of $\scrG(p, \ell)$} 

Assume $p$ is very large (say $p\approx 2^{500}$). For all but a vanishingly small proportion of supersingular elliptic curves $E$ over $\FF_{p^2}$, the smallest non-integer element of $\End(E)$ will have extremely large norm. To wit, the number of supersingular curves with a non-integer endomorphism of degree at most $M$ grows like $O(M^\frac{3}{2})$; for example, of the roughly $\frac{p}{12}$ supersingular elliptic curves, only around $\sqrt{p}$ of them will have a non-integer endomorphism with degree less than $\sqrt[3]{p}$. One implication of this observation is that for the vast majority of supersingular curves over $\FF_{p^2}$, one can pick any relatively small $\ell$ (say $\ell\leq 100$) and the smallest non-backtracking cycle in $\scrG(p,\ell)$ containing $E$ will have length on the order of $\log p$ (as such a cycle produces a non-integer endomorphism of any elliptic curve appearing as a vertex on the cycle). This means that a breadth-first search algorithm starting at this vertex will not be able to distinguish $\scrG(p,\ell)$ from an $(\ell+1)$-regular tree in polynomial time. This near-homogeneity provides a certain amount of security: for a typical $E$, it is very difficult to use the local structure of any of the graphs $\scrG(p,\ell)$ in a neighborhood of $E$ to obtain any useful information.

In his thesis \cite{Love} (see also \cite{LB}), the second author studied the structure of the set of elliptic curves for which this near-homogeneity is not guaranteed. Fix $M< \sqrt{p}/2$, considered as very small compared to $p$. Let $SS^M(p)$ denote the supersingular $j$-invariants $j\in\fpbar$ with the property that if $E(j)$ is a corresponding elliptic curve, then $E(j)$ has an endomorphism $\alpha \not \in \ZZ$ such that $\Nm(\alpha) \leq M$. As said, the cardinality of $SS^M(p)$ has order of magnitude $O(M^{\frac{3}{2}})\ll p$.

Given $j\in SS^M(p)$, the $\alpha\in \End(E(j))\setminus\ZZ$ of small norm is almost unique. Namely, if there is another $\alpha^\prime \in \End(E(j))$ with norm bounded by $M$, then $\alpha^\prime \in \QQ(\alpha)$ (otherwise, one calculates that $\ZZ\oplus \ZZ\alpha \oplus \ZZ\alpha^\prime \oplus \ZZ\alpha\alpha^\prime$ is a subring of $\End(E)$ with discriminant smaller than $p$). By abuse of notation we shall refer to any such $\alpha$ as $\alpha_j$ and say it is a {\bf witness} that $j \in SS^M(p)$. Thus, for any $j\in SS^M(p)$ there is a well-defined quadratic field $\QQ(\alpha_j)$, where $\alpha_j$ is a witness that $j\in SS^M(p)$. The discriminant $D$ of such a quadratic field satisfies 
\[ -4M < D < 0.\]
We may thus partition the set $SS^M(p)$ into disjoint subsets called {\bf clusters}:
\[ SS^M(p) = \bigcup_D T_D, \qquad j \in T_D \Leftrightarrow \QQ(\alpha_j) \text{ has discriminant } D.\]
(To simplify the exposition, we allow some $T_D$ to be empty.) 
Setting $b:=\frac{2}{\sqrt{3}}\sqrt{M}$, we can define a graph $\scrG'(p, b)$ whose vertices are Galois orbits of supersingular $j$-invariants and whose edges correspond to Galois orbits of isogenies $f\colon E_1 \arr E_2$ of any prime degree $\ell\leq b$, as usual considered up to $\Aut(E_2)$. One may think of $\scrG'(p, b)$ as obtained by taking the quotient of $\scrG(p, \ell)$ by the Frobenius involution for all primes $\ell\leq b$, and then gluing together the corresponding vertices of all these $\scrG(p,\ell)$. Two elliptic curves are close to each other in $\scrG'(p,b)$ if it is possible to reach one from the other using isogenies of sufficiently small degree as well as the Frobenius isogeny; there are many more isogenies available than in $\scrG(p,\ell)$, which only uses $\ell$-isogenies for a single prime $\ell$.

The set $SS^M(p)$ forms a subgraph of $\scrG'(p,b)$. If $M$ is small enough,  $M < \sqrt[3]{3p/16}$, then one has the following \cite[Theorem 13.]{LB}: (1) The non-empty clusters are exactly the connected components of the full subgraph on $SS^M(p)$; (2) the shortest path in $\scrG'(p,b)$ between two vertices in distinct clusters has length greater than $\frac12\log_b(p)-\log_b(2M)$.
Otherwise said, for any two vertices in a single cluster there exists a path of short isogenies between them that stays within the cluster, but for two vertices in distinct clusters, there is no path between them that stays within $SS^M(p)$.  See Figure~\ref{fig:SM} for an illustration of the graph $\scrG^\prime(p, b)$ with the clusters in $SS^M(p)$ labeled.

\

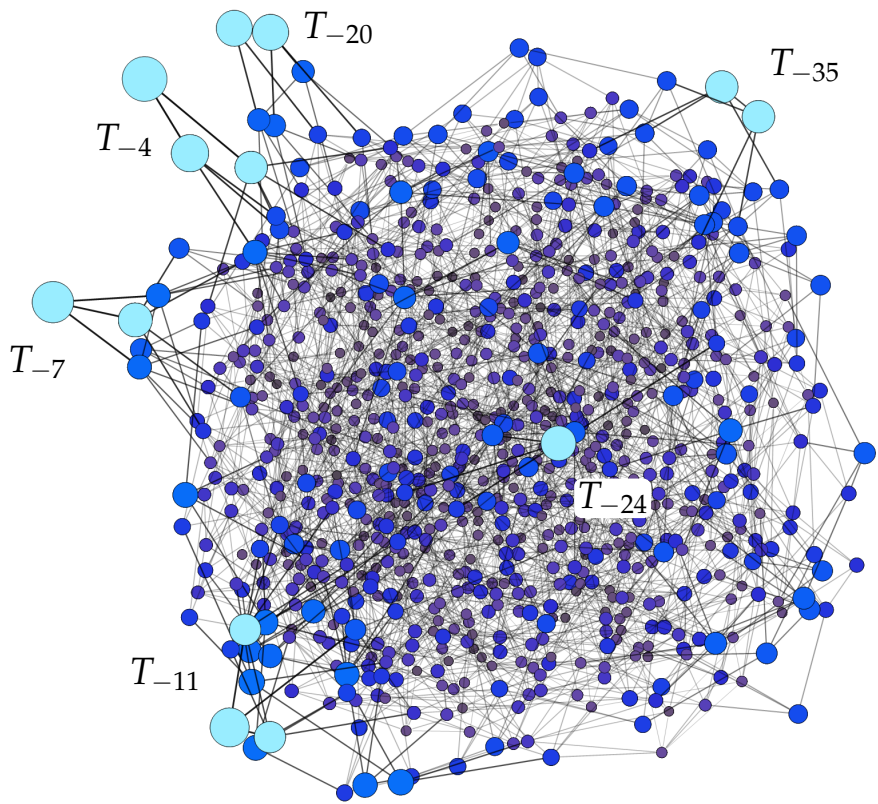
\begin{figure}[h]
	\centering
	\begin{tikzpicture}
		\node (background) at (0,0) {\includegraphics[width=324pt]{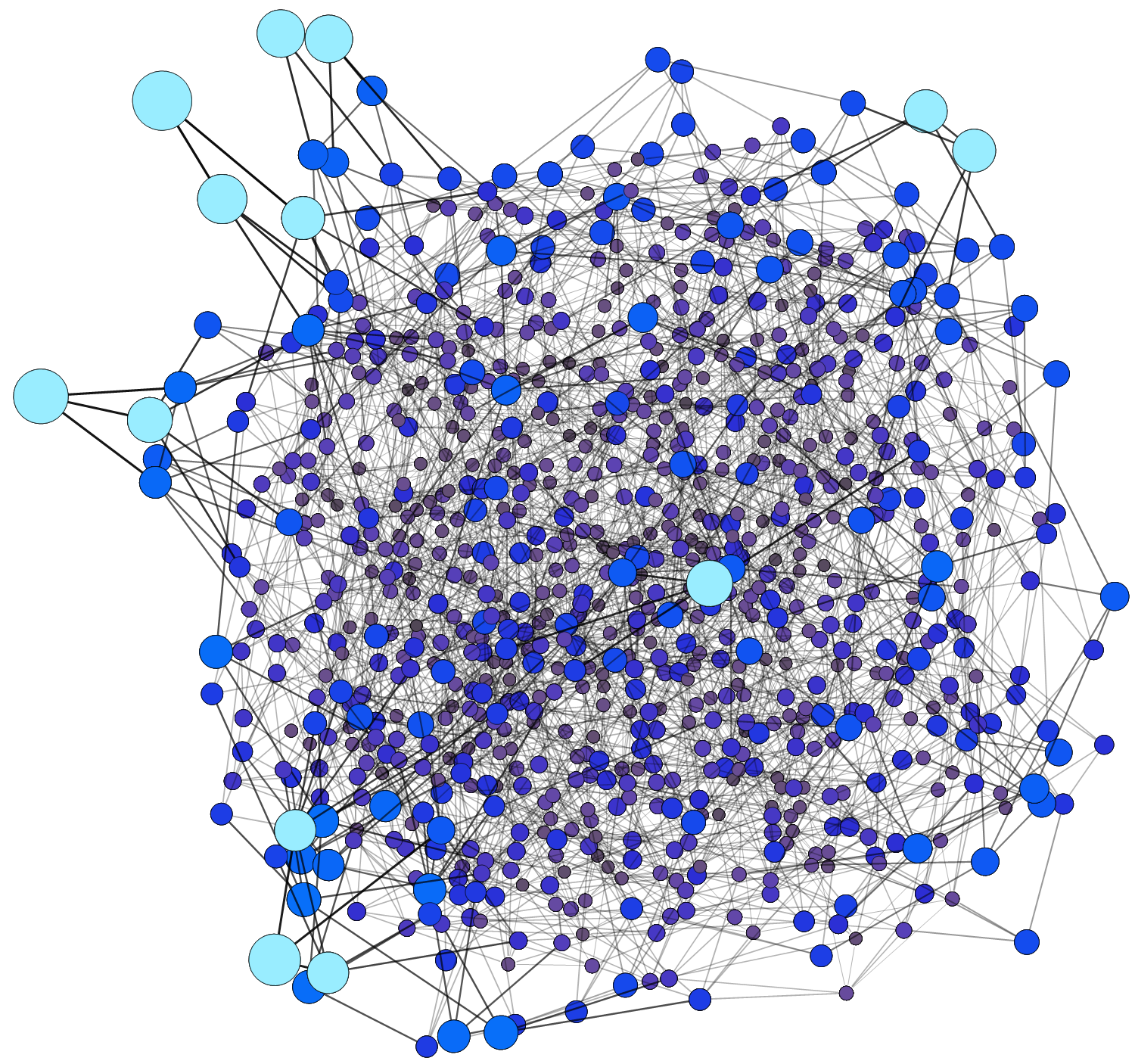}};
		\node at (-4.35,3.5) {\Large $T_{-4}$};
		\node at (-5.5,0.55) {\Large $T_{-7}$};
		\node at (-3.8,-3.6) {\Large $T_{-11}$};
		\node [fill=white, fill opacity=1, rounded corners=3pt,inner sep=1pt] at (2.15,-1.25) {\color{white}\Large $T_{-24}$};
		\node at (2.15,-1.25) {\Large $T_{-24}$};
		\node at (4.7,4.5) {\Large $T_{-35}$};
		\node at (-1.5,5) {\Large $T_{-20}$};
	\end{tikzpicture}
	\caption{\small{The graph $\scrG'(20011,\frac{2}{\sqrt{3}}(12))$. Vertices are supersingular curves in characteristic $p=20011$ with conjugate pairs $\{E,E^{(p)}\}$ identified; two curves $E,E'$ are connected by an edge if there is an isogeny $E\to E'$ or $E\to (E')^{(p)}$ of degree $2$ or $3$. Brighter, larger vertices correspond to curves with smaller non-integer endomorphisms; the light blue vertices are those in $SS^{12}(20011)$, and each cluster is labelled.  Data computed using Magma, plotted using Mathematica. A version of this image first appeared in \cite{LB}.}}
	\label{fig:SM}
\end{figure}

Let $\ell$ be a prime, $\ell \leq M < \sqrt[3]{3p/16}$. While the clusters of $SS^M(p)$ form connected subgraphs of $\scrG'(p,b)$, the path between two curves in a common cluster may require isogenies of different prime degrees as well as a Frobenius map. For this reason, the cluster may not form a connected subgraph of $\scrG(p,\ell)$ in general. Each cluster in $SS^M(p)$ is a juxtaposition of certain very small ``patches'' from within the graphs $\scrG(p, \ell)$, $\ell\leq b$.

\medskip
\id There are algorithms to compute the endomorphism rings of curves in $\text{SS}^M(p)$ with a runtime that is polynomial in $M$ and $\log p$, and so, despite the large distance between clusters, it is always possible to find paths in $\scrG(p,\ell)$ between any two curves in $\text{SS}^M(p)$. In light of the previous sections, this means that $\text{SS}^M(p)$ constitutes a portion of $\scrG(p,\ell)$ that is particularly insecure for cryptographic applications. For instance, \textit{the CGL hash function can be easily broken if the starting and ending curves both happen to lie in $\text{SS}^M(p)$. Likewise, the security of the SQIsign protocol depends on $E_1,E_2,E_A$ all avoiding $SS^M(p)$.}

At present, the only known method to generate a supersingular curve in large characteristic~$p$ in a reasonable amount of time is to start with an elliptic curve with CM and reduce it mod $p$, as in \S\ref{sec:CM}; this reduction will necessarily land in $\text{SS}^M(p)$ for a value of $M$ proportional to the discriminant of the CM order, and is therefore in the ``insecure part'' of the isogeny graph, as one must take $M$ to be quite small in order to be able to perform the computation. One may then take such a curve and follow random walks in the isogeny graphs $\scrG(p,\ell)$ for various values of~$\ell$ to approximate choosing a supersingular curve uniformly at random; however, the party that computes this path can use this information to compute the endomorphism ring of the output curve. This means that there is currently no known way to produce a supersingular curve in a feasible amount of time without at least one party knowing the endomorphism ring of this curve (see also \cite{Boo}). This is a serious issue for some applications, such as \cite{dFMPS}, which due to this limitation can only be implemented using ``trusted setup''. Other applications, including Diffie-Hellman protocols, avoid this issue by having multiple users each produce curves independently, ensuring that no individual user has complete information.

\section{Some new results concerning $\scrG(p, \ell)$} We report here on three new results concerning the subjects of this article. Since the details have or will appear in print elsewhere, we only provide explanation, context and comments regarding the proofs, referring the interested reader to the articles themselves. 
\subsection{Rigidity} \label{subsec: rigiditiy}From the point of view of cryptographic applications, any extra structure the graphs $\scrG(p, \ell)$ may have could be viewed as potential vulnerability of their cryptographic applications. Thus, one would be happy to know that they have very few automorphisms. 

Let $\Aut(\scrG(p, \ell))$ denote the {\bf automorphism group} of the graph. Such an automorphism is a pair $(f_V, f_E)$, where $f_V$ is a bijection on the set on the set of vertices and $f_E$ is a bijection on the set of oriented edges such that if $v$ is the origin (respectively, end) of an edge $e$ then $f_V(v)$ is the origin (respectively, end) of the edge $f_E(e)$. Among these automorphisms there are the trivial ones $A^{\text{triv}}(\scrG(p, \ell))$, which are the automorphisms $(f_V, f_E) = (\Id, f_E)$. This is a very small and easy to describe normal subgroup of $\Aut(\scrG(p, \ell))$ that simply permutes the loops at a vertex, or the multiple edges between two vertices. Since the number of such loops and multiple edges is bounded in terms of $\ell$ alone and is completely understood in terms of embeddings of quadratic orders into maximal orders of $B_{p, \infty}$, the structure of $A^{\text{triv}}(\scrG(p, \ell))$ is also grosso modo well-understood. We let the {\bf essential automorphism group} be
\[ \Aut^{\text{ess}}(\scrG(p, \ell)) = \Aut(\scrG(p, \ell))/A^{\text{triv}}(\scrG(p, \ell)).\]
In the following theorem, one assumes $p>71$ simply to guarantee that the automorphism $\Fr$ of $\scrG(p, \ell)$ is not an element of $A^{\text{triv}}(\scrG(p, \ell))$. Furthermore, the assumption made in the theorem is now known to hold for density $1$ of primes $\ell$ for a given $p$ \cite{KSW}. There are cases where this assumption is not satisfied and also $\Aut^{\text{ess}}(\scrG(p, \ell))$ is bigger. Sam Mayo proved the following theorem during his undergraduate summer program with the first author of this article \cite{Mayo}.
\begin{thm}\label{thm: Mayo} Let $p>71$ and assume that $T_\ell$ generates the Hecke algebra $\TT_\CC$ of $X_0(p)$ over the complex numbers. Then 
\[  \Aut^{\text{\rm ess}}(\scrG(p, \ell)) = \{ Id, \Fr\}.\]
\end{thm}
For the proof and references, see loc. cit. The first move is to note that an automorphism $\varphi$ of the graph commutes with the adjacency matrix of $\scrG(p, \ell)$, thus is a linear operator on the functions $\text{\rm Fun}(p, \ell)$ that commutes with $T_\ell$, hence with $\TT$. However, by a result of Emerton, that implies that $\varphi \in \TT$ (in fact, over $\ZZ$). As $\TT_\QQ$ is product of totally real fields, that already gives that $\varphi$ belongs to an elementary $2$ group of size at most $2^g$, where $g$ is the genus of $X_0(p)$. Further study puts further conditions on $\varphi$ and in particular shows that it arises from an automorphism of $X_0(p)$ itself. Those have been well-studied, and one concludes that the only non-trivial option is for $\varphi$ to be induced by the Atkin-Lehner involution.

\subsection{Generation by $\ell$-isogenies}\label{subsec: Generation by isogenies} In analyzing the security of hash functions, or other cryptographic protocols, associated with $\scrG(p, \ell)$, sometimes one requires the assumption that for a supersingular curve $E$, the ring $\End(E)$ is generated as a $\ZZ$-module (equivalently, as a ring) by homomorphisms $f$ of degree $\{ \ell^k: k \geq 0\}$; for example, in \cite[\S 3.3]{EHLMP}. It turns out that this is but a very special case of a general result concerning the generation of lattices of rank at least~$4$ by elements of norms in a specified set. In \S\ref{section6}, we generalize this result to lattices over totally real fields. 

Let $Q:\ZZ^n\to \QQ$ be a quadratic form; that is, the function $(x, y) := \frac12(Q(x+y) - Q(x) - Q(y))$ is a bilinear form. We say that $Q$ is {\bf integral} if $Q(x)\in\ZZ$ for all $x\in\ZZ^n$ (which implies that $(x,y)\in\frac 12\ZZ$ for $x,y\in\ZZ^n$), that $Q$ is {\bf nondegenerate} if for all $x\in\ZZ^n$ there exists $y\in\ZZ^n$ with $(x,y)\neq 0$, and that $Q$ is {\bf primitive} if the ideal of $\ZZ$ generated by the values $Q(x), x \in \ZZ^n$ is $\ZZ$ itself. In \cite[Corollary 2.3]{GL}, we proved the following theorem.

\begin{thm}\label{thm: special generators for a lattice} Let $Q$ be an integral, non-degenerate, primitive quadratic form in $n\geq 4$ variables. Let $S$ be an infinite set of positive integers. Assume that for all $s\in S$ and all primes~$r$, there exists a basis of $\ZZ_r^n$ comprised of elements $x$ with norm $Q(x)=s$. Then $\ZZ^n$ is generated by elements $x$ with norm $Q(x) \in S$.
\end{thm}

\id One can apply this result to other lattices $\Lambda\neq \ZZ^n$ by picking a basis, defining an isomorphism $\Lambda\simeq \ZZ^n$. The proof given in \cite{GL} uses a strong approximation theorem of N. Sardari that allows finding integral points of sufficiently large norm satisfying local congruence conditions.  
The conditions required to apply Sardari's theorem hold by the assumptions in our theorem, but in the application to maximal orders of $B_{p, \infty}$ they reduce to a very easy verification about existence of bases of maximal orders $B_{p, \infty}\otimes \QQ_r$ (for \textit{all} primes $r$, including $r=p$) with given norms $s\in S$, that are easy to check given we have an explicit and simple description of the local maximal orders. We thus deduce:

\begin{thm} \label{thm: generation of orders} Let $\ell \neq p$ be a prime. Every maximal order $\calO$ of $B_{p, \infty}$ is generated by elements of norm in the set $\{\ell^k: k = 0, 1, 2, 3, \dots\}$.
\end{thm}

In \S\ref{section6}, we would prefer to formulate our approach slightly differently. In truth, in both approaches the deus ex machina is the same, but a change of perspective will be enlightening and useful. 

\subsubsection{An application to even unimodular lattices}
While somewhat tangential to our main goal of studying supersingular isogenies, the authors would like to use this opportunity to share another intriguing and charming corollary of Theorem~\ref{thm: special generators for a lattice}.

\begin{cor}\label{cor: even unimodular}
	Let $S$ be any infinite set of positive even integers, and $\Lambda$ an even unimodular lattice. Then $\Lambda$ is generated by elements with norm in $S$.
\end{cor}
\begin{proof}
	All even unimodular lattices are in the same genus as $E_8^n$ for some $n$; see \cite{Omeara}, especially \S 106 A. So by Theorem~\ref{thm: special generators for a lattice}, it suffices to check that for any $s\in S$ and every prime $p$, there exists a basis of the localization of $E_8$ at $p$ consisting of elements of norm $s$.
	
	Locally at $2$, the quadratic form associated to $E_8$ is equivalent to $2x_1x_2+2x_3x_4+2x_5x_6+2x_7x_8$ (loc. cit.). Writing $s=2t$, the matrix
	\[\left(\begin{smallmatrix}
		1&0&0&0&0&0&1&t\\
		0&1&0&0&0&0&1&t\\
		0&0&1&0&0&0&1&t\\
		0&0&0&1&0&0&1&t\\
		0&0&0&0&1&0&1&t\\
		0&0&0&0&0&1&1&t\\
		0&0&0&0&0&0&1&t\\
		1&t&0&0&0&0&0&1
	\end{smallmatrix}\right)\]
	has determinant $1$, and so the row vectors form a basis of vectors of norm $s$ for this lattice.
	
	Locally at $p\geq 3$, the quadratic form is equivalent to $x_1^2+\cdots +x_8^2$ (loc. cit.). There exist $p$-adic integers $x,y,z$ solving $x^2+y^2+z^2=s-1$ and such that $z$ is a $p$-adic unit: explicitly, we can take $x,y\in\ZZ$ satisfying $x^2+y^2\equiv s-2\pmod p$ (possible by the pigeonhole principle because the sets $\{s-2-y^2:y\in\FF_p\}$ and $\{x^2:x\in\FF_p\}$ both have $\frac{p+1}{2}$ elements), so that $z_0=1$ solves $x^2+y^2+z_0^2\equiv s-1\pmod p$, and we obtain $z\in\ZZ_p^\times$ with $x^2+y^2+z^2=s-1$ by Hensel lifting. The matrix
	\[\left(\begin{smallmatrix}
		1&0&0&0&0&x&y&z\\
		-1&0&0&0&0&x&y&z\\
		0&1&0&0&0&x&y&z\\
		0&0&1&0&0&x&y&z\\
		0&0&0&1&0&x&y&z\\
		0&0&0&0&1&x&y&z\\
		x&y&z&0&0&1&0&0\\
		x&y&z&0&0&0&1&0
	\end{smallmatrix}\right)\]
	has determinant $2z\in\ZZ_p^\times$ and so the row vectors form a basis of vectors of norm $s$ for this lattice.	
\end{proof}

\subsection{Identifying an elliptic curve from its endomorphism ring} Let $E$ be a supersingular elliptic curve. One can ask whether $E$ is determined up to isomorphism by $\End(E)$. The answer to that goes back more than a century and was already mentioned above: two supersingular elliptic curves with $j$-invariants $j$ and $j^\prime$ have isomorphic endomorphism rings if and only if $j^\prime \in \{ j, j^p\}$. Thus, one may ask further if there is a way to generate the elliptic curve given its maximal order. If one knows a single base supersingular elliptic curve $E_0$ as a Weierstrass equation and also a description of its maximal order, then the problem can be solved by finding a path in the isogeny graph $\scrG(p, \ell)$ by working first completely at the level of quaternion algebras. This path is then transformed to a path in $\scrG(p, \ell)$, labelled with elliptic curves now, that starts with $E_0$ and ends in $E$. V\'elu's formula can be used to derive a Weierstrass equation for $E$. There is thus a solution to this problem that is polynomial time in $\log(p)$; see \cite{EHLMP, KLPT}.

For many primes we can find such a base elliptic curve $E_0$ quite easily. We have already given an example in \S\ref{subsec: SQIsign} when $p\equiv 3\pmod{4}$. 

\medskip

\id
We can consider another invariant for a supersingular elliptic curve $E$ which is its theta function as in \S\ref{sssection: thetas}
\[ \Theta_{E}(q) = \sum_{n=0}^\infty \sharp \{f\in \End_{\fpbar}(E): \deg(f) = n\} \cdot q^n = \sum_{n=0}^\infty r_{E}(n)\cdot q^n. \]
It is natural to ask whether $\Theta_E$ determines $E$ up to Frobenius transform. While the theta function is not a fine enough invariant to be able to distinguish between rank $4$ lattices in general \cite{ConwaySloane}, the authors showed that it is strong enough to distinguish between endomorphism rings of elliptic curves \cite{GL}.

\begin{thm}\label{thm: theta determines E} Let $E, E^\prime$ be supersingular elliptic curves with $j$ invariants $j, j^\prime$, respectively. Then 
\[ \Theta_E = \Theta_{E^\prime}\;\; \Leftrightarrow \;\; j^\prime \in \{j, j^p\}.\]
\end{thm}

This raises the question whether one can construct the order $\End(E)$ from $\Theta_E$. The answer to that is yes, and we discuss that after discussing the proof Theorem~\ref{thm: theta determines E}. 

\subsubsection{Discussion of the proof of Theorem~\ref{thm: theta determines E}} 
The proof of Theorem~\ref{thm: theta determines E} is hard to conceptualize, and that is one of the difficulties in extending it. Here are some of the main ideas used in the proof.
\begin{enumerate}
\item Throughout the proof there is an interplay between a maximal order $\calO$ of $B_{p, \infty}$ and its {\bf Gross lattice} $\calO^T$ used much in \cite{Gross}. This is a ternary lattice which is the image of $\calO$ under the map $x \mapsto 2x - \Tr(x)$. Its importance is that it classifies embeddings of quadratic imaginary orders into $\calO$; this construction can easily be generalized to quaternion algebras $B$ over number fields $L$ and the classification of embeddings of monogenic orders containing $\calO_L$ in quadratic extensions of $L$ into maximal orders of $B$. 
	\item \textit{Peeling off the first layer:} the power series $\Theta_{E}(q)-\Theta_{\ZZ}(q)$ counts the number of non-integer elements in $\End(E)$ of each norm; thus if the first coefficient is $cq^n$, we know there is at least one non-integer element $\alpha$ of norm $n$. Since there are no elements of smaller norm, we know the trace of $\alpha$ must be $0$ or $\pm1$ (otherwise we could subtract an integer to find an element of smaller norm); i.e. we have either $\alpha=\sqrt{-n}$ or $\alpha=\frac12(\pm 1+\sqrt{1-4n})$ in the corresponding quadratic field. Minkowski's bounds tell us that~$n$ is small in comparison to $p$. 
	\item A key lemma, the original proof of which we can find in the literature is due to Kaneko \cite{Kaneko}, but it has been used by various other authors independently, e.g. \cite{GLa}, gives constraints not only on the norms of elements of a quaternion order generating distinct imaginary quadratic orders, but as proven in \cite{GL} also on  the \textit{angles} between such elements. 	
	This constraint is especially strong when both elements are relatively small compared to the discriminant of the quaternion order. 
	\item Combining  Minkowski's bounds with  Kaneko's lemma, we find that the elements of norm $n$ cannot lie in distinct-but-isomorphic quadratic orders. Thus we can learn the full minimal polynomials of the elements of norm $n$ from the value of $c$ alone (since $\ZZ[\sqrt{-n}]$ has two elements of norm $n$ and $\ZZ[\frac12(1+\sqrt{1-4n})]$ has four), giving us the full lattice structure of a rank $2$ sub lattice $\ZZ[\alpha]$ of $\calO$.
	\item \textit{Peeling off the second layer:} Repeat the above steps but for $\Theta_{E}(q)-\Theta_{\ZZ[\alpha]}(q)$ to find the next independent short element $\beta$. Kaneko's lemma  can again be used to identify the minimal polynomial of $\beta$ as well as its angle with $\alpha$, giving us now the full lattice structure of a rank $3$ sub lattice $\ZZ +\alpha\ZZ+\beta\ZZ$ of $\calO$.
	\item \textit{Peeling off the third layer:} Unfortunately, the Minkowski bounds on the fourth successive minimum are not strong enough for Kaneko's lemma to apply directly anymore; that is, it is possible for there to be multiple shortest elements of $\End(E) - (\ZZ +\alpha\ZZ+\beta\ZZ)$ that satisfy the same minimal polynomial. At this point a more careful analysis of the lattice geometry is needed: we prove that for each possible minimal polynomial of the last successive minimum, only certain initial sequences of coefficients of $\Theta_E(q)-\Theta_{\ZZ +\alpha\ZZ+\beta\ZZ}(q)$ are possible, and these possibilities do not overlap. Thus the theta function allows us to determine the minimal polynomial of a fourth independent vector~$\gamma$. 
	\item Once again, Kaneko's lemma allows us to determine the angles between $\alpha,\beta,\gamma$, and from this also the angles between their images in the Gross lattice. This determines the isometry type of the Gross lattice, which uniquely determines the quaternion order up to isomorphism by a result of Chevyrev and Galbraith \cite{CG}. 
\end{enumerate}

\subsubsection{Determining the order from its theta function} Let $\Theta$ be a weight $2$ modular form on $\Gamma_0(p)$ that we know to be equal to $\Theta_E$ for some supersingular elliptic curve $E$. Let also $\calO = \End(E)$. We can also say that $\Theta$ is associated to the maximal order $\calO$ of $B_{p, \infty}$. The question we address here is how to determine $\calO$. 

The proof of Theorem~\ref{thm: theta determines E} proceeds, as we have discussed, by determining from the theta function $\Theta$ the successive minima of the Gross lattice $\calO^T$, and the angles between the vectors realizing them. Thus, one is able to reconstruct the Gross lattice $\calO^T$ from $\Theta$. One can then find the maximal order $\calO$ by the taking the even integral Clifford algebra of $(\calO^T)^\vee$, the dual lattice of $\calO^T$. See \cite[\S\S 22.1 - 22.3]{Voight}. 

\subsubsection{Determining the elliptic curve from the statistics of closed paths in $\scrG(p, \ell)$} Consider a supersingular elliptic curve $E$ as a vertex of the supersingular graph $\scrG(p, \ell)$ and consider the statistics of closed paths in $\scrG(p, \ell)$, based at $E$, of all lengths $k$. Those are precisely the diagonal entries $B(\ell^k)_{11}$ (where in the enumeration of the supersingular elliptic curves $E$ is declared the first one). Those are also the coefficients $\{r_E(\ell^k): k = 0, 1, 2, 3 \dots \}$ in $\Theta_E$. We claim that if $T_\ell$ generates the Hecke algebra $\TT$ (which, as noted above, occurs for a density $1$ set of primes $\ell$) then this information alone is sufficient to determine $E$ up to Frobenius base change. 

\begin{cor} Assume that $T_\ell$ generates the Hecke algebra $\TT$ of $X_0(p)$ and $E, E^\prime$ are supersingular elliptic curves such that in their theta functions the coefficients of the powers $q^{\ell^k}, k = 0, 1, 2, 3...$ agree. Then, $\Theta_E = \Theta_{E^\prime}$ and so $E^\prime \cong E$ or $E^\prime \cong E^{(p)}$ over $\fpbar$. 
\end{cor}

In fact, a similar statement holds even if $\TT$ is not generated by a single Hecke operator. 
In general, $\TT$ is generated by some finite set of Hecke operators, say $T_{\ell_1},\ldots,T_{\ell_s}$. Using an argument like the one below, one can show that if the coefficients of the powers $q^{\ell_i^k}$ agree for all $i=1,\ldots,s$ and $k\geq 0$, then $\Theta_E = \Theta_{E^\prime}$. But for simplicity we only consider the case $s=1$ below. 

By considering the form $F=\Theta_E-\Theta_{E'}$, which is a cusp form since the Eisenstein space of $M_2(\Gamma_0(p))$ is one-dimensional, it suffices to prove the following under the same assumption that $T_\ell$ generates $\TT$: given a cusp form $F\in S_2(\Gamma_0(p))$, if the coefficient in $F$ of $q^{\ell^k}$ vanishes for all $k=0,1,2,3,\ldots$, then $F=0$. 

Let $f_1,\ldots,f_n$ be a basis of normalized eigenforms in $S_2(\Gamma_0(p))$ (here $n$ is one less than the class number of $B_{p,\infty}$), and write
\[F=\sum_{r=1}^n c_rf_r,\qquad T_\ell(f_r)=\lambda_rf_r.\]
If $\lambda_{a} = \lambda_b$ then, since $T_\ell$ generates the Hecke algebra, we have $Tf_a = Tf_b$ for every Hecke operator $T\in \TT$. But by multiplicity one for $M_2(\Gamma_0(p))$, any two distinct eigenforms must have distinct eigenvalues for some Hecke operator. Thus we can conclude that the eigenvalues $\lambda_1,\ldots,\lambda_n$ of $T_\ell$ are distinct.

Now write a Fourier expansion for each cusp form, $f_r=\sum a_r(i)q^i$. By assumption on the coefficients of $q^{\ell^k}$ in $F$, we have 
\[\sum_{r=1}^n c_ra_r(\ell^k)=0\]
for all $k$. On the other hand, we can compute the coefficients $a_r(\ell^k)$ explicitly using the relations
\begin{align*}
	a_r(\ell)&=\lambda_r,\\
	a_r(\ell^k)&=a_r(\ell)a_r(\ell^{k-1})-\ell^{k-1}a_r(\ell^{k-2})
\end{align*}
that hold for any normalized eigenform. This allows us to write each $a_r(\ell^k)$ as a monic integer degree $k$ polynomial in $\lambda_r$, for instance
\begin{align*}
	a_r(\ell^2)&=\lambda_r^2-\ell,\\
	a_r(\ell^3)&=\lambda_r^3-(\ell+\ell^2)\lambda_r,
\end{align*}
and so on. Thus, using the fact that $\sum c_r a_r(\ell^k)=0$, by induction on $k$ we find that
\[\sum_{r=1}^n c_r\lambda_r^k=0\]
for all $k\geq 0$. Since $\lambda_1,\ldots,\lambda_n$ are distinct, this implies that $c_1=c_2=\cdots=c_n=0$ by a Vandermonde determinant calculation: that is, we have
\[\begin{psmallmatrix}
	1 & 1 & 1 & \cdots & 1 \\
	\lambda_1 & \lambda_2 & \lambda_3 & \cdots & \lambda_n \\
	\lambda_1^2 & \lambda_2^2 & \lambda_3^2 & \cdots & \lambda_n^2 \\
	 & \vdots &  &  & \vdots \\
	 \lambda_1^{n-1} & \lambda_2^{n-1} & \lambda_3^{n-1} & \cdots & \lambda_n^{n-1} \\
\end{psmallmatrix}\begin{psmallmatrix}
	c_1 \\ \phantom{\lambda_1}\hspace{-9pt}c_2 \\ \phantom{\lambda_1^n}\hspace{-9pt}c_3 \\ \vdots \\ \phantom{\lambda_1^n}\hspace{-9pt}c_n
\end{psmallmatrix}=0,\]
where the square matrix has determinant $\prod_{1\leq i<j\leq n}(\lambda_j-\lambda_i)\neq 0$.
In fact, this shows that we only need to know the $q^{\ell^k}$ coefficient vanishes for $k=0,\ldots, n-1$ in order to conclude $F=0$.

\section{Generation of totally real lattices by special elements}\label{section6} Our purpose in this section is to generalize Theorem~\ref{thm: special generators for a lattice}  to lattices over totally real fields. Besides its intrinsic appeal, 
this result also has implications for a generalization of
supersingular isogeny graphs $\scrG(p, \ell)$ that we define below. This generalization uses Hilbert modular varieties and should have interesting applications to cryptography.

Much more background is assumed of the reader in this section, as it would be unrealistic to introduce from scratch the rather high-powered theories from arithmetic geometry and automorphic forms that we need. 

A natural, but different, extension of the theory is to superspecial principally polarized abelian varieties -- see  \cite{ATY, CDS, CostelloSmith,  FT, Jordan, JZ} and references therein. In fact, the superspecial isogeny graphs coming from Hilbert modular varieties have implications to the study of those coming from Siegel modular varieties (see \S \ref{sec: ssg for tot real} below).

\subsection{A brutally short introduction to abelian varieties with real multiplication}\label{sec: intro RM}
The books \cite{GorenBook, vdG} and the foundational papers \cite{DP, Rapoport} provide general background on abelian varieties with real multiplication 
and their moduli spaces.

To simplify the exposition, and to rely on results already in the literature, we put ourselves in the following setting. Let $L$ denote a totally real field of degree $g$ over $\QQ$, and let $\calO_L$ be the ring of integers of $L$. We assume that $L$ has strict class number $1$; that is, all ideals are principal, and there exist units in $\calO_L$ of all possible sign combinations. We denote by $\gerd_L$ the different ideal, $\delta_L$ a totally positive generator of $\gerd_L$, and $d_L:=N_{L/\QQ}(\delta_L)$ the discriminant of $L/\QQ$. We also choose a rational prime $\ell$ that is inert in~$L$, so that $\ol/\ell\ol \cong \FF_{\ell^g}$. From the point of view of algebraic number theory the assumptions on the class number and decomposition of $\ell$ are rather severe restrictions, but much of what we discuss below
can be generalized to arbitrary class numbers and arbitrary rational primes $\ell$ that are, say, unramified in $L$. We also fix a rational prime $p$, $p \neq \ell$, that is unramified in $L$. Using this data, we will construct analogues $\scrG_L(p, \ell)$ of the supersingular graphs $\scrG(p, \ell)$. 

\

\id Let $A$ be a $g$-dimensional abelian variety over a field. Suppose that $A$ comes equipped with an action of $\ol$ -- that is, with an injective ring homomorphism $\iota\colon \ol \arr \End(A)$ -- and a principal $\ol$-linear polarization $\lambda:A\to A^\vee$. Here $\ol$-linear means that if we equip the dual abelian variety $A^\vee$ with the dual $\ol$-action $\iota^\vee\colon \ol \arr \End(A^\vee)$, where $\iota^\vee(\alpha):=(\iota(\alpha))^\vee$, then we require $\lambda \circ \iota = \iota^\vee \circ \lambda$. 
With this setup, the triple  $\uA = (A, \iota, \lambda)$ is called an {abelian variety with real multiplication by $\ol$}, which we abbreviate as an {\bf AV with RM} (the field $L$ being understood from the context).  

There is a coarse moduli scheme $\calM_L$ for such objects. It is a quasi-projective scheme, flat of relative dimension $g$ over $\Spec(\ZZ[d_L^{-1}])$. The scheme has quotient singularities that can be resolved by introducing any rigid level-$N$ structure, for $N\geq 3$ and $(N, d_L) = 1$, at which point the moduli scheme becomes smooth over $\Spec(\ZZ[\zeta_N][(Nd_L)^{-1}])$. There is a natural isomorphism
\[ \calM_L(\CC) \cong \SL_2(\ol) \backslash \mathfrak{H}^g, \]
where $\mathfrak{H} = \{ z: \Ima(z)>0\}$ is the complex upper half space and $\SL_2(\ol)$ acts diagonally as M\"obius transformations, using the $g$ different embeddings $\sigma_i\colon L \arr \RR$. 

If $A$ is complex abelian variety with RM and $A \cong \CC^g/\Lambda$, then $\Lambda$ is a torsion-free $\ol$-module of rank $2$, and is hence projective. So for some $\ol$-ideals $\gera, \gerb$ we have $\Lambda \cong \gera \oplus \gerb$ as an $\ol$-module, and since we assume $L$ has strict class number $1$ this implies $\Lambda \cong \mathcal{O}_L^2$. However, the realization of $\Lambda$ as a lattice in $\CC^g$ depends on $A$ (much as is the case for elliptic curves). This is the first step in parametrizing complex abelian varieties with RM by $(\tau_1, \dots, \tau_g)\in \gerH^g$ (loc. cit. and \cite{BiLa}).

\medskip

\id
The mod $p$ fiber of $\calM_L$, denoted $\overline{\calM}_L$ ($p$ being understood from the context), classifies such triples $\uA$ over fields of characteristic $p$. In the Goren-Oort stratification \cite{GO} of  $\overline{\calM}_L$, which is stable under prime-to-$p$ Hecke operators, the smallest stratum consists of finitely many points called {\bf superspecial points} that correspond to principally polarized superspecial abelian varieties with real multiplication ({\bf ssAV with RM}). 

There are several equivalent definitions of superspecial abelian varieties. For instance, over any algebraically closed field $k$ of characteristic $p$, one can define a $g$-dimensional superspecial abelian variety as an algebraic group isomorphic to $E_1 \times \cdots \times E_g$, where the $\{ E_i\}$ are supersingular elliptic curves. In fact, by a remarkable theorem of Deligne, if $g>1$ then for any fixed  supersingular elliptic curve $E$, $E_1 \times \cdots \times E_g\cong E^g$ \cite{Shioda}. Fix such an elliptic curve $E$ and let $\calO$ be its endomorphism ring, which is a maximal order of $B_{p, \infty}$. Thus, what distinguishes superspecial abelian varieties with RM is compatible data:
\[ \iota\colon \ol \arr M_g(\calO) = \End(E^g), \quad \Xi\in M_g(\calO), \]
where $\Xi$, which encodes the polarization, is a matrix that (1) is symmetric and positive relative to the Rosati involution $(m_{ij}) \mapsto (\overline{m}_{ji})$ induced by the product polarization of $E^g$; (2) commutes with $\iota(\ol)$; (3) is a unit of $M_g(\calO)$. All this taken modulo the action of $\Aut(E^g) = M_g(\calO)^\times$. Note that this is rather explicit data. 

There are finitely many points corresponding to ssAVs  with RM in the moduli space $\overline{\calM}_L$; let $\eta = \eta(L)$ denote their number. Explicit mass formulae exist (see \cite{BG} for $g=2$ and \cite{Yu} in general), but a rough estimate is 
\begin{equation}\label{eqn: 1341}
\eta \approx 2^{1-g}p^g \vert \zeta_L(-1)\vert,
\end{equation}
where $\zeta_L$ is the Dedekind zeta function of $L$ \cite[Proposition 3.6]{CGL2}. Using the functional equation for $\zeta_L$,
we have  $\zeta_L(-1) = \frac{(-1)^g}{2^g\pi^{2g}} d_L^{3/2} \zeta_L(2)$,  
from which one can conclude that 
\begin{equation}\label{eqn: 1342}
\vert \zeta_L(-1) \vert \approx  \frac{1}{2^g\pi^{2g}}d_L^{3/2}. 
\end{equation} 
So when $g$ is held constant, the number of ssAVs  with RM by $\ol$ in characteristic $p$ grows like $p^g d_L^{3/2}$ 
in the sense that the ratio is bounded above and below by constants that can be explicitly computed.
There is much more to say, especially for $L$ quadratic; see \cite{Zagier} for a masterful exposition. 

\vspace{0.6cm}

\id Equivalent characterizations of $A$ being superspecial are \begin{enumerate}
\item  The degree $p^g$ isogeny $\Ver_{A} := \Fr_{(A^{(1/p)})^\vee}^\vee\colon A \arr A^{(1/p)}$, satisfies  $\Ker(\Ver_A) = \Ker(\Fr_A)$; 
\item There is an embedding $\alpha_p^g \injects A[p]$. 
\end{enumerate}The equivalence of (1) and (2) follows immediately from the theory of finite commutative group schemes, and the equivalence of (2) with the definition of superspecial given above is a theorem due to Oort \cite{Oort}. 

We remark that for $g>1$, these are stronger requirements than just requiring that $A[p]$ has no nontrivial $k$-points. Using these characterizations, it is not hard to prove that if $A$ is superspecial and $f\colon A \arr B$ is an isogeny of $g$-dimensional abelian varieties of degree prime to $p$, then $B$ is superspecial too; this statement may fail if $p \vert \deg(f)$.

Suppose that $k$ is a field of characteristic $p$ and $\uA, \uB$ are ssAV with RM over $k$; we let
\[ \Hom_k(\uA, \uB) := \Hom_{\ol,k}(A, B).\]
Namely, this is the $\ol$-module of homomorphisms $A \arr B$ that commute with the action of~$\ol$ and are defined over $k$. If $k$ is algebraically closed, we typically omit the subscript $k$ in $\Hom_k(\uA, \uB)$.

If $k$ is algebraically closed, it was proven in \cite{Nicole} that $\End(\uA)$ is an order of discriminant ideal $p\ol$ in the quaternion algebra $B_{p, L}:=B_{p, \infty} \otimes_\QQ L$. Note that the discriminant of the quaternion algebra $B_{p, L}$ may strictly contain $p\calO_L$, in fact it may be as large as~$\ol$. Thus, $\End(\uA)$ need not be a maximal order, but it is always an Eichler order. 

\begin{exa}\label{ex:E tensor OL} Let $E$ be a supersingular elliptic curve over $\FF_{p^2}$ with endomorphism ring $\calO$, a maximal order of $B_{p, \infty}$. The functor $E\otimes_\ZZ\ol$ is the functor that takes a characteristic $p$ commutative ring $R$ to 
\[ (E\otimes_\ZZ\ol) (R) : = E(R) \otimes_\ZZ \ol.\]
This functor is representable by $E^g$, where an isomorphism is established by choosing a $\ZZ$-basis of $\ol$, but writing it as a tensor product as above endows it with a canonical action of~$\ol$: $\iota(\alpha)(P\otimes \beta) = P \otimes \alpha\beta$. There is also a compatible $\ol$-linear polarization on $E\otimes_\ZZ\ol$ because we have canonical identifications  \[(E\otimes_\ZZ\ol)^\vee = E^\vee \otimes_\ZZ \gerd_L^{-1} = E\otimes_\ZZ \gerd_L^{-1},\]
and the map $\Id \otimes \delta_L^{-1}$ is the $\ol$-linear principal polarization. We thus obtain the structure of an ssAV with RM, and  one can verify directly that
\[\End\left(\underline{E\otimes_\ZZ\calO}_L\right)=\End_{\ol}(E\otimes_\ZZ\ol) = \calO \otimes_\ZZ \ol\subset B_{p, \infty}\otimes_\QQ L = B_{p, L}.\]
The paper \cite{Conrad} contains useful results about such tensor constructions. 
\end{exa}

\subsection{RM superspecial isogeny graphs}\label{sec: ssg for tot real} Our main reference is \cite{CGL2}, where many of the missing details can be found, and where the construction is given in much greater generality. 

With the conventions of \S\ref{sec: intro RM}, let $j_1, \dots, j_\eta$ ($\eta = \eta(L)$) denote the superspecial points of $\overline{\calM}_L$ and let $\uA(j_1), \dots, \uA(j_\eta)$ be the corresponding ssAVs with RM. In fact, we can choose for each a model over $\FF_{p^2}$ such that $\Fr_{A(j_k)}^2 = [-p]$, and so for any two we have \[\Hom_{\fpbar}(\uA(j_a), \uA(j_b)) = \Hom_{\FF_{p^2}}(\uA(j_a), \uA(j_b)).\] In particular, $\Hom_{\FF_{p^2}}(\uA(j_a), \uA(j_b))$ is a module of rank $4$ over $\ol\subset \End(A(j_a))$. Nicole \cite{Nicole} shows that it is an integral quadratic module as follows:

Write $\uA(j_a) = (A_a, \iota_a, \lambda_a)$ and similarly for $\uA(j_b)$. Given $f\in \Hom_{\fpbar}(\uA(j_a), \uA(j_b))$ consider $\lambda_a^{-1}f^\vee \lambda_b f$. This homorphism belongs to $\End(\uA(j_a))$ and is fixed by the Rosati involution associated to $\lambda_a$ (this is the involution that takes $h\in \End(\uA(j_a))$ to $\lambda_a^{-1}h^\vee \lambda_a$),
because, being polarizations, $\lambda_a$ and $\lambda_b$ are self-dual. The elements of $\End(\uA(j_a))$ fixed by the Rosati involution are precisely $\ol$, so we get a quadratic form
\[ \Hom_{\fpbar}(\uA(j_a), \uA(j_b)) \arr \ol, \qquad f \mapsto \deg_L(f):= \lambda_a^{-1}f^\vee \lambda_b f.\]
As in the case of elliptic curves, all these quadratic modules are in the same genus and so to establish some of their properties one may reduce to the case of a single ssAV with RM $\uA=(A, \iota, \lambda)$ and its endomorphism ring $\End(\uA)$. It is convenient to take $\uA = E\otimes\ol$ and $\End(\uA) = \calO\otimes \ol$ as in Example~\ref{ex:E tensor OL}, and this can be made as explicit as one wants by a judicious choice of $E$. The quadratic form on this rank $4$ $\ol$-module is 
\begin{equation} f \mapsto \lambda^{-1}f^\vee \lambda f = f^\dagger f, 
\end{equation}
where $f^\dagger := \lambda^{-1}f^\vee \lambda$ is the image of $f$ under the Rosati involution $f\mapsto f^\dagger$ associated to $\lambda$. In the model $\calO\otimes \ol$ the quadratic form is just the map $f\otimes \alpha \mapsto \Nm(f) \otimes \alpha^2$. In this form we can see that the map is a totally definite quadratic form, by which we mean that under any embedding $\sigma_i\colon L \arr \RR$ this form is positive-definite as a real-valued form. 

\label{further discussion}Even more concretely, $\calO$ always contains a finite index lattice $\Lambda = \ZZ \oplus \ZZ i \oplus \ZZ j \oplus \ZZ k$, where $i^2 = a$ and $j^2 = b$ are negative integers,  $k = ij = -ji$, and the norm form on this lattice is simply $x^2 - \alpha y^2 - \beta z^2 + \alpha\beta w^2$ for $x, y, z, w \in \ZZ$. Then $\calO \otimes \ol$ contains the $\ol$ lattice $\Lambda\otimes \ol = \ol \oplus \ol i \oplus \ol j \oplus \ol k$ and the norm form is 
\begin{equation}
x^2 - \alpha y^2 - \beta z^2 + \alpha\beta w^2, \quad x, y, z, w \in \ol.
\end{equation}

\

\id We now construct a graph $\scrG_L(p, \ell)$ (an {\bf RM superspecial isogeny graph}) whose vertices $1, \dots, \eta$ correspond to the superspecial points $j_1, \dots, j_\eta$ of $\overline{\calM}_L$ or, equivalently, to the ssAVs with RM $\uA(j_1), \dots \uA(j_\eta)$ (which we relabel more simply as $\uA_1, \dots, \uA_\eta$). Each $\uA_j$ has a preferred model as chosen above. Consider a vertex $\uA = (A, \iota, \lambda) \in\{ \uA_1, \dots, \uA_\eta\} $. Since $\ell$ is inert in $L/\QQ$ we have $\uA[\ell]\cong (\ol/\ell\ol)^2\cong \FF_{\ell^g}^2$, and so every proper nontrivial $\ol$-invariant subgroup scheme~$H$ of $\uA[\ell]$ is isomorphic to $\ol/\ell\ol$. There are $\ell^g +1$ of these, and they are automatically isotropic relative to the polarization $\lambda$. Thus, by Mumford's descent theory, there is a unique $\ol$-structure~$\iota'$ and principal $\ol$-polarization $\mu$ on $A/H$ such that under the isogeny $f\colon A \arr A/H$ we have $f^\ast\mu  = \ell \lambda$. We connect the vertex $\uA$ to all the vertices $\underline{A/H}=(A/H,\iota',\mu)$ constructed this way. Note also that any $\ol$-homomorphism $\uA \arr \uA^\prime$ of degree $\ell^g$ and whose kernel is contained in $\uA[\ell]$ must arises this way. We may thus also think of edges as defined by such homomorphisms, up to automorphisms of the image $\uA^\prime$. This gives us a directed graph $\scrG_L(p, \ell)$ with $\eta$ vertices, all with out-degree $\ell^g+1$. These graphs were explored in \cite{CGL2}; for related previous work see \cite{Livne} and references therein. 

Some important properties that these graphs have in common with supersingular isogeny graphs of elliptic curves are that (1) they are connected, (2) they have very few loops and multiple edges, and (3)  they are Ramanujan \cite{CGL2}. They are thus suitable candidates for generalizing the CGL hash function and other cryptographic applications. It will be interesting to consider which other known properties of the graphs $\scrG(p, \ell)$ generalize to $\scrG_L(p, \ell)$. 

The RM superspecial isogeny graphs $\scrG_L(p, \ell)$ are related to another class of graphs, which have vertices corresponding to principally polarized superspecial abelian varieties of dimension $g$, and edges corresponding to isogenies whose kernel is a maximal isotropic subgroup of the $\ell$-torsion. Call this graph $\scrG_{\rm S}^g(p, \ell)$ (``S'' for Siegel). In general, the graphs $\scrG_{\rm S}^g(p,\ell)$ are {\it not} Ramanujan \cite[Section 9]{JZ}, demonstrating that the presence of RM structure results in a qualitative difference in the properties of the resulting graph. However, when $L$ is a degree $g$ totally real field of strict class number $1$ and $\ell$ is inert in $L/\QQ$, there is a natural map of graphs
\begin{equation}\label{eq:RM to Siegel} \scrG_L(p, \ell) \arr \scrG_{\rm S}^g(p, \ell)
\end{equation}
that takes $\uA = (A, \lambda, \iota)$ to $(A, \lambda)$. The edges $\uA \arr \uA/H$ in $\scrG_L(p, \ell)$ go to edges in $\scrG_{\rm S}^g(p, \ell)$, because any such $H$ is automatically a maximal isotropic subgroup of the $\ell$-torsion. 

We now focus on the case $g=2$ and present a heuristic to suggest that the map (\ref{eq:RM to Siegel}) is in fact surjective for infinitely many $L$. Let $d\equiv 1 \pmod{4}$ be prime and set $L = \QQ(\sqrt{d})$; note that for any such field the class number $h_L$ and strict class number $h_L^+$ are equal (\cite[Theorems 6.19 and 9.3]{Buell}). To have $\ell$ inert in $L$ is the condition that~$d$ is not a square modulo $\ell$ if $\ell>2$, and $d \equiv 5 \pmod{8}$ if $\ell = 2$. A positive proportion of primes~$d$ satisfy these conditions, and by the Gauss conjecture, or the Cohen-Lenstra heuristics, one expects the majority of real quadratic fields to have class number $1$; taking the intersection, one also expects there to be infinitely many real quadratic fields $L$ of strict class number $1$ in which $\ell$ is inert.

Combining (\ref{eqn: 1341}) and (\ref{eqn: 1342}) one sees that the number of vertices in $\scrG_L(p, \ell)$ grows to infinity like  $d_L^{3/2}$; 
on the other hand, the number of vertices of $\scrG^2_{\rm S}(p, \ell)$ remains constant, and so one would expect the map $\scrG_L(p, \ell) \arr \scrG^2_{\rm S}(p, \ell)$ to be surjective on vertices for all $L$ with $d_L$ sufficiently large. This is indeed the case \cite[Theorem 2.4]{GLb}. Thus the security of the graphs $\scrG_L(p, \ell)$ and $\scrG^2_{\rm S}(p, \ell)$ appears to be linked.

\medskip

\id We return to our original setting where the degree $g$ of $L/\QQ$ is arbitrary. In \cite{CGL2} more general versions of the graphs $\scrG_L(p, \ell)$ are defined, associated to any prime ideal $\gerl \subset \ol$ that is prime to $p$. The edges are related to isogenies $f\colon \uA \arr \uB$ of ssAVs with RM for which $\Ker(f)$ and $\ol/\gerl$ are isomorphic over $\fpbar$ as $\ol$-modules. The astute reader may wonder about the existence of a principal polarization on $\uB$; this is one place where we use that the strict class number of $L$ is $1$. This assumption allows one to essentially ignore polarizations entirely, for if $\lambda, \mu$ are two principal $\ol$-polarizations on $(A, \iota)$, then the fact that $\ol$ has strict class number $1$ implies that $\mu = \epsilon^2 \lambda$ for some $\epsilon \in \calO_L^\times$, so multiplication by $\epsilon$ induces an isomorphism  $(A, \iota, \lambda)\cong(A, \iota, \mu)$.

In his thesis \cite{Nicole}, see also \cite{Nicole2}, Nicole has developed a Deuring correspondence for superspecial abelian varieties with RM under the strict class number 1 assumption, though there is little doubt that this assumption can be removed. Fixing one superspecial point, say $E\otimes \ol$, there is a bijection between right ideal classes of the order $\End(E\otimes \ol) = \calO\otimes \ol$ and isomorphism classes of ssAV with RM. Many of the properties of the Deuring correspondence carry over to this new context, and the theory fits perfectly with a theory of Brandt matrices as in \cite{Eichler2}; see \cite{CGL2}. 

\subsubsection{Structure and applications of RM superspecial isogeny graphs} \label{subsec: structure of RM isog graph}

Setting a CGL hash function for such graphs is not immediately obvious; one needs a way to encode the vertices, and to enumerate the edges out of a given vertex in a consistent way. We explain the difficulties. 

Given the Jacobian of a genus $2$ curve, the following technique can be used to compute all isogenies from this Jacobian such that the kernel is a maximal isotropic subgroup of the $2$-torsion. Take $6$ points $\scrP$ in the complex plane that lie on a conic $C$; they define a unique double cover $X_1 \arr C \cong \PP^1$ ramified exactly along $\scrP$. Divide the points of $\scrP$ into $3$ pairs, $\{P_1, P_1^\prime\}, \{ P_2, P_2^\prime\}, \{ P_3, P_3^\prime \}$ (there are $15$ ways to do that). Let $L_i$ be the line through $P_i, P_i^\prime$ and let $t_1, t_2, t_3$ the three points where two of the lines $L_1, L_2, L_3$ intersect. Through each point $t_i$ there are two lines that are tangent to $C$ and intersect it in a pair of points $\{Q_i, Q_i^\prime\}$ and that altogether, usually, provide us with $6$ points $\scrQ$ lying on $C$ and so with another double cover $X_2 \arr C$ ramified now along the points of $\scrQ$. One can prove that this corresponds to an isogeny $f\colon \Jac(X_1) \arr \Jac(X_2)$ whose kernel is a $(2, 2)$ maximal isotropic subgroup of $\Jac(X_1)[2]$ (of which there are $15$). And in fact, one gets all such subgroups by a suitable partition of $\scrP$. Here one uses that by means of the Weierstrass points one can write all $2$-torsion points on $\Jac(X_1)$; in fact the differences $\alpha-\beta$ with $
\alpha, \beta \in \scrP$ give all the 2-torsion points of $\Jac(X_1)$ and the kernel of $f$ is then given by $\{ 0 \} \cup \{ (P_i+P_i^\prime) - (P_j+P_j^\prime): 1\leq i< j \leq 3\}$.
Moreover, the isogeny defined by the partition $\{ Q_i, Q_i^\prime\}_{i=1}^3$ of $\scrQ$ is exactly the dual isogeny $f^\vee\colon \Jac(X_2) \arr \Jac(X_1)$. See \cite{BM, DL}.

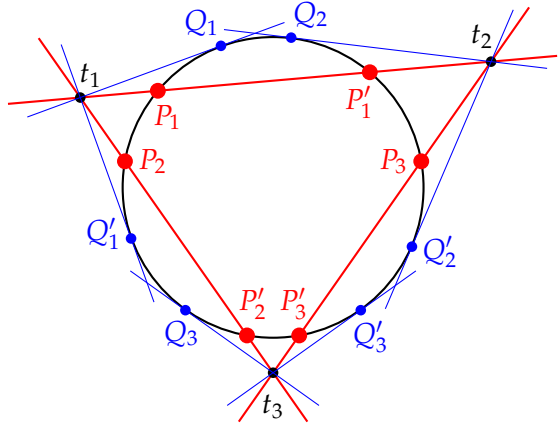
\begin{figure}
	
\begin{tikzpicture}
	
	\draw[name path=circ, thick] (0,0) circle (2);
	\node[circle] (S) at (0,0) [minimum size=4cm] {};
	
	\path[name path=P1] (140:2) coordinate (P1);
	\path[name path=P1'] (50:2) coordinate (P1');
	\path[name path=P2] (170:2) coordinate (P2);
	\path[name path=P2'] (260:2) coordinate (P2');
	\path[name path=P3] (10:2) coordinate (P3);
	\path[name path=P3'] (280:2) coordinate (P3');
	
	\fill[red] (P1) circle (3pt) node[anchor=110,inner sep=5pt] {$P_1$};
	\fill[red] (P1') circle (3pt) node[anchor=70,inner sep=5pt] {$P_1'$};
	\fill[red] (P2) circle (3pt) node[anchor=180,inner sep=5pt] {$P_2$};
	\fill[red] (P2') circle (3pt) node[anchor=260,inner sep=6pt] {$P_2'$};
	\fill[red] (P3) circle (3pt) node[anchor=0,inner sep=5pt] {$P_3$};
	\fill[red] (P3') circle (3pt) node[anchor=280,inner sep=6pt] {$P_3'$};
	
	\draw[name path=L1, thick,red] ($(P1)!-2cm!(P1')$) -- ($(P1')!-2.5cm!(P1)$);
	\draw[name path=L2, thick,red] ($(P2)!-2cm!(P2')$) -- ($(P2')!-1.4cm!(P2)$);
	\draw[name path=L3, thick,red] ($(P3)!-2.5cm!(P3')$) -- ($(P3')!-1.4cm!(P3)$);
	
	\path[name intersections={of=L1 and L2, by=A}];
	\path[name intersections={of=L2 and L3, by=B}];
	\path[name intersections={of=L3 and L1, by=C}];
	
	\fill[black] (A) circle (2pt) node[anchor=245] {$t_1$};
	\fill[black] (B) circle (2pt) node[anchor=90, inner sep=8pt] {$t_3$};
	\fill[black] (C) circle (2pt) node[anchor=290, inner sep=5pt] {$t_2$};
	
	\node (A1) at (tangent cs:node=S,point={(A)},solution=1) {};
	\draw[name path=T1, blue] ($(A)!-0.75cm!(A1)$) -- ($(A1)!-0.9cm!(A)$);
	\node (A2) at (tangent cs:node=S,point={(A)},solution=2) {};
	\draw[name path=T2, blue] ($(A)!-0.75cm!(A2)$) -- ($(A2)!-0.9cm!(A)$);
	\node (B1) at (tangent cs:node=S,point={(B)},solution=1) {};
	\draw[name path=T3, blue] ($(B)!-0.75cm!(B1)$) -- ($(B1)!-0.9cm!(B)$);
	\node (B2) at (tangent cs:node=S,point={(B)},solution=2) {};
	\draw[name path=T4, blue] ($(B)!-0.75cm!(B2)$) -- ($(B2)!-0.9cm!(B)$);
	\node (C1) at (tangent cs:node=S,point={(C)},solution=1) {};
	\draw[name path=T5, blue] ($(C)!-0.75cm!(C1)$) -- ($(C1)!-0.9cm!(C)$);
	\node (C2) at (tangent cs:node=S,point={(C)},solution=2) {};
	\draw[name path=T6, blue] ($(C)!-0.75cm!(C2)$) -- ($(C2)!-0.9cm!(C)$);

	\fill[blue] (A1) circle (2pt) node[anchor=350] {$Q_1'$};
	\fill[blue] (A2) circle (2pt) node[anchor=310] {$Q_1$};
	\fill[blue] (B1) circle (2pt) node[anchor=110] {$Q_3'$};
	\fill[blue] (B2) circle (2pt) node[anchor=80] {$Q_3$};
	\fill[blue] (C1) circle (2pt) node[anchor=230] {$Q_2$};
	\fill[blue] (C2) circle (2pt) node[anchor=170] {$Q_2'$};
	
\end{tikzpicture}
\caption{Defining a Richelot isogeny between genus $2$ hyperelliptic curves defined by ramified covers of a conic.}
\end{figure}

This works rather well and has been implemented for walks in $\scrG_{\rm S}^2(p,2)$ \cite{CDS, FT}. When we want to adapt this to the case of superspecial abelian surfaces with real multiplication by $\ol$, certain problems arise; we think they are attractive and hope that an interested reader would resolve them. Assume for example that $2$ is split in $\ol$. (As a concrete example we may take the superspecial curve $X: y^2 = x^5 + 1$, in characteristic $p \equiv 4 \pmod{5}$, that has real multiplication by $L = \QQ(\sqrt{5})$. Then, by \cite{GorR}, $\Jac(X)$ is a ssAV with RM.) The group $\Jac(X_1)[2]$ is a free module of rank $2$ over $\ol/2\ol \cong \FF_2e_1 \oplus \FF_2e_2$. An analysis of the interaction of the pairing with the $\ol$-structure, shows that there are nine $\ol$-invariant maximal isotropic subgroup of $\Jac(X_1)[2]$ (a choice of line in every isotypic $\ol$-submodule of $\Jac(X_1)[2]$). But how to determine which of the fifteen isotropic subgroups are the $\ol$-invariant ones? This necessitates understanding the action of $\frac{1+\sqrt{5}}{2}$ on $\Jac(X_1)[2]$, or as a correspondence on the conic $C\cong \PP^1$. Another problem that arises is how to determine the action $\frac{1+\sqrt{5}}{2}$ on $\Jac(X_2)[2]$. Nonetheless, for certain totally real fields, for example for $\QQ(\sqrt{5})$ there is a lot known. For the state of the art see \cite{CM, CFM} and the references therein.

If such an implementation is possible,
then by analogy with the elliptic curve case, one would expect the security of such hash functions to rest on the difficulty of computing endomorphism rings of ssAVs with RM. Likewise, one would expect there to be an effective method using closed paths in $\scrG_L(p, \ell)$ to generate finite index subrings of $\End(\uA)$. 

In fact, one can prove quite easily that the endomorphisms $f\in \End(\uA)$ satisfying $\deg_L(f) \in \{ \ell^n: n\in \ZZ_{\geq 0}\}$ generate over $\ol$ a subring of finite index of $\End(\uA)$. We prove this first for $\uA$ of the form $E\otimes_\ZZ \ol$, with the induced $\ol$-structure and polarization discussed in Example~\ref{ex:E tensor OL}. Since we allow $\ol$-linear combinations, it is enough to prove that endomorphisms with $\ell$-power degree generate a finite-index sublattice of $\End(E)$. While this statement follows of course from Theorem~\ref{thm: generation of orders}, it can also be easily derived from the earlier criterion of Bank et. al. (\S\ref{subsubsec:security}), and can be proven even more directly as follows.

Suppose $K$ is an imaginary quadratic field field in which $p$ is non-split and $\ell$ is split; note that this is just a congruence condition on the discriminant of $K$, so one can find infinitely many such fields. Let $\gerl$ be a prime of $K$ dividing $\ell$, and let $a$ be a positive integer such that $\gerl^{a}$ is a principal ideal with generator~$f$. Let $\tilde{E'}$ be an elliptic curve with CM by $\calO_{K}$, defined over a number field~$M$ and having good reduction at some prime above $p$. The reduction $E'$ by this prime then has CM by $\calO_{K}$ as well, so we can view $f$ as an element of $\End(E')$.
Then 
\[\deg(f) = N_{K/\QQ}(f) = \ell^{a},\]
and since $\Ker(f) \cong \calO_{K}/\gerl^{a}$ is cyclic of order $\ell^a$ we see that $f \not\in \ZZ$. Using that the graph $\scrG(p, \ell)$ is connected, we can find an $\ell$-power isogeny $h\colon E \arr E'$, and then $\phi:= h^\vee f h$ is an endomorphism of~$E$ with degree a power of $\ell$ and $\QQ(\phi) \cong \QQ(f)=K$, as $\phi$ has the same minimal polynomial as $\deg(h)f$. If we run a similar construction with two distinct fields $K_1,K_2$, we obtain endomorphisms $\phi_1$ and $\phi_2$ of $E$ that generate different quadratic subfields of $\End(E)\otimes \QQ$ and thus the lattice spanned by $\{1, \phi_1, \phi_2, \phi_1\phi_2\}$ (all elements of norm a power of $\ell$) has finite index in $\End(E)$.

We now use a similar argument to bootstrap from $B:=E\otimes \ol$ to any superspecial $\uA$. As the graph $\scrG_L(p, \ell)$ is connected, there is a path from $\uA$ to $\uB$ and this defines an isogeny $h\colon \uA \arr \uB$, well-defined up to $\End(\uB)^\times$. {There is also an isogeny $\tilde{h}\colon \uB \arr \uA$ with the property that $\tilde{h}\circ h=[\ell^r]$ for some $r$, obtained by selecting $r$ so that $\Ker(h)\leq \uA[\ell^r]$ and then taking the quotient of $\uB$ by $h(\uA[\ell^r])$.}
Consider the homomorphism of $\ol$-modules
\begin{equation} \label{eqn: 9793}
\End(\uB) \arr \End(\uA), \qquad f\mapsto \tilde{h}f h.
\end{equation}
We claim that $\deg_L(\tilde{h}f h)=\ell^{2r}\deg_L(f)$ for all $f\in\End(\uB)$. Indeed, using $\lambda_E$ to denote the principal polarization on $E\otimes \ol$ and $\lambda_A$ the one appearing in $\uA$, the diagram 
\begin{equation}
	\xymatrix{A^\vee \ar[d]_{\lambda_A^{-1}} & B^\vee \ar[l]_{ h^\vee}\ar@<0.3ex>[d]^{\lambda_B^{-1}}& B^\vee \ar[l]_{f^\vee}\ar[d]<0.3ex>^{\lambda_B^{-1}} & A^\vee\ar[l]_{\tilde h^\vee} \\
	A\ar[r]^h&B \ar[r]^f\ar@<0.3ex>[u]^{\lambda_B} &B \ar[r]^{\tilde h}\ar@<0.3ex>[u]^{\lambda_B}&A \ar[u]^{\lambda_A}}
\end{equation}
proves that 
\[\deg_L(\tilde hfh) = \deg_L(\tilde h)\deg_L(f) \deg_L(h).\]
Note that these are elements of $\ol$ and therefore commute.
A similar diagram without the middle square shows 
\[\deg_L(\tilde h)\deg_L(h)=\deg_L([\ell^r])=\ell^{2r}.\]
Thus we find $\deg_L(\tilde{h}f h)=\ell^{2r}\deg_L(f)$ as desired, so the image of (\ref{eqn: 9793}) is simply a dilation of $\End(\uB)$. We showed above that $\End(\uB)$ has a finite index sub-lattice generated by elements with $\deg_L(f)$ a power of $\ell$, and the image of this sub-lattice is finite index in $\End(\uA)$ and generated by elements with $\deg_L(f)$ a power of $\ell$.

\

\id Once again we find ourselves led to the question of whether a lattice can be generated by particular sets of elements, only in this case the lattices are now $\ol$-modules. For the rest of this article we turn our attention to this problem. 
The key result we are building towards is the following local-global principle for lattice cosets over totally real fields; see \S\ref{subsec: quad spaces lattices} for the definitions of any terms that have not yet been defined. The proof is completed in \S\ref{sec: thm proof}, and then we give applications in \S\ref{sec:inhomogenous theta functions}.

\begin{thm}\label{thm: coset local global}
	Let $n\geq 4$ be an even integer, let $\Lambda$ be an integral $\ol$-lattice of rank $n$ with positive-definite quadratic form $Q$, and let $\gerp$ be a prime ideal of $\ol$. Let $t\in\Lambda\setminus\gerp\Lambda$ and set $M:=t+\gerp\Lambda$. For any $\mu\in\ol$ coprime to $\disc \Lambda$ with $N_{L/\QQ}(\mu)$ sufficiently large, if for all places $v$ of $L$ there exists $x\in M_v:=t+(\gerp\Lambda\otimes_{\ol}\calO_{L,v})$ with $Q(x)=\mu$, then there exists $x\in M$ with $Q(x)=\mu$.
\end{thm}
\begin{rmk}
	The restriction that $n$ be even can almost certainly be removed, but justifying this would add significantly to the length of the paper and is not necessary for our main application. See Remark~\ref{rmk: why n even} for more details.
	
	This theorem (with the $n$ even restriction removed) is a consequence of Chan--Oh's \cite[Theorem 4.9]{CO}. Indeed, note that if $Q$ is anisotropic on $V_v$ then $v$ divides $\disc  Q$ \cite[Remark 3.8.1]{Hanke1}, and so if $\mu$ is coprime to $\disc \Lambda$ and is represented by $M_v$ then it must be primitively represented by $M_v$. However, the proof of \cite[Theorem 4.9]{CO} for the case $n=4$ depends critically on an unpublished result. 
With considerable effort, Theorem~\ref{thm: coset local global} can also be deduced from \cite{Shimura-inhomogeneous}.
\end{rmk}

\subsection{Lattices and cosets over totally real fields}\label{sec:lattices cosets} As a general reference to quadratic forms we use \cite{Omeara}, especially Part IV, although our terminology differs at times; see also \cite{Hanke1, Hanke2}. Let $L$ be a totally real field of degree $g$ over $\QQ$ and $\ol$ its ring of integers.  
Let $\sigma_1, \dots, \sigma_g\colon L \arr \RR$ denote its infinite places. For each finite place $v$ of $L$, we let $L_v$ denote the completion at $v$ and $\calO_{L,v}$ its valuation ring, i.e.~the completion of $\ol$ at $v$.

\subsubsection{Quadratic spaces and lattices}\label{subsec: quad spaces lattices}
By a  {\bf quadratic space}, we mean an $n$-dimensional vector space~$V$ over $L$, for some integer $n>0$, together with an $L$-bilinear symmetric form:
\[ (\cdot,\cdot)\colon V \times V \arr L.\]
To such a bilinear form we associate the quadratic form $Q(x) = (x, x)$ and so we have the relations 
\[ Q(x) = (x, x), \quad (x, y) = \frac{1}{2}(Q(x+y) - Q(x) - Q(y)).\]
One can therefore also define a quadratic space by means of the quadratic form $Q$, which is simply a function $Q\colon V \arr L$ such that the form $(\cdot,\cdot)$ thus defined is $L$-bilinear. We shall exclusively consider {\bf positive-definite} quadratic forms $Q$, meaning that under any embedding $\sigma_i\colon L \arr \RR$, the induced real-valued quadratic form is positive-definite. Therefore $Q:(V\setminus\{0\})\arr L$ takes values in the set $L^+$ comprising $x\in L$ such that $\sigma_i(x)>0$ for all embeddings $\sigma_i$. 

By a {\bf quadratic module} (or a {\bf lattice}) $\Lambda$ over $L$ we mean a finite rank $n\geq 1$ projective $\ol$-module, together with a function
\[ Q\colon \Lambda \arr L, \]
such that the associated function $(\cdot,\cdot)$ is a bilinear form. Then $V = \Lambda\otimes_\ol L$ is a quadratic space. Conversely, a finitely generated $\ol$-submodule $\Lambda$ of a quadratic space $V$, that contains a basis of $V$, is a lattice. If $Q$ takes values in $\ol$, we refer to $\Lambda$ as an {\bf integral lattice}.  By a procedure similar to the one for $\ZZ$-lattices, one can associate to $\Lambda$ a discriminant $\gerd_\Lambda$, which is a fractional ideal of $L$.

We will also define local and adelic versions of these structures. Let $\AA$ denote the adeles of~$L$. That is, \[\AA = \prod^\prime_{v} L_v,\] which is the restricted product over all places of~$L$ (for each element, all but finitely many components are in the valuation ring $\mathcal{O}_{L,v}$). Denote also $\AA_\gerf$ the restricted product over all finite places, and $\AA_\infty \cong L \otimes_\QQ \RR$ the product over the infinite places. Thus, $\AA = \AA_\gerf \times \AA_\infty$. Let $O_{Q}$ denote the orthogonal group of the quadratic form $Q$: that is, the $L$-linear automorphisms of $V$ that preserve $Q$. This is an algebraic group over $L$.

Given a finite place $v$ of $L$ and a lattice $\Lambda$,  we let $\Lambda_v:=\Lambda\otimes_\ol  \mathcal{O}_{L, v}$, which is a quadratic module in the quadratic space $V_v:= V\otimes_L L_v$ with quadratic form $Q\colon \Lambda_v  \arr L_v$. For finite places~$v$, we write $O_Q(\calO_v)$ to denote elements of $O_Q(L_v)$ that stabilize $\Lambda_v$. 

We can define adelic structures as restricted products of these local structures. First we have
$\Lambda_{\AA_\gerf} = \prod_{v< \infty}\Lambda_v$. We define $V_\AA$ and $V_{\AA_\gerf}$ as the restricted product of $V_v$ over all (resp.~all finite) places $v$, where all but finitely many components lie in $\Lambda_v$. Similarly, we define $O_Q(\AA)$ and $O_Q({\AA_\gerf})$ as the restricted product of $O_Q(L_v)$ over all (resp.~all finite) places $v$, where all but finitely many components lie in $O_Q(\calO_v)$.

\subsubsection{Lattice cosets and their genera}

See \cite{Shimura-inhomogeneous,KK,CO}.
Given an $\ol$-lattice $\Lambda$ in $V$ and a vector $t\in V$, we say that $M=\Lambda+t$ is a {\bf lattice coset} in $V$ with modulus $\Lambda$.
The {\bf conductor} $\gerc$ of $M$ is defined to be the ideal of elements $c\in \ol$ satisfying $cM\subseteq \Lambda$ (or equivalently $ct\in\Lambda$). Note that $M$ is a lattice if and only if $t\in\Lambda$ if and only if $\gerc=\ol$.

Given a finite place $v$ of $L$ and a lattice coset $M$ with modulus $\Lambda$, we let
\[ M_v:= \Lambda_v + t\subseteq V\otimes_L L_v;\]
such $M_v$ will be called a {\bf local lattice coset}. Note that $M_v=\Lambda_v$ at all finite places $v$ coprime to the condcutor of $M$. The lattice coset $M$ is uniquely determined by the collection of local lattice cosets $M_v$. Conversely, suppose we are given a collection of local lattice cosets $\Lambda^\prime_v + t^\prime_v$ for each finite~$v$, where $\Lambda^\prime_v$ is a lattice in $V_v$ and $t^\prime_v\in V_v$. If $\Lambda^\prime_v = \Lambda_v$ for almost all $v$, and there exists a positive integer $N$ such that $Nt^\prime_v\in \Lambda^\prime_v$ for all $v$, then there is a unique $\ol$-lattice coset $\Lambda^\prime +t$ in $V$ whose completions are equal to the $\Lambda^\prime_v + t^\prime_v$.
Thus, for any $\gamma= (\gamma_v)_v \in  O_Q(\AA_\gerf)$ there is a lattice coset in $V$, denoted $\gamma M$, such that $(\gamma M)_v = \gamma_v M_v$ for all $v$.

Given another lattice coset $M^\prime = \Lambda^\prime + t^\prime$, where $\Lambda^\prime$ is another $\ol$-lattice in the same quadratic space $V$, we say that $M$ is {\bf isomorphic} (or {\bf isometric}) to $M^\prime$ if $\gamma M =M^\prime$ for some $\gamma \in O_Q(L)$. Note that this is equivalent to the conditions $\gamma \Lambda  = \Lambda^\prime$ and $\gamma t- t^\prime \in \Lambda^\prime$. For a finite place~$v$ we say $M$ and $M^\prime$ are {\bf locally isomorphic} at~$v$ if $\gamma_v M_v = M_v^\prime$ for some $\gamma_v\in O_Q(L_v)$. We say that $M^\prime$ is in the {\bf genus} $\gerg = \gerg(M)$ of~$M$ if the two lattice cosets are locally isomorphic at every finite place, or equivalently if $\gamma M=M^\prime$ for some $\gamma\in O_Q(\AA_\gerf)$. We let $[\gerg] = [\gerg(M)]$ denote the set of isomorphism classes of lattice cosets within the genus. This is a finite set: see \cite[103:4 Theorem]{Omeara} for lattices, and the result for general lattice coset follows from this by a simple argument given below. We define the {\bf genus number} $h(\gerg)$ to be the cardinality of $[\gerg]$.

Let $M$ be a lattice coset with modulus $\Lambda$, and $\gamma = (\gamma_v)_v \in O_Q(\AA)$. Recall that, by the definition of $O_Q(\AA)$, $\gamma_v$ preserves $\Lambda_v$ for all but finitely many $v$. Let 
\[K^M_\AA := \{ \gamma \in O_Q(\AA): \gamma M = M\}\]
denote the {\bf adelic stabilzer} of $M$. Then we can write $K^M_\AA = \prod_v K^M_v$, where for $v$ infinite $K_v$ is the compact group $O_Q(L_v)$, and for $v$ finite $K_v$ is the compact open group $\{ \gamma_v \in O_Q(L_v): \gamma_v M_v = M_v\}$. Thus, $K_\AA^M$ is a compact open subgroup of $O_Q(\AA)$, and there is a bijection
\begin{align*}
	O_Q(\AA)/K^M_\AA &\longleftrightarrow \gerg(M),\\
	\gamma &\mapsto \gamma M.
\end{align*}
This induces a bijection
\begin{equation}\label{eqn: genus class number}
	O_Q(L)\backslash O_Q(\AA)/K^M_\AA \longleftrightarrow [\gerg(\Lambda)].
\end{equation}
For $M=\Lambda+t$ of conductor $\gerc$, observe that any automorphism of $\Lambda_v$ that acts as the identity on the finite group $\Lambda_v/\gerc \Lambda_v$ will stabilize $M_v$. Thus
\[ \{ \gamma \in K_\AA^\Lambda: \gamma_v \equiv Id_v \pmod{\gerc\Lambda_v}\text{ for all finite }v\} \subseteq K_\AA^M \subseteq K_\AA^\Lambda,\]
and since $\gerc\Lambda_v=\Lambda_v$ for all but finitely many places, the left-hand group is finite index in $K_\AA^\Lambda$. Hence the index of $K_\AA^M$ in $K_\AA^\Lambda$ is finite, showing that finiteness of the genus number of $M$ follows from finiteness of the genus number of the lattice $\Lambda$.
Note that one may also obtain finiteness of the genus number $h(\gerg(M))$ directly from Equation~(\ref{eqn: genus class number}) using general properties of reductive groups \cite[Theorems 5.1 \& 8.1]{PR}. 

\subsubsection{Adelic double cosets}

For a given genus $\gerg=\gerg(M)$, let $\scrB = \{ \gamma_1, \dots, \gamma_{h(\gerg)}\} \subset \calO_Q(\AA)$ be a finite set such that the lattice cosets $\{M_i: = \gamma_i M: 1\leq i \leq h(\gerg)\}$ are a set of representatives for the isomorphism classes of lattice cosets in $\gerg$. We let $\gamma_1 = Id$ so that $M_1 = M$, and without loss of generality we can take the archimedean components of each $\gamma_i$ to be the identity. The elements $\gamma_i$ give us an explicit parametrization of the double cosets in Equation~(\ref{eqn: genus class number}):
\begin{equation}\label{eqn: double coset decomp}
	O_Q(\AA)=\bigsqcup_{i=1}^{h(\gerg)} O_Q(L)\gamma_i K^M_\AA.
\end{equation}
These double cosets are not necessarily all the same size. More precisely, we consider the image of each double coset $O_Q(L)\gamma_i K^M_\AA$ in the compact quotient $O_Q(L)\backslash O_Q(\AA)$; this is the same as the image of $\gamma_i K^M_\AA$. Two elements $h,h'\in\gamma_i K^M_\AA$ have the same image in $O_Q(L)\backslash O_Q(\AA)$ if and only if $h'h^{-1}$ is in the finite group $O_Q(L)\cap \gamma_i K^M_\AA\gamma_i^{-1}$, which is the stabilizer $\Aut(M_i)\leq O_Q(L)$ of $M_i$. Thus, letting 
\[w_i:=\#(O_Q(L)\cap \gamma_i K^M_\AA\gamma_i^{-1})=\sharp \Aut(M_i)\]
denote the cardinality of this group, we see that the map from $\gamma_i K^M_\AA$ to $O_Q(L)\backslash O_Q(\AA)$ is $w_i$-to-one. 

Consider a Haar measure on $O_Q(\AA)$ normalized so that the compact open subgroup $K^M_\AA$ has measure $1$. For any $\gamma M\in \gerg(M)$ the adelic stabilizer of $\gamma M$ is $\gamma K^M_\AA \gamma^{-1}$, which also has measure $1$ becaues $O_Q(\AA)$ is unimodular (this follows for example from \cite[Lemma 5.5]{Borel} because $O_Q(L_v)$ is compact at all infinite $v$). Thus this choice of Haar measure depends only on the genus $\gerg$. 

The left coset $\gamma_i K^M_\AA$ also has measure $1$, so the induced measure of its image in $O_Q(L)\backslash O_Q(\AA)$ is $\frac{1}{w_i}$. 
As a consequence, we can compute the volume of $O_Q(L)\backslash O_Q(\AA)$ with respect to the induced Haar measure to be
\begin{equation}\label{eq:mass}
	m(\gerg):={\rm Vol}_\gerg(O_Q(L)\backslash O_Q(\AA))= \sum_{i=1}^{h(\gerg)}\frac{1}{w_i},
\end{equation}
where we write ${\rm Vol}_\gerg$ to emphasize that the choice of measure depends on $\gerg$. We call $m(\gerg)$ the {\bf mass} of the genus $\gerg$. Thus any fundamental domain for the action of $O_Q(L)$ on $O_Q(\AA)$ has volume $m(\gerg)$ with respect to this measure, and the intersection of the double coset $O_Q(L)\gamma_i K^M_\AA$ with any such fundamental domain has measure $\frac1{w_i}$.

\subsection{Hilbert modular forms and adelization}
In this section, we discuss some general facts about Hilbert modular forms and their adelic versions. For motivation, we begin in \S\ref{subsec: intro inhomogenous theta functions} by introducing the Hilbert modular forms that will be used for our main results and their relation to the theory of lattice cosets just discussed. The rest of the section is devoted to the general theory. We return to our specific focus on the theta functions associated to lattice cosets in \S\ref{sec:theta functions via theta corr}.

\subsubsection{Inhomogeneous theta functions of totally real lattice cosets}\label{subsec: intro inhomogenous theta functions}

See especially \cite{Shimura-inhomogeneous}. The setup is nearly identical to the homogeneous case (that is, the case $M$ is a lattice); for this case see for example \cite{Eichler2, KK, Walling1}, though mind the different normalizations across these references.

Let $\gerH$ denote the complex upper half plane. As above we let $L$ be a totally real field of degree $g$ over $\QQ$, and $\Lambda$ a rank $n$ quadratic module over $\ol$. To a lattice coset $M$ with modulus $\Lambda$ we can associate a {\bf theta function} $\theta_M:\gerH^g\to \CC$, defined for $\tau:=(\tau_1, \dots, \tau_g) \in \gerH^g$ by
\begin{equation}\label{eq:theta coset}
	\theta_{M}(\tau): = \sum_{x\in M} e^{2\pi i \Tr Q(x)\cdot\tau}
\end{equation}
where we write $\Tr\mu\cdot\tau:=\sigma_1(\mu)\tau_1+\cdots+\sigma_g(\mu)\tau_g$ for $\mu\in L$. We define the {\bf representation numbers} for the lattice coset $M$ and $\mu\in L$ by
\[ r_M(\mu) := \sharp \{ x \in M: Q(x) = \mu\}.\]
Note that $r_M(0)$ is $1$ if $M=\Lambda$ is a lattice and $0$ otherwise. Recalling that $Q$ is positive-definite, so $Q(x)\in L^+$ for all nonzero $x\in V$, we can rewrite the theta function as
\begin{equation} \theta_M(\tau) = r_M(0) + \sum_{\mu\in L^+} r_M(\mu)e^{2\pi i\Tr\mu\cdot \tau}. 
\end{equation}
Although the summation ranges over $L^+$, the set of values $Q(M)$ acquired by $Q$ on $M$ is in fact contained in some fractional ideal in $L$.
A fundamental fact is that $\theta_M$ is a {\bf Hilbert modular form} in the sense we shall define below. It has weight $n/2$, 
and one has a precise description of the level and character; see \cite[Theorem 3.7]{Walling1} for the case $M=\Lambda$.  Shimura \cite[Theorem A3.8]{Shimura-arithmeticity} deals with a general case, but a more accessible reference might be \cite[Theorem p.154]{Garrett} though the level it gives may not be optimal.

\begin{exa} Let $L$ be the lattice $R \otimes \ol$, where $R$ is a maximal order in the quaternion algebra $B_{p, \infty}$. We discussed this  above in Example~\ref{ex:E tensor OL} and on page~\pageref{further discussion}. The quadratic form is simply the norm and its discriminant is $p^2\ol$. We choose an integer $N\geq 1$ and a vector $t \in R \otimes \ol$. We let $M$ be the lattice coset $t + N R \otimes \ol$. If we define the ideal $\gerf =4pN\gerd_L$, then $\gerf$ satisfies the five conditions appearing in \cite[p. 153]{Garrett}: (i) holds because $\gerf \subseteq 4\gerd_L$, (iii) amounts to $\gerf \subseteq N\ol$, (v) holds because $\gerf \subseteq p\ol$, and (ii) and (iv) hold vacuously.  Then \cite[Theorem, p. 154]{Garrett} gives that $\Theta_M$ is a Hilbert modular form of parallel weight $2$ and level $\Gamma(\gerf) = \{ \gamma\in \SL_2(\ol): \gamma \equiv I_2 \pmod{\gerf}\}$, a priori with a character, but the comments in loc. cit. following the Theorem show that this character is trivial in our example.
\end{exa}

\medskip

\id
For a genus $\gerg=\gerg(M)$, we define the {\bf theta function of the genus},
\begin{equation}\label{eq:theta genus}
	\theta_\gerg(\tau) = \frac{1}{m(\gerg)} \sum_{i = 1}^{h(\gerg)} \frac{1}{w_i} \theta_{M_i}(\tau)
\end{equation}
where $w_i$ is the order of the stabilizer of $M_i$ in $O_Q(L)$, and the mass $m(\gerg)$ of $\gerg$ is as defined in Eq.~(\ref{eq:mass}).
Its Fourier coefficients are averaged representation numbers $\frac{1}{m(\gerg)} \sum_{i = 1}^{h(\gerg)} \frac{1}{w_i}r_{M_i}(\mu)$, with each $\theta_{M_i}$ weighted by the measure of the double coset $O_Q(L)\gamma_i K^M_\AA$ in $O_Q(L)\backslash O_Q(\AA)$. 
The key points that we will need are that:

\medskip
\begin{enumerate}
\item  $\Theta_\gerg$ is an \textit{Eisenstein series} on $\gerH^g$ of weight $n/2$ for a suitable discrete subgroup of $\SL_2(L)$;
\item For every $i$, $\theta_{M_i}  - \theta_\gerg$ is a \textit{cusp form}. That is, $\theta_\gerg$ is the Eisenstein component of \textit{every} $\theta_{M_i}$.
\end{enumerate}

\medskip
\id So if we can show that the Fourier coefficients of the Eisenstein series grow faster than the corresponding coefficients of the cusp forms, we can use this to show that $r_M(\mu)>0$ for certain values of $\mu$. We will explain these points in \S\ref{subsec: eisenstein} and \S\ref{subsec: cuspidal part} respectively, after we have given an adelic description of the theta functions involved.

\subsubsection{Adelic and classical modular functions}\label{subsec: adelic classical}

Following \cite{Garrett, Gelbart, KL}, we discuss the connection between classical Hilbert modular forms and adelic Hilbert automorphic forms. We restrict to the case of parallel weight $(w, \dots, w)$ as this suffices for the applications in this paper, and we additionally assume that the totally real field $L$ is not $\QQ$. 
Recall that $\sigma_1,\ldots,\sigma_g$ denote the embeddings $L\to\RR$.

\medskip

\id
On the classical side, define an {\bf automorphy factor} 
\[j(\alpha,\tau):=\prod_{i=1}^g (c_i\tau_i+d_i),\qquad \tau=(\tau_1,\ldots,\tau_g)\in\gerH^g,\qquad \alpha=\left(\begin{psmallmatrix}
	a_i&b_i\\c_i&d_i
\end{psmallmatrix}\right)_{i=1,\ldots,g}\in\SL_2(\RR)^g,\]
which satisfies
\[ j(\alpha_1 \alpha_2, \tau) = j(\alpha_1, \alpha_2\tau) \cdot j(\alpha_2, \tau).\]
Let $i_\infty\colon\SL_2(L) \injects \SL_2(\RR)^g$ be the embedding induced by the $g$ embeddings $\sigma_i\colon L\to\RR$. For $\gamma\in\SL_2(L)$ we also write $j(\gamma,\tau)$ to denote $j(i_\infty(\gamma),\tau)$.
Using this, we define the {\bf slash operator}~$\vert_w \gamma$ for $w\in\ZZ$ and $\gamma \in \SL_2(L)$ as an operator on continuous functions $f\colon \gerH^g \arr \CC$ by 
\[ (f\vert_w \gamma) (\tau) := j(\gamma, \tau)^{-w} f(\gamma \tau).\]
This defines a right action of $\SL_2(L)$ on such functions, namely $((f\vert_w\gamma_1)\vert_w\gamma_2) = f\vert_w (\gamma_1\gamma_2)$.
Let~$\Gamma$ be a congruence subgroup of $\SL_2(L)$. We  say a continuous function $f\colon\gerH^g\to \CC$ is {\bf modular} of weight $w$ and level $\Gamma$ if $f\vert_w\gamma = f$ for all $\gamma \in \Gamma$. If $f$ is modular of weight $w$ and level~$\Gamma$ then $f\vert_w\gamma$ is modular of weight $w$ and level $\gamma^{-1}\Gamma \gamma$. If in addition $f$ is holomorphic then we call $f$ a {\bf Hilbert modular form} and then all $f\vert_w\gamma$ will be Hilbert modular forms as well.

\medskip

\id
On the adelic side, we consider complex-valued functions on $\SL_2(\AA)$. We may view $\SL_2(\RR)^g$ as a subgroup of $\SL_2(\AA)$ by sending $\alpha = (\alpha_1,\ldots,\alpha_g)$ to the element that is $I_2$ at each finite place and $\alpha_i$ at the infinite place corresponding to $\sigma_i$, an element we shall also denote $\epsilon_\infty(\alpha)$, if confusion may arise. Within this subgroup we have the compact subgroup ${\rm SO}_2(\RR)^g$, and we can write an element of this group as 
\[ k(\theta) = \left(\left(\begin{smallmatrix}
	\hspace{9.5pt}\cos\; \theta_i & \sin\; \theta_i \\ -\sin\; \theta_i & \cos\; \theta_i
\end{smallmatrix}\right)\right)_{i=1,\ldots,g},\qquad \theta=(\theta_1,\ldots,\theta_g)\in\RR^g.\]
Let $\KK$ be a compact open subgroup of $\SL_2(\AA_\gerf)$. We say $F\colon \SL_2(\AA)\to \CC$ is an {\bf adelic automorphic function} of weight $w$ and level $\KK$ if it satisfies 
\[F(\gamma \g k(\theta) r)=e^{iw(\theta_1+\cdots+\theta_g)}F(\g)\]
for $\gamma\in \SL_2(L)$ (here embedded diagonally in $\SL_2(\AA)$), $\g\in\SL_2(\RR)^g$, $k(\theta)\in {\rm SO}_2(\RR)^g$, and $r\in \KK$; in  words,  $F$ is invariant under left multiplication by $\SL_2(L)$, invariant under right multiplication by $\KK$, and equivariant under right multiplication by ${\rm SO}_2(\RR)^g$ with character $k(\theta)\mapsto  e^{iw\sum\theta_i} = j(k(\theta), i)^{-w}$.

\medskip

\id
We now establish a one-to-one correspondence between modular functions $f\colon\gerH^g\to\CC$ and adelic automorphic functions. First, given a modular function $f$ of weight $w$ and level $\Gamma$, we describe how to elevate it to an adelic automorphic function 
\[f^\sharp\colon\SL_2(\AA)\to \CC\]
of weight $w$ and an appropriate choice of level $\KK$. We call $f^\sharp$ the {\bf adelization} of $f$.

We first define $f^\sharp$ on the subgroup $\SL_2(\RR)^g$ of $\SL_2(\AA)$
by
\[ f^\sharp(\g) = j(\g, i)^{-w} f(\g(i)), \qquad \g \in \SL_2(\RR)^g.\]
The modular property of $f$ translates into left-$\Gamma$-invariance of $f^\sharp$, and we have the relation
\[ f^\sharp(\gamma\cdot \g\cdot  k(\theta)) = e^{iw (\theta_1+\cdots+\theta_g)} f^\sharp (\g), \qquad \gamma\in \Gamma, \quad \g\in \SL_2(\RR)^g,\quad k(\theta)\in {\rm SO}_2(\RR)^g.\] 
A consequence of strong approximation for $\SL_2(L)$ \cite[Proposition, p. 88]{Garrett} implies that the inclusion $\epsilon_\infty\colon\SL_2(\RR)^g\to\SL_2(\AA)$ induces a homeomorphism
\begin{align*}
	\Gamma\backslash \SL_2(\RR)^g  &\overset{\cong}{\longrightarrow}  \SL_2(L) \backslash \SL_2(\AA) /\KK,\\
	\Gamma \g &\mapsto  \SL_2(L) \g \KK,
\end{align*}
(more pedantically, $\Gamma \g \mapsto  \SL_2(L) \epsilon_\infty(\g) \KK$)
where $\KK := \prod_{v<\infty} \Gamma_v$ for $\Gamma_v$ the $v$-adic completion of~$\Gamma$; this is equal to $\SL_2(\calO_{L_v})$ at all but finitely many places $v$.
Thus, any left-$\Gamma$-invariant function~$\varphi$ on $\SL_2(\RR)^g$ extends uniquely to a function on $\SL_2(\AA)$ that is left-$\SL_2(L)$-invariant and right-$\KK$-invariant. Explicitly, we can write any $t\in \SL_2(\AA)$ as $t = \gamma \g k$ for some $\gamma \in \SL_2(L)$, $\g \in \SL_2(\RR)^g$, and $k\in \KK$, and then the assignment $\varphi(t) := \varphi(\g)$ gives a well-defined extension of $\varphi$ to all of $\SL_2(\AA)$.
Thus we see that $f^\sharp$, a priori defined only on $\SL_2(\RR)^g$, extends uniquely to an adelic automorphic function $f^\sharp$ on $\SL_2(\AA)$.

Conversely, given an adelic automorphic function $F$ of some level $\KK$ and weight $w$, we produce a modular function 
\[F^\flat\colon \gerH^g \arr \CC\]
of level $\Gamma := \SL_2(L) \cap (\KK \times \SL_2(\RR)^g)$ and weight $w$.  We call $F^\flat$ the {\bf unadelization} of $F$. To compute the value of $F^\flat$ at $\tau=x+iy = (x_1+iy_1,\ldots,x_g+iy_g)\in\gerH^g$, we define the element 
\begin{equation}
\label{eqn: Mtau}
\alpha_\tau\in\SL_2(\RR)^g \subset \SL_2(\AA),
\end{equation} 
which equals $\begin{psmallmatrix} y_i^{\sfrac 12} & x_iy_i^{-\sfrac 12} \\ 0 & y_i^{-\sfrac 12} \end{psmallmatrix} $ at the infinite place corresponding to the embedding $\sigma_i$, and is $I_2$ at every finite place. Note that $\alpha_{x+iy}= \alpha_{x+i}\alpha_{iy}$ and that $\alpha_\tau i=\tau$. Then we define
\begin{equation}
	F^\flat(\tau) := j(\alpha_\tau, i)^w F(\alpha_\tau) = \left(\prod_{i=1}^g y_i\right)^{-w/2}F(\alpha_\tau).
\end{equation}
It is straightforward to check that $(F^\flat)^\sharp=F$: note that every element of $\SL_2(\AA)$ is contained in $\SL_2(L)\cdot \alpha_\tau\cdot {\rm SO}_2(\RR)^g\cdot\KK$ for some $\tau\in\gerH^g$, so one merely needs to check that the functions agree on all elements of the form $\alpha_\tau$ and satisfy the same transformation properties with respect to $\SL_2(L)$, $\KK$, and ${\rm SO}_2(\RR)^g$. Conversely, if $f$ is a modular function, then $(f^\sharp)^\flat=f$. This establishes a one-to-one correspondence between modular continuous functions and adelic automorphic functions.

To characterize those $F$ for which $F^\flat$ is holomorphic is a rather delicate issue and we refer the reader to \cite{Garrett}. For the particular choices of $F$ we use, holomorphicity of $F^\flat$ will be evident from the Fourier expansions that we obtain.

We now address how adelization interacts with the slash operator. It is easy to check from the definitions that for $\gamma\in \SL_2(L) \overset{i_\infty}{\injects} \SL_2(\RR)^g \subset \SL_2(\AA)$ one has 
\begin{equation}\label{eqn: slash} (f\vert_w \gamma)^\sharp (\alpha) = f^\sharp(  i_\infty(\gamma) \alpha), \qquad \gamma\in \SL_2(L), \; \alpha\in \SL_2(\RR)^g.\end{equation}
Indeed, 
\begin{multline*}
 f^\sharp(i_\infty(\gamma)\alpha) = j(i_\infty( \gamma)\alpha, i)^{-w} f(\gamma \alpha (i)) = j(\alpha, i)^{-w}j(\gamma, \alpha(i))^{-w} f(\gamma\alpha (i)) \\= j(\alpha, i)^{-w} (f\vert_w \gamma)(\alpha(i)) = (f\vert_w\gamma)^\sharp (\alpha).
 \end{multline*}

\subsubsection{Adelic and classical Fourier coefficients}\label{subsec: adelic classical fourier}

One of the key features of (holomorphic) Hilbert modular forms is that they have a Fourier expansion. If $\Gamma\subset \SL_2(L)$ is a congruence subgroup, then there exists a fractional $\calO_L$-ideal $\gera$ in $L$ such that 
\[U_\gera:= \left\{ \begin{psmallmatrix} 1 & a \\ 0 & 1
\end{psmallmatrix}: a\in \gera\right\} \subset \Gamma.\]
Given any two such ideals $\gera,\gerb$ we have $U_{\gera+\gerb}=U_\gera U_\gerb\subset\Gamma$, and therefore there exists one such $\gera$ that contains all the others; this fractional $\calO_L$-ideal $\gera$ is called the {\bf translation ideal} of the cusp $i\infty$ relative to $\Gamma$. 

If $f$ is a Hilbert modular form of level $\Gamma$ and $\gera$ is the translation ideal of the cusp $i\infty$ relative to $\Gamma$, then $f$ is invariant under $U_\gera$, and thus $f$ has a classical Fourier expansion at the cusp $i\infty = (i\infty, \dots, i\infty)$,
\[ f(z) =  \sum_{\mu\in L} c_\mu(f) e^{2\pi i \;\Tr \mu\cdot z}, \]
for uniquely determined complex numbers $c_\mu(f)$. These complex numbers, the {\bf classical Fourier coefficients} of $f$ at the cusp $i\infty$, are defined by
\[c_\mu(f)={\rm vol}((L\otimes\RR)/\gera)^{-1} \int_{x\in (L\otimes \RR)/\gera} f(x+iy)e^{-2\pi i\Tr\mu\cdot(x+iy)}\,d_\infty x\qquad\text{for }y\in\RR_{>0}^g;\]
when $f$ is holomorphic this integral does not depend on $y$ \cite[\S 1.4]{Garrett}. Here $d_\infty x$ is the standard measure on $L\otimes \RR\simeq\RR^g$ such that ${\rm vol}(\RR^g/\ZZ^g) = 1$, and the volume of $(L\otimes\RR)/\gera$ is computed with respect to this measure. Let $\gera^\vee$ the $\ZZ$-dual of $\gera$ relative to the trace form; that is, $\gera^\vee = (\gerd_L \gera)^{-1}$. Then we have $c_\mu(f)=0$ for all $\mu\notin \gera^\vee$, and so the Fourier expansion could instead be written as a sum over $\gera^\vee$. In fact, as $L \neq \QQ$, we have additionally that $c_\mu = 0$ unless $\mu = 0$ or $\sigma_i(\mu) > 0$ for all  $i = 1, \dots, g$ (the {\bf Koecher principle}). 

The $q$-expansion of $f$ at other cusps can be defined as the $q$-expansion of $f\vert_w \gamma$ at the cusp $i\infty$ for $\gamma \in \SL_2(L)$. Since $f\vert_w\gamma$ has level $\gamma^{-1}\Gamma\gamma$, these Fourier coefficients are computed using the translation ideal of the cusp $i\infty$ associated to $\gamma^{-1}\Gamma\gamma$. (This gives a $q$-expansion of $f$ at the cusp $\gamma^{-1} i\infty$, but note that even if $\gamma_1^{-1}i\infty=\gamma_2^{-1}i\infty$, the $q$-expansions of $f\vert_w \gamma_1$ and of $f\vert_w \gamma_2$ may not agree, as the elements $\gamma_1,\gamma_2$ may parametrize the neighborhood of the cusp in different ways.) In particular, for every $\mu \in L$ we obtain a function \[c_\mu(\cdot, f)\colon \SL_2(L) \arr \CC,\]  defined for $\gamma\in\SL_2(L)$ as the coefficient $c_\mu(f\vert_w\gamma)$ in the Fourier expansion of $f\vert_w \gamma$ at the cusp $i\infty$. (Note in particular that $c_\mu(f) = c_\mu(I_2, f)$). We say a Hilbert modular form $f$ is a {\bf cusp form} if $c_0(\gamma,f)=0$ for all $\gamma\in\SL_2(L)$.

\medskip

\id On the adelic side, we first define an {\bf adelic exponential} function
\[e_\AA:\AA\to \CC^\times,\] which is a product of local exponential functions, 
$e_\AA((x_v)_v)=\prod_v e_v(x_v)$. For $v$ infinite and $x_v\in \RR$ we let $e_v(x_v)=e^{2\pi i x_v}$. For $v$ finite and $x_v\in L_v$, let $p$ be the rational prime that $v$ lies over, so that $\Tr_{L_v/\QQ_p}(x_v)\in\QQ_p$. The map $\ZZ[\frac 1p]\to\QQ_p$ induces an isomorphism $\ZZ[\frac 1p]/\ZZ\to \QQ_p/\ZZ_p$; thus there exists $z_v\in\ZZ[\frac 1p]$, unique up to translation by $\ZZ$, satisfying $\Tr_{L_v/\QQ_p}(x_v)\in z_v+\ZZ_p$. We define $e_v(x_v)=e^{-2\pi i z_v}$. The resulting adelic function $e_\AA$ evaluates to $1$ on $L$ and so defines a nontrivial character on the compact abelian group $\AA/L$. From its definition we can show that $e_\AA(\widehat{\gerd_L^{-1}}) = 1$, where for a fractional ideal $\gera$ of $L$, 
\[\hat{\gera} = \prod_{v<\infty} \gera\otimes \calO_{L, v}\]
 denotes the adelization of $\gera$.

Let $F$ be a function on $\SL_2(\AA)$ that is left-$\SL_2(L)$-invariant. For $\mu\in L$, define a function $\calW_\mu\colon \SL_2(\AA) \arr \CC$ -- the $\mu^{\rm th}$ {\bf adelic Fourier coefficient} of $F$ -- by
\[ \calW_\mu(\g) = \int_{\AA/L}  F(u(x) \g)e_\AA(\mu x)^{-1} d_\AA x,  \]
where $u(x) = \left( \begin{smallmatrix} 1 & x \\ 0 & 1
\end{smallmatrix}\right)$ and the measure $d_\AA x$ is normalized as to give $\AA/L$ volume $1$.
Note that if $F$ is right-$\KK$-invariant then so is the function $\calW_\mu$.

\medskip

\id One has the following properties of adelic Fourier expansion (\cite[p. 74]{Garrett}):
\begin{enumerate}
	\item We have $\calW_\mu(u(x)\g) = e_\AA(\mu x) \calW_\mu(\g)$, $x\in \AA$.
	\item As $L^2$ functions, $F(\g) = \sum_{\mu\in L} \calW_\mu(\g)$.
	\item For $M =  \left(\begin{smallmatrix} a & b \\ 0 & d \end{smallmatrix}\right)\in \SL_2(L)$ one has $\calW_\mu (\g) = \calW_{ad^{-1}\mu}(M\g)$.
\end{enumerate}
The function $\calW_0:\SL_2(\AA)\to\CC$ is called the {\bf constant term} of $F$. An adelic automorphic function $F$ is called {\bf cuspidal} 
if its constant term $\calW_0$ is~$0$ almost everywhere.

\medskip 
\id Given a classical holomorphic Hilbert modular form $f$, we compare the adelic Fourier coefficients of $F = f^\sharp$, that are functions $\calW_\mu( \cdot, f^\sharp)\colon \SL_2(\AA) \arr \CC$, with the Fourier coefficients of $f$ that are functions $c_\mu(\cdot, f) \colon \SL_2(L) \arr \CC$. Let $\KK$ be the level of $f^\sharp$. Given $\gamma\in\SL_2(L)$, we write $\gamma_\gerf\in\SL_2(\AA)$ for the element that equals $\gamma$ at all finite places and $I_2$ at all infinite places (that is, the image of $\gamma$ under the embedding of $\SL_2(L)$ into $\SL_2(\AA_\gerf)$, where $\SL_2(\AA_\gerf)$ is viewed as a subgroup of $\SL_2(\AA)$).

\begin{prop}\label{prop : adelic to classical coefficient}
	Let $f$ be a Hilbert modular form. Let $\gamma\in\SL_2(L)$, $\alpha\in\SL_2(\RR)^g$, and $r\in\KK$. Then for all $\mu\in L$ we have
	\[\calW_\mu(\gamma_\gerf^{-1}\alpha r, f^\sharp)=j(\alpha,i)^{-w}e^{2\pi i\Tr\mu\cdot \alpha(i)}c_\mu(\gamma, f).\]
\end{prop}

Before the proof, we mention the special case
\begin{equation}\label{eqn: adelic to classical coefficient}
	c_\mu(I_2,f)=e^{2\pi\Tr\mu}\calW_\mu(I_2, f^\sharp),
\end{equation}
as well as the following corollary. Recall that by strong approximation, every element of $\SL_2(\AA)$ can be written in the form $\gamma_\gerf^{-1}\alpha r$ as in the statement of the proposition, meaning that every adelic Fourier coefficient of $f^\sharp$ can be written in terms of some classical Fourier coefficient. We thus obtain the following.

\begin{cor}\label{cor: cusp classical and adelic} A classical Hilbert modular form $f$ is a cusp form if and only if its adelization $f^\sharp$ is a cuspidal automorphic function. 
\end{cor}

\begin{proof}[Proof of Proposition~\ref{prop : adelic to classical coefficient}]

We follow \cite[Proposition 12.3]{KL}, which proves this result for $L=\QQ$. Let $\Gamma$ be the level of $f$, so that $\gamma^{-1}\Gamma\gamma$ is the level of $f\vert_w \gamma$, and let $\gera$ be the translation ideal of the cusp $i\infty$ relative to $\gamma^{-1}\Gamma\gamma$. Then $U_\gera$ is contained $\gamma^{-1}\Gamma\gamma$ and so $\gamma U_\gera \gamma^{-1}\subset \Gamma$. Taking the product of $v$-adic completions for all finite $v$ we obtain $\gamma_\gerf U_{\hat\gera} \gamma_\gerf^{-1}\subset \KK$. So for any $x\in \AA$ and $\xi\in\hat\gera$ we have
\[f^\sharp(u(x)u(\xi)\gamma_\gerf^{-1}\alpha) = f^\sharp(u(x)\gamma_\gerf^{-1}\gamma_\gerf u(\xi)\gamma_\gerf^{-1}\alpha) = f^\sharp(u(x)\gamma_\gerf^{-1}\alpha),\]
since $\gamma_\gerf u(\xi)\gamma_\gerf^{-1}\in\KK$ commutes with $\alpha\in\SL_2(\RR)^g$ and $f^\sharp$ is right-$\KK$-invariant.

By strong approximation we have $\AA=L((L\otimes\RR)\times \hat{\gera})$, and so $\AA/L \simeq ((L\otimes\RR)\times \hat{\gera})/\gera$. 
From this we obtain an isomorphism of measure spaces (but not of groups),
\[\AA/L\simeq ((L\otimes\RR)/\gera) \times \hat{\gera};\]
see \cite[Theorem 7.40]{KL}.
Thus if we let $d_\gerf \xi$ denotes the Haar measure on $\AA_\gerf$ with the property that $d_\infty \times d_\gerf =d_\AA$, we can write
\begin{align*}
	\calW_\mu(\gamma_\gerf^{-1}\alpha r) &= \int_{\AA/L}  f^\sharp(u(x) \gamma_\gerf^{-1}\alpha r)e_\AA(\mu x)^{-1}\; d_\AA x\\
	&= \int_{(L\otimes\RR)/\gera}\int_{\hat{\gera}}  f^\sharp(u(x)u(\xi)\gamma_\gerf^{-1}\alpha)e_\AA(\mu \xi)^{-1}e_\AA(\mu x)^{-1} \;d_\gerf \xi \;d_\infty x\\
	&= \int_{(L\otimes\RR)/\gera}f^\sharp(u(x)\gamma_\gerf^{-1}\alpha)e_\AA(\mu x)^{-1} \;d_\infty x \cdot \int_{\hat{\gera}}  e_\AA(\mu \xi)^{-1}\;d_\gerf \xi.
\end{align*}
If $\mu\notin\gera^\vee$ then the integral over $\hat{\gera}$ is $0$; recalling that we have $c_\mu(\gamma,f)=0$ for $\mu\notin\gera^\vee$, we obtain the desired identity $0=0$. So we now consider the case $\mu\in\gera^\vee$. In this case we have $\mu\xi\in\widehat{\gerd}_L^{-1}$ for all $\xi\in\hat\gera$, so that $e_\AA(\mu\xi)=1$. We can conclude that
\begin{align*}
	\calW_\mu(\gamma_\gerf^{-1}\alpha r) = \vol(\hat{\gera})\int_{(L\otimes\RR)/\gera}f^\sharp(\gamma_\gerf^{-1}u(x)\alpha)e_\AA(\mu x)^{-1} \;d_\infty x.
	\end{align*}
Here we also used the fact that $\gamma_\gerf$ and $u(x)$ commute as they are supported on disjoint sets of places.
Writing $z:=\alpha(i)$ we have by definition of adelization
\begin{align*}
	f^\sharp(\gamma_\gerf^{-1}u(x)\alpha)&=f^\sharp(\gamma^{-1}i_\infty(\gamma)u(x)\alpha)\\
	&=j(i_\infty(\gamma)u(x)\alpha,i)^{-w}f(i_\infty(\gamma)u(x)z)\\
	&=j(\alpha,i)^{-w}j(u(x),z)^{-w}j(i_\infty(\gamma),z+x)^{-w}f(i_\infty(\gamma)(z+x))\\
	&=j(\alpha,i)^{-w}(f\vert_w \gamma)(z+x).
\end{align*}
Thus
\begin{align*}
	\calW_\mu(\gamma_\gerf^{-1}\alpha r) &= j(\alpha,i)^{-w}\vol(\hat{\gera})\int_{(L\otimes\RR)/\gera}(f\vert_w \gamma)(z+x)e^{-2\pi i\Tr\mu\cdot x}\;d_\infty x\\
	&= j(\alpha,i)^{-w}e^{2\pi i\Tr\mu\cdot z}\vol(\hat{\gera})\int_{(L\otimes\RR)/\gera}(f\vert_w \gamma)(z+x)e^{-2\pi i\Tr\mu\cdot (z+x)}\;d_\infty x\\
	&=j(\alpha,i)^{-w}e^{2\pi i\Tr\mu\cdot z}\vol(\hat{\gera})\vol((L\otimes \RR)/\gera)c_\mu(\gamma, f).
\end{align*}
Finally note that
\[\vol(\hat\gera)\vol((L\otimes\RR)/\gera)=\int_{(L\otimes\RR)/\gera}\int_{\hat\gera} d_\gerf \xi \;d_\infty x=\int_{\AA/L}d_\AA=1.\]
\end{proof}

\

\id
We remark that for a complete theory, and in particular for introducing Hecke operators into the picture, it is necessary to work with $\GL_2(\AA)^+$ and that introduces  complications; see \cite{Garrett}. For our applications this will not be necessary. We also remark that working with non-parallel weight does not introduce substantial difficulties, but  we avoid that as  we will have no need for that and loc. cit. restricts to parallel weight.

\subsubsection{The Deligne bound}\label{subsec:deligne}

To obtain an asymptotic on the growth of Fourier coefficients, we will need to have an upper bound on the growth of Fourier coefficients of a cuspidal Hilbert modular form of weight $w\geq 0$. Recall that for $L = \QQ$, and a cusp form $f = \sum_{n\in \QQ} a_n q^n$, $q = e^{2\pi i nz}$ on some congruence subgroup of $\SL_2(\QQ)$, one has the {\bf trivial bound} ({\bf Hecke bound})
\[ \vert a_n\vert \leq \Omega \cdot n^{w/2} , \]
where $\Omega$ is a constant depending only on $\Gamma$ and $w$. For a general modular form (not necessarily a cusp form) one has 
\[ \vert a_n\vert \leq \Omega\cdot   n^{w-1}\]
for some possibly larger constant $\Omega$. Note that for modular forms of weight $2$ the two bounds coincide. However, for cusp forms one has the much deeper {\bf Deligne bound}, for every $\epsilon > 0$
\[ \vert a_n \vert  \leq O( n^{(w-1)/2 + \epsilon}).\]
Similary, for Hilbert modular cusp forms $f(z) = \sum_{\mu \in L} c_\mu e^{2\pi i \Tr \mu\cdot z}$ of weight $w$ and level given by a congruence subgroup $\Gamma$, we have the trivial bound \cite[p. 24]{Garrett}
\[ \vert c_\mu \vert \leq \Omega\cdot  N_{L/\QQ}(\mu)^{w/2}.\] 
However, we will need the stronger Deligne bound. Unfortunately, this bound is proven \cite{Blasius, BL, Livne} (see also \cite{Shahidi}) in the language of automorphic representations and so some translation of concepts is needed. Due to limitation of space, we will omit this\footnote{In a nutshell, one deduces the general case from the case of cuspidal eigenforms, and those are treated through the associated $L$-series. The associated $L$ series have an Euler product and the control over the factors comes from constructing the associated global Galois representation in a way that provides information on the eigenvalues of Frobenii; typically, one finds the representations in this case in cohomology of algebraic varieties.} and use the following bound.
\begin{thm}\label{thm:deligne}
	For a Hilbert modular cusp form $f(z) = \sum_{\mu \in L} c_\mu e^{2\pi i \Tr \mu\cdot z}$ of weight $w$ and level $\Gamma$, a congruence subgroup, for every $\epsilon > 0$, we have the bound 
	\[ \vert c_\mu \vert \leq \Omega\cdot  N_{L/\QQ}(\mu)^{(w-1)/2+ \epsilon}.\] 
\end{thm}
\begin{rmk} In fact, we will only use this bound for $w \geq 2$ and it would suffice for us that $\epsilon < \sfrac{1}{2}$. In particular, the case $\epsilon = 1/5$ proven by Shahidi will suffice.
\end{rmk}

\subsubsection{The theta correspondence}\label{sec: theta correspondence} In this section and the following, we discuss a recipe for turning functions on the adelic group $O_Q(\AA)$ into classical Hilbert modular forms. We will then identify those functions on $O_Q(\AA)$ that give rise to inhomogeneous theta series, Eisenstein series, and cusp forms.

The first ingredient in this recipe is the Weil representation, a special case of which we consider here. We write as before $\AA$ for the adeles of $L$. Let $\scrS(V_\AA)$ denote the space of Schwartz functions on the group $V_\AA$.\footnote{We will also introduce Schwartz functions on $\SL_2(\AA)$ and $O_Q(\AA)$. We do not give full definitions of these functions; for our purposes, it is enough that the reader trusts that the Schwartz functions form a space of functions on which all the manipulations we do are well defined. See \cite[\S 3.1]{Bump} for details.} The orthogonal group $O(\AA)$ acts on $\scrS(V_\AA)$ by change of variable, $(\h\cdot\phi)(x)=\phi(\h^{-1} x)$.

The {\bf Weil representation} of $\SL_2(\AA)$ is a representation acting on $\scrS(V_\AA)$,
\[ \rho\colon \SL_2(\AA) \arr \GL(\scrS(V_\AA)).\]
The construction of the Weil representation $\rho$ is involved, but once granted its existence, it is uniquely determined by the specifying it on matrices that generate $\SL_2(\AA)$: 
\begin{align}\label{eq:weilrep on diagonal}
	\left(\rho\left(\begin{psmallmatrix}
		a&0\\0&a^{-1}
	\end{psmallmatrix}\right)\phi\right)(v)&=(a,(-1)^{n/2}\det(Q))_L|a|_\AA^{n/2}\phi(av),\\
	\label{eq:weilrep on unipotent}
	\left(\rho\left(\begin{psmallmatrix}
	1&x\\0&1
\end{psmallmatrix}\right)\phi\right)(v)&=e_\AA(xQ(v))\phi(v),\\
\nonumber 
\left(\rho\left(\begin{psmallmatrix}
	0&1\\-1&0
\end{psmallmatrix}\right)\phi\right)(v)&=\hat{\phi}(-v),
\end{align}
where $(\cdot,\cdot)_L$ denotes the Hilbert symbol (the product of local Hilbert symbols $(\cdot,\cdot)_{L_v}$), $|a|_\AA$ is the adelic absolute value (product of all local absolute values), $e_\AA$ is the adelic exponential defined in \S\ref{subsec: adelic classical fourier}, and $\hat{\phi}$ denotes an adelic Fourier transform (we will not need its explicit description in this article; see \cite{Hanke2} for details).
The Weil representation extends multiplicatively to a representation of $\SL_2(\AA) \times O_Q(\AA)$ on $\scrS(V_\AA)$, where $\h\in O_Q(\AA)$ acts by change of variable $(\rho(\h)\phi)(x)=\phi(\h^{-1}x)$ as above; a crucial fact we make use of is that the actions of $\SL_2(\AA)$ and $O_Q(\AA)$ commute. See \cite[p. 156]{Hanke2}\footnote{Note a minor inaccuracy in the reference: $\rho$ is not defined modulo ${\rm Sp}_2(F)$ as is written there, but the action of ${\rm Sp}_2(F)$ on the theta kernel introduced below is indeed trivial.}, \cite{Prasad}, and \cite{Garrett} for a more complete treatment.

Using the Weil representation we can define the adelic {\bf theta kernel}, which is a function
\[ \theta\colon \scrS(V_\AA) \times \SL_2(\AA) \times O_Q(\AA) \arr\CC, \]
taking a triple $(\phi, \g, \h)$ to 
\[ \theta_\phi(\g, \h) = \sum_{\x\in V} (\rho(\g, \h) \phi)(\x) = \sum_{\x\in V} (\rho(\g) \phi)(\h^{-1}\x).\]
An important fact about the theta kernel is that it is invariant under left multiplication by $\SL_2(L)$: if $\gamma\in\SL_2(L)$, then $\theta_\phi(\gamma \g,\h)=\theta_\phi(\g,\h)$.
Using the theta kernel one defines the {\bf theta correspondence} $\Theta$ that takes a Schwartz function $\Psi$ on $O_Q(\AA)$ to a Schwartz function $\Theta(\Psi)$ on $\SL_2(\AA)$ defined by
\[ \Theta(\Psi)(\g):= \int_{O_Q(L) \backslash O_Q(\AA)} \Psi(\h) \theta_\phi(\g, \h) d\h, \qquad \g\in \SL_2(\AA),\]
where $d\h$ is a Haar measure on $O_Q(\AA)$.

\subsection{Theta functions via the theta correspondence}\label{sec:theta functions via theta corr}

Fix a lattice coset $M=\Lambda+t$ in $V$. Let 
\[\phi:=\prod_v\phi_v\]
be the Schwartz function on $V_\AA$ such that $\phi_v$ is the characteristic function of $M_v$ at each finite place $v$, and $\phi_\sigma(x) = e^{-2\pi \sigma(Q(x))} $ at infinite places $\sigma\colon L\to\RR$. Note that for $\h \in O_Q(\AA)$ we have $\h\cdot \phi_M=\phi_{\h M}$. In addition, we fix the Haar measure $d\h$ on $O_Q(\AA)$ so that the adelic stabilizer $K_\AA^M$ of $M$ has measure $1$.

Recall that $\gamma_1=1,\gamma_2,\ldots,\gamma_{v(\gerg)}\in\SL_2(\AA)$ are elements supported at the finite places such that $M_i:=\gamma_iM$ form a set of representatives for the genus of $M$. For each $1\leq i\leq h(\gerg)$, let $\Psi_i$ be the function on $O_Q(\AA)$ such that $\Psi_i(\h)=1$ if $\h M\simeq M_i$  and $\Psi_i(\h)=0$ otherwise; thus $\Psi_i$ is the characteristic function of the double coset $O_Q(L)\gamma_i K^M_\AA$ in $O_Q(\AA)$ as in Equation~(\ref{eqn: double coset decomp}).
Consequently, the sum of all the $\Psi_i$ is the constant function $\mathbb{1}$ on $O_Q(\AA)$. 
\begin{lem}\label{lem:theta series from theta lift}
	We have $\Theta(\Psi_i)^\flat=\frac1{w_i}\theta_{M_i}$ and $\Theta(\mathbb{1})^\flat=m(\gerg)\theta_{\gerg}$.
\end{lem}
\id (Here $\theta_{M_i}$ and $\theta_\gerg$ are as defined in Eqs. (\ref{eq:theta coset}) and (\ref{eq:theta genus}), respectively.)
\begin{proof}
	We follow \cite[Section 4.5]{Hanke2}, which performs this computation in the case $M$ is a lattice. We decompose $O_Q(L)\backslash O_Q(\AA)$ into double cosets via Equation~(\ref{eqn: double coset decomp}). If $\h\in O_Q(L)\gamma_j K_\AA^M$, then we have $\Psi_i(\h)=1$ if $i=j$ and $\Psi_i(\h)=0$ otherwise, so 
	\begin{align*}
		\Theta(\Psi_i)(\g)&=\int_{O_Q(L)\backslash O_Q(\AA)} \Psi_i(\h)\theta_\phi(\g,\h)\,d\h\\
		&=\int_{O_Q(L)\backslash (O_Q(L)\gamma_i K_\AA^M)}\theta_\phi(\g,\h)\,d\h.
		\intertext{Recalling that the map from $\gamma_i K_\AA^M$ to its image in $O_Q(L)\backslash O_Q(\AA)$ is $w_i$-to-one,}
		\Theta(\Psi_i)(\g)&=\frac1{w_i}\int_{\gamma_i K_\AA^{M}} \theta_\phi(\g,\h)\,d\h\\
		&=\frac1{w_i}\int_{K_\AA^{M}} \theta_\phi(\g,\gamma_i \h)\,d\h\\
		&=\frac1{w_i}\int_{K_\AA^{M}} \sum_{\x\in V}(\rho(\g)\phi)(\h^{-1}\gamma_i^{-1}\x)\,d\h.
	\end{align*}
	
	Taking $\tau\in\gerH^g$, we evaluate $\Theta(\Psi_i)$ at the element $\g_\tau\in \SL_2(\AA)$ as defined in \S\ref{subsec: adelic classical}: this is $\begin{psmallmatrix}
		\phantom{y_i^{\sfrac{1}{2}}}\hspace{-10pt}1\hspace{2pt}& x_i\\
		0&\phantom{y_i^{\sfrac{1}{2}}}\hspace{-10pt}1\hspace{2pt}
	\end{psmallmatrix}
	\begin{psmallmatrix}
		y_i^{\sfrac{1}{2}}&0\\0&y_i^{-\sfrac{1}{2}}
	\end{psmallmatrix}$ at the infinite place $\sigma_i$, and the identity at all finite places. We can explicitly compute the action of $\rho(\g_\tau)$ on $\phi$ using Equations (\ref{eq:weilrep on diagonal}) and (\ref{eq:weilrep on unipotent}). Note in particular that the adelic Hilbert symbol, absolute value, and exponential are each a product over all places, as is $\phi$, so we can compute the action at each place independently. At all finite places $v$ we have $(\rho(\g_\tau)\phi)_v=\phi_v$, and since $\h\in K_\AA^M$ we find that $\h_v$ stabilizes $M_v$. Hence $\phi_v(\h_v^{-1}\gamma_{i,v}^{-1}\x)$ is $1$ if $\x\in \gamma_{i,v} M_v$ and $0$ otherwise. We conclude that
	\[\prod_{v<\infty}((\rho(\g_\tau)\phi)(\h^{-1}\gamma_{i}^{-1}\x))_v=\left\{\begin{array}{ll}
		1,&\x\in \gamma_i M=M_i,\\ 0&\text{otherwise.}
	\end{array}\right.\] 
	Now consider the infinite place $v$ corresponding to the embedding $\sigma_i$. Recall that we chose $\gamma_i$ to be the identity at infinite places, and since $y_i^{\sfrac{1}{2}}>0$ the local Hilbert symbol evaluates to $1$. We therefore have
	\begin{align*}
		((\rho(\g_\tau)\phi)(\h^{-1}\gamma_{i}^{-1}\x))_v&=y_i^{n/4}e^{2\pi i x_i\sigma_i(Q(\h_v^{-1}\x))}\phi_v(y_i^{\sfrac{1}{2}}\h_v^{-1}\x)\\
		&=y_i^{n/4}e^{2\pi i (x_i+iy_i)\sigma_i(Q(\x))},
	\end{align*}
	using the fact that $\h_v\in O_Q(\RR)$. 
	We can conclude that
	\begin{align*}
		\Theta(\Psi_i)^\flat(\tau)&=j(\g_\tau,i)^{n/2}\Theta(\Psi_i)(\g_\tau)\\
		&=\left(\prod_{i=1}^g y_i^{-\sfrac{1}{2}}\right)^{n/2}\left(\frac1{w_i}\int_{K_\AA^M}\sum_{\x\in M_i}\left(\prod_{i=1}^g y_i^{n/4}e^{2\pi i (x_i+iy_i)\sigma_i(Q(\x))}\right)\,d\h\right)\\
		&=\frac1{w_i}\sum_{\x\in M_i}e^{2\pi i \sum_{i=1}^g (x_i+iy_i)\sigma_i(Q(\x))}\left(\int_{K_\AA^M}\,d\h\right)\\
		&=\frac1{w_i}\theta_{M_i}(\tau).
	\end{align*}
	The result for $\Theta(\mathbb{1})^\flat$ follows from the fact that $\mathbb{1}$ is a sum over all $\Psi_i$ and the theta correspondence and unadelization are both linear.
\end{proof}

\subsubsection{The cuspidal part}\label{subsec: cuspidal part}
Recall that the integral of $\Psi_i$ on $O_Q(L)\backslash O_Q(\AA)$ is $\frac1{w_i}$, and the integral of~$\mathbb{1}$ on $O_Q(L)\backslash O_Q(\AA)$ is $m(\gerg)$. Thus we can write
\[\Psi_i=\frac{1}{w_i m(\gerg)}\mathbb{1}+\Psi_i^{({\rm cusp})}\]
for some function $\Psi_i^{({\rm cusp})}$ with integral $0$ on $O_Q(L)\backslash O_Q(\AA)$. We saw above what $\Psi_i$ and $\mathbb{1}$ map to under the theta correspondence and unadelization, so we now consider the suggestively named~$\Psi_i^{({\rm cusp})}$.

\begin{lem}\label{lem:cusp form from theta lift}
	If $\Psi$ has integral $0$ on $O_Q(L)\backslash O_Q(\AA)$ then $\Theta(\Psi)^\flat$ is a cusp form.
\end{lem}
\begin{proof}
	By Corollary~\ref{cor: cusp classical and adelic}, it suffices to prove that $\Theta(\Psi)$ is a cusp form on $\SL_2(\AA)$. We compute
	\begin{align*}
		\calW_0(\g)&=\int_{\AA/L} \Theta(\Psi)(u(x)\g)\,dx\\
		&=\int_{\AA/L} \int_{O_Q(L)\backslash O_Q(\AA)}\Psi(\h)\left(\sum_{v\in V}(\rho(u(x)\g)\phi)(\h^{-1}\x)\right)\,d\h\,dx.
		\intertext{
			Since $\AA/L$ and $O_Q(L)\backslash O_Q(\AA)$ are both compact and $\phi$ is a Schwartz function, we can interchange the integrals and sum:}
		\calW_0(\g)&=\int_{O_Q(L)\backslash O_Q(\AA)}\Psi(\h)\sum_{\x\in V}\left(\int_{\AA/L} (\rho(u(x)\g)\phi)(\h^{-1}\x)\,dx\right)\,d\h.
	\end{align*}
	We now examine the inner integral with $\g,\h,\x$ fixed. Using the fact that $\rho$ is a representation, we can apply Equation (\ref{eq:weilrep on unipotent}) to obtain
	\begin{align*}
		\int_{\AA/L} (\rho(u(x)\g)\phi)(\h^{-1}\x)\,dx&=\int_{\AA/L} (\rho(u(x))\cdot(\rho(\g)\phi))(\h^{-1}\x)\,dx\\
		&=(\rho(\g)\phi)(\h^{-1}\x)\int_{\AA/L} e_\AA(xQ(\x))\,dx
	\end{align*}
	(Recall that $\beta\in O_Q(\AA)$ so $Q(\beta_v s)=Q(s)$ for all places $v$).
	For $\x\neq 0$, we have $Q(\x)\in L^\times$, and so $x\mapsto e_\AA(xQ(\x))$ is a nontrivial character on $\AA/L$; thus this integral equals $0$. Thus only the $\x=0$ summand survives, and we have
	\begin{align*}
		\calW_0(\g)&=\int_{O_Q(L)\backslash O_Q(\AA)}\Psi(\h)(\rho(\g)\phi)(0)\,d\h.
	\end{align*}
	Now $(\rho(\g)\phi)(0)$ is constant in $\h$, so $\calW_0(\g)$ is some constant times the integral of $\Psi(\h)$ over $O_Q(L)\backslash O_Q(\AA)$, which is $0$ by assumption. Therefore $\Theta(\Psi)$ is a cusp form on $\SL_2(\AA)$.
\end{proof}

\begin{rmk}
	In \cite{Rallis} Rallis gives a criterion under which the theta correspondence takes cusp forms to cusp forms in a much more general setting.
\end{rmk}

Combining Lemmas \ref{lem:theta series from theta lift} and \ref{lem:cusp form from theta lift} we obtain the following (proven previoulsy in \cite[Lemma 2.2(4)]{Shimura-inhomogeneous}; see also \cite{Walling2} for the case $M=\Lambda$).

\begin{prop}
	For all $1\leq i\leq h(\gerg)$, $\theta_{M_i}-\theta_\gerg$ is a cusp form.
\end{prop}
\id
We can therefore bound the Fourier coefficients of $\theta_{M_i}-\theta_\gerg$ using Theorem~\ref{thm:deligne}. If we can prove a sufficiently large \textit{lower bound} for the Fourier coefficients of $\theta_\gerg$, we will obtain a \textit{nontrivial lower bound} for the Fourier coefficients of $\theta_{M_i}$ for all $i$.
 
\subsubsection{The Eisenstein part}\label{subsec: eisenstein}

The key idea of this section is that the Fourier coefficients of the theta function of a genus, $\theta_\gerg$, can be written as a product of local densities. A version of this claim is expressed in \cite[\S 5]{Hanke1} in the case that $M$ is a lattice. A product formula of this kind can be done by combining the {\bf Siegel-Weil formula} --- a special case of which says that the theta function of a genus (\ref{eq:theta genus}) is equal to an appropriately defined Eisenstein series --- with a {\bf product formula} for the Fourier coefficients of an Eisenstein series. Both the Siegel-Weil formula and the product formula are originally due to Siegel~\cite{Siegel} but we shall use more general versions of these results below. These steps were carried out by Shimura in order to get an explicit formula for the Fourier coefficients of the theta function of a lattice coset genus, \cite[Theorem 1.5]{Shimura-inhomogeneous}. For convenience and ease of exposition we will want to use a slightly different form for the coefficients.

Let $\Lambda$ be an \textit{integral} $\ol$-lattice. Breaking with the notation of the previous sections, we will assume throughout this section that $M$ is a lattice coset {\it contained in} $\Lambda$; that is, for some sublattice $\Lambda'\subset\Lambda$ and some $t\in\Lambda$, we have $M=t+\Lambda'$.

Fix an isomorphism $V\simeq L^n$. For each finite place $v$ of $L$, let $\ell_v:V_v\to\CC$ denote the characteristic function of the lattice coset $M_v$. Using this setup, we can define an {\bf Eisenstein series} $E_M(x, s)$ for $x\in\AA$ and $s\in\CC$ as in \cite[(2.3.4)]{Shimura-representations}.  This Eisenstein series has a Fourier expansion
\begin{equation}\label{eqn: eisenstein expansion}
E_M(x,s)=\sum_{\mu\in L}c_\mu(E;s)e_\AA(\mu x),
\end{equation}
and by the Siegel--Weil formula, we have
\[E_M(x,0)=\Theta(\mathbb{1})(u(x))\]
for all $n\geq 4$. When $M$ is a lattice this is identical to \cite[(5.3.1)]{Shimura-representations}. For $n\geq 5$ this uses the version of the Siegel--Weil formula proven by Weil \cite[Th\'eor\`eme 5]{Weil}. For $n=4$ this uses the version of the Siegel--Weil formula proven by Kudla and Rallis \cite{KudlaRallis}, though it does not follow immediately as they use a different definition of Eisenstein series from Shimura; see \cite[\S 5.3]{Shimura-representations} for how to deduce this identity from \cite{KudlaRallis} (the argument is given for lattices but holds equally well for lattice cosets). We can therefore compute the Fourier coefficients of $\Theta(\mathbb{1})$ using (\ref{eqn: eisenstein expansion}): 
\[\calW_\mu(I_2,\Theta(\mathbb{1}))=\int_{\AA/L} \left(\sum_{\mu\in L}c_\mu(E_M;0)e_\AA(\mu x)\right)e_\AA(\mu x)^{-1}d_\AA x=c_\mu(E_M;0),\]
so by Lemma~\ref{lem:theta series from theta lift} and (\ref{eqn: adelic to classical coefficient}),
\begin{equation}\label{eqn: classical coefficient of theta genus}
	c_\mu(\theta_{\gerg(M)}) 	=\frac{e^{2\pi\Tr\mu}}{m(\gerg(M))}c_\mu(E_M;0).
\end{equation}
The Fourier coefficients $c_\mu(E_M;s)$ of $E_M(x,s)$ can each be expressed as an infinite product of local factors 
\begin{equation}\label{eqn: product formula}
c_\mu(E_M;s)=d_{L}^{\sfrac12}\prod_{v} c_{\mu,v}(E_M;s)
\end{equation}
as in \cite[p.~71]{Shimura-representations}, where $d_L$ denotes the discriminant of $L/\QQ$ as above.

For $v$ a finite place of $L$, $c_{\mu,v}(E_M;s)$ is a {\bf local density}, measuring the $v$-adic volume of the level set $Q^{-1}(\mu)$ inside the lattice coset $M_v$. To be more precise, let $\mathfrak p$ denote the maximal ideal of $\calO_{L,v}$, and $q:=[\calO_{L,v}:\mathfrak p]$ the residue degree. For $\mu\in L$ and $k\geq 0$, set  
\begin{equation}\label{eqn: solution set}
A_{\mu,k}(M_v):=\{x\in M_v/\gerp^k \Lambda_v:Q(x)\equiv \mu\pmod{\mathfrak p^k}\}.
\end{equation}
Recall that the Eisenstein series $E_M$ was defined using a choice of isomorphism $V\simeq L^n$; this induces a measure on $V_v\simeq L_v^n$ normalized so that $\calO_{L,v}^n$ has measure $1$.
\begin{lem}\label{lem: local density}
	For all sufficiently large $k$, $c_{\mu,v}(E_M;0)=\tau_v\cdot \sharp A_{\mu,k}(M)/q^{k(n-1)}$, where $\tau_v$ is the volume of $\Lambda_v$ according to the measure defined above.
\end{lem}
\begin{proof}
	This is proven in the case $M=\Lambda$ in \cite[Lemma 2.5]{Shimura-representations}; we will merely point out the changes needed for our context. The fact that $\sharp A_{\mu,k}(M)/q^{k(n-1)}$ stabilizes for sufficiently large $k$ follows from the proof of \cite[Proposition 14.3]{Shimura-Euler-products}: the proposition is stated for lattices, but the proof proceeds by working modulo a sufficiently large power of $\gerp$ and considering one congruence class at a time (it is essentially an argument by Hensel lifting), and therefore holds even if we impose congruence conditions such as those defining the coset $M_v$. Let $d$ denote the limiting value of $\sharp A_{\mu,k}(M)/q^{k(n-1)}$ as $k\arr \infty$.
	
	As in \cite{Shimura-representations} we can compute
	\[(1-q^{-s})^{-1}|\beta|_v^{-s}c_{\mu,v}(E_M;s)=\sum_{k=0}^\infty q^{-ks}\int_{M_v}\int_{\gerp^{-k-e}}e_v(\delta_v^{-1}\sigma(Q(x)-\mu))\;d\sigma\;dx\]
	where $\beta\in L_v^\times$ is a parameter used to define the Eisenstein series \cite[(2.2.3)]{Shimura-representations}, $\delta_v$ is a generator of the different ideal of $L_v/\QQ_p$, $e$ is the $v$-adic valuation of $\beta/\delta$, $e_v$ is the $v$-adic exponential defined in \S\ref{subsec: adelic classical fourier}, and $dx$ is the Haar measure on $V_v\simeq L_v^n$ under which $\calO_{L,v}^n$ has measure $1$. Observe that $e_v(\delta_v^{-1}x)$ is identically $1$ on $\calO_{L,v}$ but is a nontrivial character on $\gerp^{-\ell}$ for any $\ell>0$, so the inner integral over $\gerp^{-k-e}$ is $q^{k+e}$ if $Q(x)\equiv\mu\pmod{\gerp^{k+e}}$ and $0$ otherwise. Setting \[W_k:=\{x\in M_v:Q(x)\equiv \mu\pmod{\mathfrak p^k}\},\] we can therefore write
	\[(1-q^{-s})^{-1}|\beta|_v^{-s}c_{\mu,v}(E_M;s)=\sum_{k=0}^\infty q^{k+e-ks}\vol(W_{e+k}).\]
	The equality $\vol(W_k)/\vol(\Lambda_v)=\sharp A_{\mu,k}(M_v)/q^{kn}$ holds for sufficiently large $k$, namely for all $k$ such that the modulus of $M$ contains $\gerp^k\Lambda$ (so that $M$ is a union of cosets of $\gerp^k\Lambda$).  Because of this, and the fact that $\sharp A_{\mu,k}(M_v)/q^{k(n-1)}=d$ for sufficiently large $k$, we have $d\cdot \tau_v=q^k\vol(W_k)$ for all but finitely many $k$, and so
	\[(1-q^{-s})^{-1}|\beta|_v^{-s}c_{\mu,v}(E_M;s)-d\tau_v\sum_{k=0}^\infty q^{-ks}\]
	is a polynomial in $q^{-s}$. Multiplying by $1-q^{-s}$ and evaluating at $s=0$ we obtain the desired equality $c_{\mu,v}(E_M;0)=d\cdot \tau_v$.
\end{proof}

In light of (\ref{eqn: classical coefficient of theta genus}), our next goal is to find a lower bound for $c_\mu(E_M;0)$ as a function of $\mu$; if we can show the growth rate is sufficiently large (in particular, fast enough that it overtakes the growth rate of any cusp form), then we will be able to conclude that the $\mu^{\rm th}$ Fourier coefficient of the theta function of $M$ is nonzero. Obtaining an explicit formula for this growth rate requires a fairly involved computation. Since this has already been done in \cite{Hanke1} in the case $M$ is a lattice, we will avoid this computation by relating the Fourier coefficients coming from lattice cosets to the Fourier coefficients coming from lattices. 

All the above setup holds for the lattice $\Lambda$ as well: that is, there is an Eisenstein series $E_\Lambda$ for which 
\[c_\mu(\theta_{\gerg(\Lambda)})=\frac{e^{2\pi\Tr\mu}}{m(\gerg(\Lambda))}c_\mu(E_\Lambda;0)\qquad\text{and}\qquad c_\mu(E_\Lambda;s)=d_L^{\sfrac12}\prod_{v} c_{\mu,v}(E_\Lambda;s).\]
For finite places $v$, defining
\begin{equation}\label{eqn: solution set}
A_{\mu,k}(\Lambda_v):=\{x\in \Lambda_v/\gerp^k \Lambda_v:Q(x)\equiv \mu\pmod{\mathfrak p^k}\}
\end{equation}
we have $c_{\mu,v}(E_\Lambda;0)=\tau_v (\sharp A_{\mu,k}(\Lambda))/q^{k(n-1)}$.

For further simplicity we now restrict our attention to cosets with modulus $\gerp\Lambda$ for a single prime $\gerp$ of $\ol$.

\begin{lem}\label{lem: bound coset coeff by lattice coeff}
	Let $\Lambda$ be an $\ol$-integral lattice of rank $n\geq 4$, and let $\gerp$ be a prime ideal of $\ol$ with ideal norm $q$. Consider the lattice coset $M:=t+\gerp \Lambda$ for some $t\in\Lambda\setminus \gerp\Lambda$ (so $M$ has modulus $\gerp\Lambda$). Suppose $\mu\in\ol$ is coprime to $\disc \Lambda$, and that $\mu$ is locally represented at $v$ by $M$. Then
	\[c_\mu(\theta_{\gerg(M)})> \frac{m(\gerg(\Lambda))}{2q^{\lambda n}m(\gerg(M))}c_\mu(\theta_{\gerg(\Lambda)}),\]
	where $\lambda=2\ord_\gerp(2)+1$.
\end{lem}
\begin{proof}
	Let $E_M$ and $E_\Lambda$ denote the Eisenstein series corresponding to $M$ and $\Lambda$ respectively. Let~$v$ denote the place corresponding to $\gerp$. At all places $v'\neq v$ (including infinite places) we have $M_{v'}=\Lambda_{v'}$ and so $c_{\mu,v'}(E_M;s)=c_{\mu,v'}(E_\Lambda;s)$.
	By (\ref{eqn: product formula}) and its analogue for $\Lambda$,
	\[c_\mu(E_M;0)c_{\mu,v}(E_\Lambda;0) = c_\mu(E_\Lambda;0)c_{\mu,v}(E_M;0).\]
	By (\ref{eqn: classical coefficient of theta genus}) and its analogue for $\Lambda$ we can then write
	\[c_\mu(\theta_{\gerg(M)})c_{\mu,v}(E_\Lambda;0) = \frac{m(\gerg(\Lambda))}{m(\gerg(M))}c_\mu(\theta_{\gerg(\Lambda)})c_{\mu,v}(E_M;0).\]
	So it suffices to show $c_{\mu,v}(E_\Lambda;0)\neq 0$ and $c_{\mu,v}(E_M;0)/c_{\mu,v}(E_\Lambda;0)> \frac12 q^{-\lambda n}$.
	
	First consider the case $\mu\notin\gerp$. We say $x\in\Lambda_v$ is ``good'' (cf.~\cite[Remark 3.1.1]{Hanke1}) if there exists $y\in\Lambda_v$ with $(x,y)\notin\gerp$, where $(\cdot,\cdot)$ is the inner product associated to $Q$ as defined in \S\ref{sec: quat alg basics}. If $x\in\Lambda_v$ satisfies $Q(x)\equiv\mu$, then $x$ is good as we can take $y=x$. Thus $A_{\mu,k}(M_v)$ and $A_{\mu,k}(\Lambda_v)$ both consist only of good elements. We can therefore use Hensel lifting to show that for any $\ell\geq 0$, the number of solutions mod $\gerp^{\lambda+\ell}\Lambda_v$ to $Q(x)\equiv \mu\pmod{\gerp^{\lambda+\ell}}$ is $q^{\ell(n-1)}$ times the number of solutions mod ${\gerp^{\lambda}}\Lambda_v$ to $Q(x)\equiv \mu\pmod{\gerp^{\lambda}}$ (see for instance \cite[Lemma 3.2]{Hanke1}). We can conclude that the equality in Lemma~\ref{lem: local density} holds already for $k=\lambda$, as does the corresponding equality for $c_{\mu,v}(E_\Lambda;0)$.
	By assumption there exists $z\in M$ with $Q(z)=\mu$, so $\sharp A_{\mu,\lambda}(\Lambda_v)\geq \sharp A_{\mu,\lambda}(M_v)\geq 1$. From this we can conclude that $c_{\mu,v}(E_\Lambda;0)\neq 0$ and
\begin{align*}
	\frac{c_{\mu,v}(E_M;0)}{c_{\mu,v}(E_\Lambda;0)}&=\frac{\sharp A_{\mu,\lambda}(M_v)}{\sharp A_{\mu,\lambda}(\Lambda_v)}
	\geq\frac{1}{\sharp\Lambda_v/\gerp^\lambda\Lambda_v}
	=\frac{1}{q^{\lambda n}}> \frac{1}{2q^{\lambda n}}.
\end{align*}

	Now suppose instead that $\mu\in\gerp$. Then the assumption that $\mu$ is coprime to $\disc\Lambda$ implies that $\gerp\nmid \disc \Lambda$. In particular, this implies that the quadratic form $Q$ is non degenerate on $\Lambda_v/\gerp\Lambda_v$, so that every element of $\Lambda_v\setminus\gerp\Lambda_v$ (and in particular every element of $M$) is good. As above, we can conclude that 
	\[c_{\mu,v}(E_M;0)=\tau_v \frac{\sharp A_{\mu,\lambda}(M_v)}{q^{\lambda(n-1)}}\geq \frac{\tau_v }{q^{\lambda(n-1)}}\]
and that $c_{\mu,v}(E_\Lambda;0)\neq 0$. To compute an upper bound for $c_{\mu,v}(E_\Lambda;0)$ will take a bit more work, as there may be elements $x\in\gerp\Lambda_v$ with $Q(x)=\mu$. These are not good solutions, and therefore the solutions mod $\gerp^\lambda\Lambda_v$ may have more than the expected number of lifts to $\gerp^{\lambda+\ell}\Lambda_v$. 
	
	To resolve this, we will partition the solutions to $Q(x)=\mu$ based on their $\gerp$-adic valuation. Let $k$ denote the $\gerp$-adic valuation of $\mu$. Every $x\in \Lambda_v$ with $Q(x)\equiv\mu\pmod{\gerp^{k+1}}$ has valuation at most $k/2$, because if $x\in\gerp^r$ for $r>k/2$, then $Q(x)\in\gerp^{k+1}$ while $\mu\notin\gerp^{k+1}$. For each $0\leq r\leq k/2$, the set of $x\in\gerp^r\Lambda_v\setminus\gerp^{r+1}\Lambda_v$ with $Q(x)\equiv\mu\pmod{\gerp^{2r+\lambda}}$ is in bijection (by dividing by a power of a uniformizer $\pi$ of $\gerp$) with the set of $x\in\Lambda_v\setminus\gerp\Lambda_v$ with $Q(x)\equiv\mu/\pi^{2r}\pmod{\gerp^\lambda}$. This latter set consists only of good elements. 	
	Thus for all $\ell\geq k$,
	\begin{align*}
	\sharp A_{\mu, \ell+\lambda}(\Lambda_v)&=\#\{x\in\Lambda_v/\gerp^{\ell+\lambda}\Lambda_v:Q(x)\equiv \mu\pmod{\gerp^{\ell+\lambda}}\}\\
		& =\sum_{r=0}^{k/2}\#\{x\in(\gerp^r\Lambda_v\setminus \gerp^{r+1}\Lambda_v)/\gerp^{\ell+\lambda}\Lambda_v:Q(x)\equiv \mu\pmod{\gerp^{\ell+\lambda}}\}\\
		& =\sum_{r=0}^{k/2}\#\{x\in(\Lambda_v\setminus \gerp\Lambda_v)/\gerp^{\ell-r+\lambda}\Lambda_v:Q(x)\equiv \mu/\pi^{2r}\pmod{\gerp^{\ell-2r+\lambda}}\}\\
		& \leq \sum_{r=0}^{k/2}q^{rn}\#\{x\in(\Lambda_v\setminus \gerp\Lambda_v)/\gerp^{\ell-2r+\lambda}\Lambda_v:Q(x)\equiv \mu/\pi^{2r}\pmod{\gerp^{\ell-2r+\lambda}}\}\\
		& = \sum_{r=0}^{k/2}q^{rn}q^{(\ell-2r)(n-1)}\#\{x\in(\Lambda_v\setminus \gerp\Lambda_v)/\gerp^{\lambda}\Lambda_v:Q(x)\equiv \mu/\pi^{2r}\pmod{\gerp^{\lambda}}\}.
	\end{align*}
	The second last line follows because each coset of $\gerp^{\ell-r+\lambda}\Lambda_v$ contains $q^{rn}$ cosets of $\gerp^{\ell-2r+\lambda}\Lambda_v$, and the last line follows by Hensel lifting of good solutions as above.
	Since the set in the last line is a subset of $\Lambda_v/\gerp^\lambda\Lambda_v$ which has $q^{\lambda n}$ elements, we have
	\[\sharp A_{\mu, \ell+\lambda}(\Lambda_v) \leq \sum_{r=0}^{k/2}q^{rn+(\ell-2r)(n-1)+\lambda n}< q^{\ell(n-1)+\lambda n}\sum_{r=0}^{\infty}q^{r(2-n)}=\frac{q^{(\ell+\lambda)(n-1)+\lambda}}{1-q^{2-n}},\]
	so that by Lemma~\ref{lem: local density},
	\[c_{\mu,v}(E_\Lambda;0)< \frac{\tau_v q^{\lambda}}{1-q^{2-n}}< 2q^{\lambda}\tau_v .\]
	Therefore $c_{\mu,v}(E_M;0)/c_{\mu,v}(E_\Lambda;0)> \frac12 q^{-\lambda n}$ as desired.
\end{proof}

\subsubsection{Proof of Theorem~\ref{thm: coset local global}}\label{sec: thm proof}
	We first recall the theorem statement. We have $n\geq 4$ even, and $\Lambda$, $\gerp$, and $M$ as in the statement of Proposition~\ref{lem: bound coset coeff by lattice coeff}. Assuming $\mu\in\ol$ is coprime to $\disc \Lambda$, $N_{L/\QQ}(\mu)$ is sufficiently large, and $\mu$ and is everywhere locally represented by $M$, we wish to show that $\mu$ is (globally) represented by $M$. We use the notation $f(\mu)\gg g(\mu)$ to mean that there exists a positive constant $c$ (not depending on $\mu$, though it may depend on $L,\Lambda,M,\varepsilon,\gerp$, etc.) such that for all $\mu$ with $N_{L/\QQ}(\mu)$ sufficiently large we have $f(\mu)> cg(\mu)$.
	
	We begin by bounding the Fourier coefficients of $\theta_{\gerg(\Lambda)}$ from below; this was done by Hanke \cite{Hanke1}. If $n\geq 5$ then by \cite[Theorem 5.7(c)]{Hanke1} we have
	\[c_\mu(\theta_{\gerg(\Lambda)})\gg N_{L/\QQ}(\mu)^{(n-2)/2}.\]
	Now assume $n=4$. The assumption that $\mu$ is coprime to $\disc \Lambda$ implies that $Q$ is isotropic on $V_v$ at all places $v$ dividing $\mu$ \cite[Remark 3.8.1]{Hanke1}. So by \cite[Theorem 5.7(b)]{Hanke1} we have
	\[c_\mu(\theta_{\gerg(\Lambda)})\gg N_{L/\QQ}(\mu)\prod_{\gerp\mid \mu,\;\chi(\gerp)=-1}\frac{N_{L/\QQ}(\gerp)-1}{N_{L/\QQ}(\gerp)+1}\]
	where $\chi$ is the Hecke character on $L$ defined by $\chi(\gerp):=\left(\frac{\disc\Lambda_\gerp}{\gerp}\right)$.  Noting that $\frac{x-1}{x+1}\geq \left(1-\frac1x\right)^2$ for all $x\geq 1$, we can bound
	\[\prod_{\gerp\mid \mu,\;\chi(\gerp)=-1}\frac{N_{L/\QQ}(\gerp)-1}{N_{L/\QQ}(\gerp)+1}\geq \prod_{\gerp:N_{L/\QQ}(\gerp)\leq N_{L/\QQ}(\mu)}\left(1-\frac{1}{N_{L/\QQ}(\gerp)}\right)^2\gg \log(N_{L/\QQ}(\mu))^{-2}\]
	by Merten's third theorem for number fields \cite[Theorem 5]{Leb}.
	In particular, for all $n\geq 4$ and all $\varepsilon>0$ we have 
	\[c_\mu(\theta_{\gerg(\Lambda)})\gg N_{L/\QQ}(\mu)^{(n-2)/2-\varepsilon}.\]
	By Lemma~\ref{lem: bound coset coeff by lattice coeff}, the ratio of $c_\mu(\theta_{\gerg(M)})$ to $c_\mu(\theta_{\gerg(\Lambda)})$ is bounded below by a positive constant (not depending on $\mu$), so this bound still holds if we replace $\gerg(\Lambda)$ by $\gerg(M)$.
	
	Now assume $n$ is even; then by Theorem~\ref{thm:deligne}, $|c_\mu(\theta_{\gerg(M)})-c_\mu(\theta_M)|\ll N_{L/\QQ}(\mu)^{(n/4-1/2)-\varepsilon'}$ for all $\varepsilon'>0$. Since $n/2-1 > n/4-1/2$ for all $n\geq 4$, we see that $c_\mu(\theta_{\gerg(M)})$ grows faster than $|c_\mu(\theta_{\gerg(M)})-c_\mu(\theta_M)|$, and therefore $c_\mu(\theta_M)>0$ for $N_{L/\QQ}(\mu)$ sufficiently large. This concludes the proof.
			
\begin{rmk}\label{rmk: why n even}
	The theorem statement should hold for odd $n\geq 5$ as well (and potentially also for $n=3$ if additional constraints are imposed). The main missing piece is a bound on the Fourier coefficients of the cusp form $\theta_\gerg-\theta_M$. When $n$ is odd, this cusp form is a modular form of half-integral weight. It is certainly possible to bound the Fourier coefficients of cusp forms of half-integral weight, by relating a cusp form of weight $w\in\ZZ+\frac12$ to its ``Shimura lift'' of weight $2w-1$ as described in \cite{Shimura-half-integral}. This is worked out explicitly for theta functions of lattices with $L=\QQ$ in \cite[\S 4]{Hanke1}, but to reproduce such bounds for theta functions of lattice cosets over totally real fields would take us too far afield for this paper, and is not necessary for our main application.
\end{rmk}

\begin{rmk}
	The condition that $\mu$ be coprime to $\disc\Lambda$ is slightly stronger than necessary, but for $n=4$ it cannot be removed completely. In general, the size of the Fourier coefficient $c_\mu(\theta_{\gerg(\Lambda)})$ depends not on $N_{L/\QQ}(\mu)$, but rather on  $N_{L/\QQ}(\mu_{\rm Iso})$, where $\mu_{\rm Iso}$ is the greatest divisor of $\mu$ that is coprime to all primes $v$ such that $Q$ is anisotropic on $\Lambda_v$. For $n\geq 5$ no such primes exist and so $\mu=\mu_{\rm Iso}$ always; in general, such primes always divide $\disc\Lambda$, and so taking $\mu$ coprime to $\disc\Lambda$ ensures $\mu=\mu_{\rm Iso}$. 
	
	However, if $n=4$, $Q$ is anisotropic on $\Lambda_v$, and $\mu$ is divisible by a very large power of the prime at $v$, then $c_\mu(\theta_{\gerg(\Lambda)})$ can be extremely small even if $N_{L/\QQ}(\mu)$ is large; thus it is possible to have $\mu$ of arbitrarily large norm that is everywhere locally representable by $M$ and yet $c_\mu(\theta_M)=0$. An example of this behavior is given by Watson, who showed that for the lattice $\Lambda=\ZZ^4$ with quadratic form $w^2+x^2+7y^2+7z^2$, and for $\mu:=3\cdot 7^{2u}$ with any $u\geq 0$, the Fourier coefficient $c_{\mu}(\theta_\Lambda)$ is zero (there is no element of $\Lambda$ with norm $\mu$) even though $\mu$ is everywhere locally represented \cite[Section 7.7]{Watson}. In this example $\mu$ can be arbitrarily large, but $\mu_{\rm Iso}=3$ since the lattice is anisotropic at $7$.
\end{rmk}

\subsection{Generating lattices over totally real fields by elements of specified norm} \label{sec:inhomogenous theta functions}

\begin{thm}\label{thm: totally real local-global generators}
	Let $\Lambda$ be an integral positive-definite $\ol$-lattice of even rank $n\geq 4$. Let $S\subset\ol$ be a set of totally positive elements coprime to $\disc\Lambda$ such that $\{N_{L/\QQ}(s): s\in S\}$ is infinite. Suppose that for all $s\in S$ and all finite places $v$ of $L$, $\Lambda_v$ is generated as an $\calO_{L,v}$-module by elements of norm $s$. Then~$\Lambda$ is generated as an $\ol$-module by elements with norm in $S$.
\end{thm}

\begin{rmk}
	The condition that $n$ be even can likely be removed; see Remark \ref{rmk: why n even}.
\end{rmk}

\begin{proof}
	Let $s_1,s_2,\ldots$ be a sequence of elements of $S$ with $N_{L/\QQ}(s_i)\to \infty$. Let $\gerp$ be a prime of $\ol$ with corresponding place $v$. For each $i=1,2,\ldots$, let $C_i\subset \Lambda_v$ be a generating set of $\Lambda_v$ consisting of elements $x$ with $Q(x) =s_i$. Since there are only finitely many subsets of $\Lambda_v/\gerp\Lambda_v$ but infinitely many $C_i$, there must exist an infinite sequence $i_1,i_2,\ldots$ such that $C_{i_j}$ have the same image $\overline{C}\subset\Lambda_v/\gerp\Lambda_v\simeq \Lambda/\gerp\Lambda$ for all $j$.
	
	Fix some element $\overline{t}\in\overline{C}$, and let $t\in\Lambda$ be a preimage of $\overline{t}$ under reduction mod $\gerp\Lambda$. Then $M= t + \gerp \Lambda$ is a lattice coset with modulus $\gerp\Lambda$, and $M_v$ is the preimage of $\overline{t}$ under the map $\Lambda_v\to\Lambda_v/\gerp\Lambda_v$. For all $j$, there exists an element of  $C_{i_j}$ in $M_v$; it has norm $s_{i_j}$. Further, for all finite places $w\neq v$, we have $M_w=\Lambda_w$, and $\Lambda_w$ has an element of norm $s_{i_j}$ by assumption. Thus by Theorem~\ref{thm: coset local global}, there exists an element of $M$ with norm $s_{i_j}$ for some sufficiently large $j$.   
	Repeating this for each $\overline{t}\in\overline{C}$, we find that there exists a set $D\subset \Lambda$ consisting of elements with norm in~$S$ such that $D$ maps surjectively onto $\overline{C}$ under reduction mod $\gerp\Lambda$. Since $\overline{C}$ generates $\Lambda/\gerp\Lambda$, the lattice generated by $D$ maps surjectively onto $\Lambda/\gerp\Lambda$.
	
	We have shown that for all primes $\gerp$, the set of elements in $\Lambda$ with norm in $S$ generate a lattice that surjects onto $\Lambda/\gerp\Lambda$. This implies that the set of elements in $\Lambda$ with norm in $S$ generate~$\Lambda$.
\end{proof}

\id As a consequence we obtain the desired theorem for endomorphism rings of ssAVs with RM.

\begin{thm}\label{thm: End generation totally real}
	Let $\uA$ be a ssAV with RM by $\ol$ defined over a finite field $k$ with characteristic $p$. Let $\ell\neq p$ be a rational prime that is inert in $L$. Then $\End(\uA)$ is generated as an $\ol$-module by elements with norm in $\{\ell^k:k\geq 0\}$.
\end{thm}
\begin{proof}
	As discussed in \S\ref{subsec: structure of RM isog graph}, $\uA$ is isogenous to $\underline{E\otimes\calO}_L$ for some supersingular elliptic curve $E/k$. As explained in \S\ref{sec: ssg for tot real}, this implies that the rank $4$ $\ol$-lattices $\End(\underline{E\otimes\calO}_L)$ and $\End(\uA)$ are in the same genus, so for all finite places $v$ of $L$ lying above a finite prime $q$, $\End(\uA)_v\simeq \calO_q\otimes_{\ZZ_q}\calO_{L,v}$. In \cite[Section 2.3]{GL} we proved that for all $s\in S$, $\calO_q$ has a $\ZZ$-basis consisting of elements of norm $s$. These same elements (tensored with $1\in \calO_{L,v}$) form a $\calO_{L,v}$-basis of $\End(\uA)_v$ consisting of elements with norm $s\otimes 1^2=s$. As every element of $S$ is coprime to $\disc \End(\uA)=p\ol$, we may apply Theorem~\ref{thm: totally real local-global generators} to $\Lambda=\End(\uA)$.
\end{proof}

\

\

\

\end{document}